\documentclass{article}
\usepackage{import}
\usepackage{microtype}
\usepackage{graphicx, caption}
\usepackage{booktabs} 
\usepackage[justification=centering]{caption}

\usepackage[hidelinks]{hyperref}

\usepackage{arxiv}

\usepackage{amsmath}
\usepackage{amssymb}
\usepackage{mathtools}
\usepackage{amsthm}
\usepackage{float}

\usepackage[capitalize,noabbrev]{cleveref}

\theoremstyle{plain}
\newtheorem{theorem}{Theorem}[section]

\newtheorem{lemma}[theorem]{Lemma}

\theoremstyle{definition}

\theoremstyle{remark}

\usepackage[textsize=tiny]{todonotes}

\usepackage{verbatim}\usepackage{hyperref}
\usepackage{color}

\newcommand{\la}{\langle}
\newcommand{\ra}{\rangle}

\newcommand{\lam}{\lambda}

\newcommand{\R}{\mathbb{R}}

\newcommand{\T}{\mathcal{T}}

\newcommand{\HH}{\mathcal{H}}
\newcommand{\U}{\mathcal{U}}

\newcommand{\tr}{\mathrm{tr}}

\newcommand{\SOC}{{\rm SOC}}

\usepackage{enumerate}
\usepackage{mdframed}
\mdfdefinestyle{MyFrame}{%
    linecolor=black,
    outerlinewidth=2pt,
    innertopmargin=4pt,
    innerbottommargin=4pt,
    innerrightmargin=4pt,
    innerleftmargin=4pt,
        leftmargin = 4pt,
        rightmargin = 4pt
        }

\DeclareMathOperator{\expl}{{\bf exp}}
\DeclareMathOperator{\lnl}{{\bf ln}}

\def\bpm{\begin{pmatrix}}
\def\epm{\end{pmatrix}}
\def\bal{\begin{aligned}}
\def\eal{\end{aligned}}

\newcommand{\cc}{\mathcal{K}}
\newcommand{\cj}{\mathcal{J}}

\newcommand{\sfT}{{\sf T}}

\newcommand{\bi}{\begin{itemize}}
\newcommand{\ei}{\end{itemize}}
\newcommand{\beq}{\begin{equation}}
\newcommand{\eeq}{\end{equation}}
\newcommand{\itr}{{\rm int}}
\newcommand{\ben}{\begin{enumerate}}
\newcommand{\een}{\end{enumerate}}

\usepackage{tikz}
\usetikzlibrary{backgrounds}
\usepackage{tikz}

\newcommand*\innerprod[2]{\left\langle #1, #2 \right\rangle}

\DeclareMathOperator{\inte}{int}

\usepackage{bbm}

\newcommand{\email}[1]{\href{mailto:#1}{\texttt{#1}}}
\renewcommand{\headeright}{}
\renewcommand{\undertitle}{}

\DeclareMathOperator*{\argmin}{arg\,min}

\newcommand*\NO{\R_{+}}
\newcommand*\psdc{\mathrm{S}_{+}}
\newcommand*\ball{\mathbb{B}}

\newcommand*\bd{\ball^{d}}
\newcommand*\hb{\frac{1}{2} \bd}
\newcommand*\sumt{\sum_{t=1}^T}
\newcommand*\ta{\frac{1}{T} \sumt}

\DeclarePairedDelimiter\abs{\lvert}{\rvert}%
\DeclarePairedDelimiter{\norm}{\lVert}{\rVert}
\DeclareMathOperator{\nexp}{{\bf norm-exp}}

\newcommand*\scnexp[1]{\text{\sc{e}-}{#1}}

\usepackage{subcaption}

\usepackage{nicematrix}

\title{Multiplicative updates for online convex optimization over symmetric cones}
\author{%
  Ilayda Canyakmaz\\  SUTD\\
  \email{ilayda\_canyakmaz@sutd.edu.sg}\\
  \And
  Wayne Lin\\  SUTD\\
  \email{wayne\_lin@mymail.sutd.edu.sg}\\
  \AND
  Georgios Piliouras\\  SUTD\\
  \email{georgios@sutd.edu.sg}\\
  \And
  Antonios Varvitsiotis\\
  SUTD\\
  \email{antonios@sutd.edu.sg}
  }
\date{}

\begin{document}

\maketitle

\begin{abstract}
We study online convex optimization where the possible actions are trace-one elements in a symmetric  cone, generalizing the extensively-studied experts setup and its quantum counterpart.  Symmetric cones provide a  unifying framework for some of the most important optimization models, including linear, second-order cone, and semidefinite optimization. 
Using tools from the field of  Euclidean Jordan Algebras, we introduce the Symmetric-Cone Multiplicative Weights Update~(SCMWU), a projection-free algorithm for online optimization over the trace-one slice of an arbitrary symmetric cone.  
We show that SCMWU is equivalent to Follow-the-Regularized-Leader and Online Mirror Descent with symmetric-cone negative entropy as regularizer. Using  this structural result  we show that SCMWU is a no-regret algorithm, and verify our theoretical results with extensive experiments. Our results unify and generalize the analysis for the Multiplicative Weights Update method over the probability simplex and the Matrix Multiplicative Weights Update method over the set of density matrices.  
\end{abstract}

\section{Introduction}
 A key challenge in modern machine learning  applications    is the ability to make decisions   in real time and with   incomplete knowledge of the future. A  popular  approach to addressing this problem is   the framework of online  optimization,
  where at each epoch decisions have to be  made  before the corresponding costs are revealed. 
  Online  optimization has found  applications in numerous areas including  finance, control, recommendation systems, and game theory, e.g., see  \cite{BOO:H19,bubeck,shwartz}.
  
 Of particular importance is the setting of online convex optimization (OCO), where both the loss functions and the decision set are convex \cite{INP:Z03}. 
 In the OCO framework, at each time step $t=1,2,\cdots,T$  an   algorithm chooses a point $p_t$ from a    convex set~$\U$ of possible decisions. After committing to this choice, an adversary  reveals a convex cost function $f_t:\U\to \R$ and the algorithm suffers the loss $f_t(p_t)$. A natural benchmark for the performance of an online algorithm 
  is the notion of \emph{regret}, defined as the difference between the expected cumulative loss of the algorithm compared to the loss incurred by the best fixed decision in hindsight. An online algorithm is called no-regret if the time-averaged regret goes to zero,~that~is, as $T\to +\infty$
$$\frac{1}{T}\left(\sum_{t=1}^Tf_t(p_t)-\min_{u\in \U}\sum_{t=1}^Tf_t(u)\right)\to0.$$

 A special case of the OCO framework that has been extensively studied is online linear optimization over the simplex, also known as the experts setting.
 An  algorithm for the experts setting has to choose at  each time step   among $n$ possible actions (so the setting is sometimes referred to as the $n$-experts setting), and  only  after that gets to observe the actions' losses $m_t(1),\ldots, m_t(n)$. 
 One of the most extensively studied no-regret algorithms for the experts setting is the   \emph{Multiplicative Weights Update}~(MWU) that has been rediscovered in diverse fields such as machine learning, game theory and optimization \cite{freund,vovk,lwar,ART:AHK12}. At each time step $t$,   MWU maintains a probability distribution $p_t$ over the $n$ available actions (i.e., an $n$-dimensional probability simplex vector) and incurs expected loss  $\la p_t,m_t\ra=\sum_\ell p_t(\ell)m_t(\ell)$, i.e., a linear function of $p_t$. The probability $p_t(\ell)$ that the $\ell$-th action is picked  is adjusted after the actions' losses are revealed according to the multiplicative rule
 $$p_{t+1}(\ell)=p_t(\ell) \frac{ \exp(-\eta m_t(\ell))} {\sum_\ell p_t(\ell) \exp(-\eta m_t(\ell))} \quad \forall t \geq 0,$$
 or equivalently
 \begin{equation}\label{MWU}\tag{MWU}
 p_{t+1}=\frac{ \expl(-\eta \sum_{i=1}^t m_i)} {\tr(\expl(-\eta \sum_{i=1}^t m_i))} \quad \forall t \geq 0,
 \end{equation}
where $\expl$ is the exponential function applied componentwise and $\tr$ denotes the sum of the components  of a vector. 
Another important instance of the OCO framework   is  online  linear optimization over the set of density matrices,~i.e., Hermitian positive semidefinite (PSD) matrices with trace equal to one, or equivalently online quadratic optimization over unit vectors.  An online algorithm for this setup maintains at each time step $t$  a density matrix  $p_t$ and incurs  linear loss $\la p_t,m_t\ra=\tr(p_tm_t)$, where $\tr$ denotes  the standard matrix trace.   This setting   can be thought as  a  non-commutative extension of the experts setup as density matrices correspond to  quantum probability distributions and   diagonal density matrices are classical probability distributions. 
A natural non-commutative 
analogue of  MWU for no-regret learning  in the quantum version of the experts setup is the {\em Matrix Multiplicative Weights Update}~(MMWU) \cite{ART:TRW05,INP:WK06,ART:AK16, kale} given by:
 \begin{equation}\label{MMWU}\tag{MMWU}
 p_{t+1}=\frac{ \expl(-\eta \sum_{i=1}^t m_i)} {\tr(\expl(-\eta \sum_{i=1}^t m_i))} \quad \forall t \geq 0,
 \end{equation}
where here  $\expl$ denotes the matrix exponential. Beyond online optimization,  MWU and MMWU have found numerous applications that include learning in games and solving linear and semidefinite programs e.g., see  \cite{ART:AHK12}. MMWU has found additional applications in quantum complexity theory \cite{jain2011qip} and graph sparsification  \cite{spectral}.

Both \ref{MWU} and \ref{MMWU} can also be recovered as   instances of two  well-known algorithmic frameworks, namely  \emph{Follow-the-Regularized-Leader} (FTRL) \cite{shwartz} and {\em Online Mirror Descent} (OMD), which is based on the Mirror Descent framework \cite{nem}. 
 The FTRL algorithm selects at each time step $t$ 
 the point that minimizes the cumulative loss so far  plus a strictly convex regularization term $\Phi(\cdot)$, i.e.,
$$p_{t+1}=\argmin_{p\in \U} \Big\{\sum_{i=1}^{t}f_i(p)+\Phi(p)
\Big\}
\quad \forall t \geq 0.$$
 On the other hand, the OMD algorithm tries to minimize the most recent loss function without moving ``too far'' from the previous iterate $p_t$, i.e.,  
\begin{align*}
   p_1&=\underset{p\in \U }{\argmin} \ \Phi(p)\\
 p_{t+1}&=\underset{p\in \U }{\argmin} \Big\{ \eta f_t(p) + \HH_\Phi(p,p_t) \Big\}  \quad \forall t \geq 1  \end{align*} 
where $\HH_\Phi(\cdot,\cdot)$ is the Bregman divergence with respect to the regularizer $\Phi$.
It is well known that under certain conditions the FTRL and OMD viewpoints are equivalent, e.g., see  \cite{INP:HK08}. Finally, when 
the set $\U$ is the set of simplex vectors and $\Phi(x)=\sum_\ell x_\ell\log x_\ell$ is the negative entropy, it follows that  $\HH_\Phi(x,y)$ is the relative entropy (or KL-divergence) and  FTRL and OMD coincide with MWU, e.g. see \cite{shwartz}. Analogously, when $\U$ is the set of density matrices and $\Phi(x)=\tr(x\log x)$ is the negative von Neumann entropy, then $\HH_\Phi(x,y)=\tr(x\log x)-\tr(x\log y)$ is the  quantum relative entropy and FTRL and OMD coincide with MMWU, e.g., see~\cite{spectral}.

{\bf This work: Online convex optimization over symmetric cones.} In this paper we introduce a framework that unifies  and generalizes online convex optimization over the simplex and density matrices.  
    In the  experts setup (both the classical and the quantum version) the state of \ref{MWU} and \ref{MMWU} is an element  in the intersection of a proper cone with a hyperplane. Specifically, in the classical experts case, the simplex can be expressed as the  intersection of the nonnegative orthant with the hyperplane $\{x: \sum_l x_l =1\}$. Similarly, in the quantum version of the experts setup, the set of density matrices is the intersection of the cone of Hermitian PSD matrices with the hyperplane $\{X:\tr(X)=1$\}. Both the nonnegative orthant and the PSD cone are  instances of symmetric cones, i.e.,  self-dual and homogeneous convex cones (see Section \ref{sec:jasc}). Symmetric cones are fundamental in mathematical optimization as they provide a common framework for studying linear optimization over the nonnegative orthant (linear programs), over the second-order cone (second order cone programs), and over the cone of positive semidefinite matrices (semidefinite programs). 
    
    Every symmetric cone is equipped with a trace, that in the case of the orthant is the sum of the entries and in the PSD case is the usual matrix trace. A  trace-one element in a symmetric cone is a distribution over primitive idempotents (see Section \ref{sec:jasc}): in the   simplex case this  is a distribution over the standard basis vectors whereas for  density matrices this is a   distribution over rank-one projectors that sum up to identity.   As trace-one symmetric cone elements generalize both simplex vectors and density matrices, it is natural to ask whether the multiplicative weights framework can be extended for online optimization over  symmetric cones. 
    
    In this work we answer this question in the affirmative by  
 introducing the framework of {\em Online Symmetric-Cone Optimization} (OSCO). 
Specifically, we  consider the setting where at each time step   a decision maker  chooses a trace-one element $p_t$  over  a symmetric cone  $\cc$
 and incurs linear~loss:
\vspace{2mm}
\begin{mdframed}[style=MyFrame]
\begin{tabular}{ll} 
 \label{alg:the_alg}
{\bf Online Symmetric-Cone Optimization (OSCO)} \vspace{2mm}\\
Choose $p_1\in \cc$ with $\tr(p_1)=1$\\
{For } $t=1,\ldots ,T$:\\
\quad Receive linear loss $\la m_t,p_t\ra$, get as feedback the loss vector $m_t$ \\
\quad Compute  new iterate  $p_{t+1}\in \cc$ with $\tr(p_{t+1})=1$
 \end{tabular}
 \end{mdframed}
 
 As explained previously, the OSCO  framework  provides a common generalization of the experts setup and its quantum counterpart. 
To study OSCO we  introduce the {\em Symmetric-Cone Multiplicative Weights Update}~(SCMWU), which unifies and generalizes \ref{MWU} and \ref{MMWU}. Similarly to MWU and MMWU, to define SCMWU we  use that  the interior of a symmetric cone is parametrized by an  exponential map defined on the underlying Euclidean Jordan Algebra. 
Using an appropriate notion of entropy for symmetric cones (and a corresponding notion of Bregman divergence) we show that SCMWU can be also expressed within the FTRL and OMD algorithmic frameworks. 
Finally, using this equivalence, we show that SCMWU is a no-regret algorithm.   

{\bf Paper organization.}  
In Section \ref{sec:jasc}, we review Euclidean Jordan Algebras and symmetric cones, followed by definitions of entropy and Bregman divergence for symmetric cones. In Section \ref{sec:OSCO}, we describe the OSCO framework and introduce the SCMWU algorithm.  In Section \ref{sec:equivalence}, we prove the equivalence of SCMWU, FTRL and OMD, and in Section \ref{sec:regret} we show that SCMWU has vanishing regret. In Section \ref{sec:_SCMWU-ball}, we study the specific application of SCMWU to online learning over the ball, and finally in Section \ref{sec:experiment2} we 
numerically evaluate the performance of~SCMWU.

\section{Preliminaries}\label{sec:jasc}
\paragraph{Euclidean Jordan Algebras.}
In this section we give a brief introduction  to Euclidean Jordan Algebras (EJAs),  symmetric cones, and their connections. 
For additional details   the reader is referred to  \cite{BOO:FK94} and \cite{PHD:V07}. Additional details on symmetric cones  can be found in \cite{vanden}.

Let $\cj$ be a finite-dimensional vector space with a bilinear product $\circ:\cj\times\cj\to\cj$. The tuple  $(\cj,\circ)$ is called a Jordan algebra if for all $x,y\in\cj$,
$$x\circ y=y\circ x,$$
$$x^2\circ(x\circ y)=x\circ (x^2\circ y),$$
where $x^2=x\circ x$. If there exists $e\in\cj$ such that
$e\circ x=x\circ e=x$ for all $x\in\cj$
then $e$ is called an identity element, which is necessarily unique. 

 A Jordan Algebra $(\cj,\circ)$ over $\R$ is called Euclidean if there exists an  inner product
$\langle\cdot,\cdot\rangle$ that is associative, i.e.,   for all $x,y,z\in\cj$
$$\langle x\circ y,z\rangle=\langle y,x\circ z\rangle.$$

An element $q\in\cj$ is an idempotent if $q^2=q$. An idempotent is said to be primitive if it is nonzero and cannot be written as the sum of two nonzero idempotents. A collection of primitive idempotents $q_1,\cdots,q_m$ form a Jordan frame if
$$q_i\circ q_j=0 \ \forall i\neq j\  \text{ and }\  \sum_{i=1}^mq_i=e.$$

Consider a rank-$r$  EJA $(\cj,\circ)$. For every $x\in\cj$ there exist real numbers $\lambda_1,\cdots,\lambda_r$ and a Jordan frame $q_1,\cdots,q_r$ such that $x=\sum_{i=1}^r\lambda_iq_i,$
known as the type-II spectral decomposition of $x$, e.g., see Theorem III.1.2 in \cite{BOO:FK94}.
The scalars  $\lambda_i$  are  called the eigenvalues of $x$, and  are uniquely specified up to reordering and accounting for multiplicities. The  L\"{o}wner extension 
of a   scalar-valued function $f:\mathcal{D}\subseteq \R \to\R$, denoted by $\mathbf{f}$,  maps elements of the EJA with spectrum in~$\mathcal{D}$ to elements of the EJA as follows:

$$\mathbf{f}:\quad \sum_{i=1}^r\lambda_i q_i \mapsto  \sum_if(\lambda_i)q_i\in~\cj.$$

Moreover, the {trace} of $x$ is given by
${\rm tr}(x)=\sum_{i=1}^r\lambda_i.$
Note that ${\rm tr}(x\circ y)$ is a symmetric bilinear mapping that is positive definite (see e.g., Proposition III.1.5 in \cite{BOO:FK94}) and associative (see e.g., Proposition II.4.3 in \cite{BOO:FK94}), i.e., it is an associative inner product. In the following, the inner product of the EJA is fixed as $\langle x, y\rangle={\rm tr}(x\circ y)$, and $\|x\|=\sqrt{\la x,x \ra}$ is the induced~norm.

\paragraph{Symmetric cones.}
\noindent Given an EJA $(\cj,\circ)$, its cone of squares is defined as the set
$$\cc(\cj):=\{x^2:x\in\cj\},$$
where to simplify notation we  typically   omit dependency on the EJA $\cj$.
 The cone of squares  $\cc$  is a proper cone  (it is self-dual with respect to the inner product $\la x,y\ra={\rm tr}(x\circ y)$) and so it induces the partial ordering on $\cj$ 
\[
    x \preceq_{\cc} y \Longleftrightarrow x - y \in \cc,
\]
see, e.g., Section 2.4 in \cite{BOO:B04}. 
 Equivalently,  $\cc \  ({\rm int}(\cc))$ is the set of elements of $x=\sum_i\lambda_iq_i\in \cj$ whose eigenvalues are all nonnegative (strictly positive) (see, e.g., Proposition~2.5.10 in  \cite{PHD:V07}), i.e., 
 \begin{align*} 
 &x \in \cc \iff   \lambda_i\ge 0 \ \forall i \iff x\succeq_\cc 0\\
  & x \in {\rm int}(\cc) \iff  \lambda_i >0\ \forall i \iff x\succ_\cc 0.
 \end{align*}

Moreover, the cone of squares of an EJA $(\cj, \circ)$ is a symmetric cone, i.e., 
it is a convex cone satisfying the following two conditions:
\begin{itemize}
    \item [$(i)$] $\cc$ is self-dual, i.e., 
    $$\cc^*=\{y\in \cj: \langle y,x\rangle\ge 0 \ \forall x\in \cc\} = \cc,$$
    \item [$(ii)$] $\cc$ is homogeneous, i.e., for any  $u,v\in{\rm int}(\cc)$ 
    there exists an invertible linear transformation $\T:\cj \to \cj$ such that $\T(u)=v$ and $\T(\cc)=\cc,$
\end{itemize}
e.g., see Proposition 2.5.8 in \cite{PHD:V07}.  The converse statement is also true, i.e.,     a  cone is symmetric if and only if it is the cone of squares of some EJA, e.g., see Theorem III.3.1 in  \cite{BOO:FK94}.  This fact combined  with the classification of all finite-dimensional EJAs (e.g., see \cite{BOO:FK94}) implies  that any symmetric cone is isomorphic to a direct sum of the following primitive~cones:
\begin{itemize}
   \item Positive semidefinite matrices   over the reals/complex numbers/quaternions, 
    \item Second-order (aka Lorentz) cones, and 
   \item  3-by-3 PSD matrices over the field of octonions.
\end{itemize}

In view of this classification result, symmetric cones provide a common language for studying  linear conic optimization problems over the nonnegative orthant (linear programs), over the second-order cone (second order cone programs), and over the cone of positive semidefinite matrices (semidefinite programs). Linear, second-order cone, and semidefinite programs are    some of the most important mathematical optimization models with extensive modeling power and efficient algorithms for computing high quality solutions.

\paragraph{Entropy and Bregman divergence for symmetric cones.} 

The L\"{o}wner extensions  corresponding to the scalar exponential  $\exp(\cdot)$ and $\ln(\cdot)$ functions  are defined as
\begin{align*}
\expl: \ &\cj\to {\rm int}(\cc), \quad x=\sum_{i=1}^r\lambda_iq_i \mapsto\sum_{i=1}^r \exp(\lambda_i) q_i\\
\lnl : \ & {\rm int}(\cc)\to \cj, \quad x=\sum_{i=1}^r\lambda_iq_i \mapsto\sum_{i=1}^r \ln(\lambda_i) q_i,
\end{align*}
and are inverses to each other, see Lemma \ref{expvslog}. 
For any symmetric cone $\cc$ there exists a corresponding notion of entropy whose domain is the interior of the cone. Specifically,  the symmetric-cone negative entropy (SCNE) of  an element $x=\sum_i\lambda_iq_i\in {\rm int} (\cc)$ is defined by
\begin{equation}\label{entropy}\tag{SCNE}
\Phi(x)={\rm tr}(x\circ \lnl x)=\sum_i\lambda_i\ln(\lambda_i),
\end{equation}
see, e.g., \cite{ART:CP10}. Note that the \ref{entropy} is well defined as $x=\sum_i\lambda_iq_i\in {\rm int}(\cc) $ iff all eigenvalues $\lambda_i$ are  positive.   

When $\cc=\R^d_+$ is the nonnegative orthant, $\Phi(x)$ is the usual negative entropy of a vector $x\in {\rm int}(\R^d_{+})$ with strictly positive entries, e.g. see~\cite{BOO:B04}. Furthermore, when  $\cc=\mathcal{S}^d_+$  is the positive semidefinite cone (with Hermitian entries), 
$\Phi(x)$ is the negative von Neumann entropy of the unnormalized quantum state $x\in {\rm int}(\mathcal{S}^d_{+}$), e.g. see~\cite{carlen}. 
 In the  next result we summarize  some important  properties of the entropy. 
 
\begin{lemma}\label{lem:entropy} The \ref{entropy} $\Phi(x )={\rm tr}(x\circ \lnl x)$ corresponding to a symmetric cone $\cc$ is a mirror map, i.e., it~satisfies:
\begin{enumerate}[(i)]
    \item\label{lem:entropy1} $\Phi(x)$ is continuously differentiable on $x\in{\rm int}(\cc)$ and 
    $\nabla  \Phi(x)=\lnl x+e.$
    \item $\Phi(x)$ is strictly convex on $\cc.$\label{lem:entropy2}
    \item 
    $\displaystyle \lim_{x\to \partial \cc} \|\nabla  \Phi(x)\|=+\infty$.\label{lem:entropy3}
    \item $\nabla \Phi(x) =\lnl x+e$ is a bijective map 
    from $\inte(\cc)$ to~$\cj$.  Moreover,  its inverse map is given by 
    \begin{equation}\label{lem:entropy4} 
    (\nabla \Phi)^{-1}(x)=\expl(x-e).
    \end{equation} \label{lem:entropy4_actualLemma}
\end{enumerate}
\end{lemma}
{\em Proof.} Parts (\ref{lem:entropy1}) and (\ref{lem:entropy2}) are proven in~\cite{ART:CP10}. For part (\ref{lem:entropy3}), given  $x=\sum_i\lambda_iq_i$ we have that  
$$\|\nabla  \Phi(x)\|^2=\|\lnl x+e\|^2=\tr(\lnl x+e)^2={\sum_i(\ln (\lam_i)+1)^2,}$$
which goes to $+\infty$ as $x\to\partial\cc.$ 

(\ref{lem:entropy4_actualLemma}) Given $y=\sum_i\lambda_iq_i \in \inte(\cc)$ with $\lambda_i > 0 \ \forall i$, we have from (\ref{lem:entropy1}) that $\nabla \Phi(y)=\lnl y+e$, and so $\nabla \Phi : \inte (\cc) \rightarrow~\cj$ is bijective since $\lnl : \inte (\cc) \rightarrow \cj$ is bijective by Lemma \ref{expvslog}. Finally,
\begin{align*}
 \nabla \Phi(\expl(x-e)) & = e+\ln(\expl(x-e))\\
& =e+\ln(\expl(x)\exp(-1)) \\
& =e+\sum_i\ln(\exp(\lambda_i-1))q_i\\
 &=e+\sum_i(\lambda_i-1)q_i=e+x-e=x,
\end{align*}
where in the second equality we used Lemma \ref{lem:expe}, and so the inverse map of $\nabla \Phi$ is
\[
    \nabla \Phi : \cj \to \inte(\cc), \quad x \mapsto \expl(x-e).\qed
\]

 The Bregman divergence corresponding  to the \ref{entropy} $\Phi(x) = \tr (x \circ \lnl x)$ is given by \begin{equation}\label{eqn:Breg}
   \HH_\Phi(x,y) = \Phi(x) - \Phi(y) - \la \nabla \Phi(y),  x - y\ra
 \end{equation}

 for $ x,y \in {\rm int}(\cc)$.  
 The Bregman divergence is the difference at the point $x$ between $\Phi$ and the first-order Taylor expansion of $\Phi$ at~$y$.

In the case  where $\cc=\R^d_+$, $\HH_\Phi(x,y)$ is the (unnormalized) Kullback-Leibler divergence of $x,y$, e.g. see \cite{BOO:B04}. Moreover, when  $\cc=\mathcal{S}^d_+$  is the positive semidefinite cone (with Hermitian entries),  $\HH_\Phi(x,y)$ is the unnormalized quantum relative entropy of $x$ with respect  to $y$, e.g., see \cite{carlen}. 

\begin{lemma}\label{bregproof}
The Bregman divergence corresponding to the \ref{entropy}  $\Phi(x)  = {\rm tr}(x\circ\lnl x)$ is given by:

\begin{equation*}
\HH_\Phi(x,y) =   {\rm tr}(x \circ \lnl x-x \circ \lnl y+y-x).
\end{equation*}

\end{lemma}
\begin{proof} From the definition of the Bregman divergence  and the formula for $\nabla\Phi(x)$ given in Lemma \ref{lem:entropy}, we get
\begin{align*}
  \HH_\Phi(x,y)&=\Phi(x)-\Phi(y)-\la \nabla \Phi(y), x-y\ra\\
  &={\rm tr}(x\circ \lnl x)-{\rm tr}(y\circ \lnl y)-\la \lnl y+e,x-y\ra\\
  &={\rm tr}(x\circ \lnl x)-{\rm tr}(y\circ \lnl y)-{\rm tr}( \lnl y\circ x)-{\rm tr}( \lnl y\circ y)-{\rm tr}(x)-{\rm tr}(y)\\
  &={\rm tr}(x\circ \lnl x- x\circ\lnl y+y-x).
\end{align*}

\end{proof}

The following lemma  summarizes  key properties of the Bregman divergence that are relevant to the scope of this work.

\begin{lemma} \label{breglemma} The Bregman divergence satisfies the following properties:
\begin{enumerate}[(i)]
    \item \label{lem:BregProp1} $\HH_\Phi(x,y)$ is continuous 
on ${\rm int}(\cc)\times{\rm int}(\cc)$ and $\HH_\Phi(\cdot,y)$ is strictly convex for any $y\in{\rm int}(\cc).$ 
    \item \label{lem:BregProp2} For any fixed $y \in {\rm int}(\cc)$, $\HH_\Phi(\cdot,y)$ is continuously differentiable on ${\rm int}(\cc)$ and ${\nabla_x \HH_\Phi(x,y) = \lnl x - \lnl y.}$
    \item\label{lem:BregProp3}  $\HH_\Phi(x,y)\geq 0$ for any $x, y\in{\rm int}(\cc)$; furthermore, equality holds if and only if $x=y.$
    \item\label{lem:3term} For $ x,y,z \in{\rm int}(\cc),$

    the ``three-point identity''~holds:
$$\la z-y , \lnl y-\lnl x\ra=\HH_\Phi(z,x)-\HH_\Phi(z,y)-\HH_\Phi(y,x).$$
\item The Bregman projection of  $y\in {\rm int}(\cc)$ 
onto the trace-one slice of $\cc$
is given by
\begin{equation}\label{brgprojection}
    \underset{x\in \cc,\tr(x)=1}{\argmin}\HH_\Phi(x,y)={y\over \tr(y)}.
\end{equation}
\end{enumerate} 
\end{lemma}
\begin{proof}
Properties $(i)-(iv)$ are derived  in \cite{ART:CP10}, and so we only provide a proof for part $(v)$. Since $\Phi$ is a mirror map the minimum is attained at a unique point $x^*$ which is in ${\rm int}(\cc)$, e.g. see \cite[Theorem 3.12]{bauschke}.  To show that the minimizer has the specific form we use the KKT conditions.  The Lagrangian for this problem is given~by
    \[
        \mathcal{L}(x, \lambda,z) = \HH_\Phi(x,y) + \lambda (  \tr(x)-1)-\la z, x\ra,
    \]
    where $z\succeq_\cc 0$ is the dual multiplier corresponding to the constraint $x\succeq_\cc 0$. Setting $\lambda^*, z^*$  to be a dual optimal solution, by the KKT conditions we have that 
    \[
        \lnl x^*-\lnl y +  \lambda^* e-z^*=0 \ \text{ and }\  \la z^*,x^*\ra=0.
    \] 
    As $x^*\in {\rm int}(\cc)$ it follows that $z^*=0$.  Indeed, let $x^*=\sum_i\lambda_iq_i$ be the spectral decomposition of $x^*$. Since $x^*\in {\rm int}(\cc)$ we have that $\lambda_i>0 \ \forall i$, which coupled with the fact that $\innerprod{z^*}{q_i} \geq 0 \ \forall i$ (since $q_i \in \cc$) while $\sum_i \lambda_i \innerprod{z^*}{q_i} = \innerprod{z^*}{\sum_i \lambda_i q_i} = \innerprod{z^*}{x^*} = 0$ implies that $\la z^*,q_i\ra=0$ for all $i$. Thus $\tr(z^*)=\la z^*, e\ra=\sum_i\la z^*, q_i\ra=0$, which together with the fact that $z^* \in \cc$ (and thus all eigenvalues of $z^*$ are nonnegative) shows that all eigenvalues of $z^*$ are zero. We thus conclude that $z^*$ is zero and get that 

       \[ \lnl x^*-\lnl y +  \lambda^* e=0.\]
Solving for $x^*$ then gives:
$$x^*=\expl(\lnl(y)-\lambda^*e)=\exp(-\lambda^*)\expl(\lnl(y))=\exp(-\lambda^*)y,$$
where for the second equality we use Lemma \ref{lem:expe}. 
Finally, $\lambda^*$ is uniquely determined using $\tr(x^*)=1$. Indeed,  
$$1=\tr(x^*)=\exp(-\lambda^*)\tr(y)\iff \lambda^*= \ln( \tr(y)),$$
thus giving
\[
    x^* = \frac{y}{\tr(y)}.
\]
\end{proof}

\paragraph{An explicit example: The second-order cone.} The second-order cone (SOC) is given by
\begin{equation}
\label{SOCspec}
\mathrm{SOC}_d=\{(x,s)\in \R^d\times \R:\  \|x\|_2\le s\},
\end{equation}
and it is the cone of squares of the EJA defined by the real vector space $\R^{d+1}$ equipped with the Jordan product 
$$(x, s)\circ  (x',s')=(sx'+s'x, x^\sfT x'+ss'),$$
with unit  $e=(0_d,1)$, e.g. see \cite{Alizadeh}. The corresponding spectral decomposition is
$$ (x,s)=({s+\|x\|_2})q_{+} +(s-\|x\|_2)q_{-}, $$
where 
$$q_{+}={1\over 2}\Big({x\over \|x\|_2}, 1\Big) \quad  \; q_{-}={1\over 2}\Big(-{x\over \|x\|_2}, 1\Big)$$ is a Jordan frame and  $\tr(x,s)=~2s.$ Finally, the EJA inner product is  twice the Euclidean inner product, i.e.,
\begin{equation}
\label{SOC_innerprod}
\la (x, s),  (x',s')\ra =\tr((x, s)\circ   (x',s')) =2(x^\sfT x'+ss').
\end{equation}

The trace-one slice of  ${\rm SOC}_d$  is isomorphic to the half-unit $d$-ball $\frac{1}{2}\ball^d = \{x\in \R^d: \|x\|_2 \le 1/2\}$.
Over  the interior of the SOC (i.e., vectors  $(x,s)\in {\rm SOC}_d$ with $\|x\|_2<s$) we have by \eqref{entropy}  that the corresponding entropy  is given by
$$
    \Phi(x,s) = (s + \|x\|_2)\ln(s + \|x\|_2) + (s - \|x\|_2)\ln(s - \|x\|_2).
$$
 Figure \ref{fig:entbreg}   shows  the level curves of the \ref{entropy} and the Bregman divergence ${x\mapsto  \mathcal{H}_\Phi(\cdot,y)}$ over ${\rm SOC}_2 \cap (\tr=1)$. 
 \begin{figure}[h]
    \centering
    \begin{minipage}{.49\linewidth}
      \centering
      \includegraphics[width=.95\linewidth]{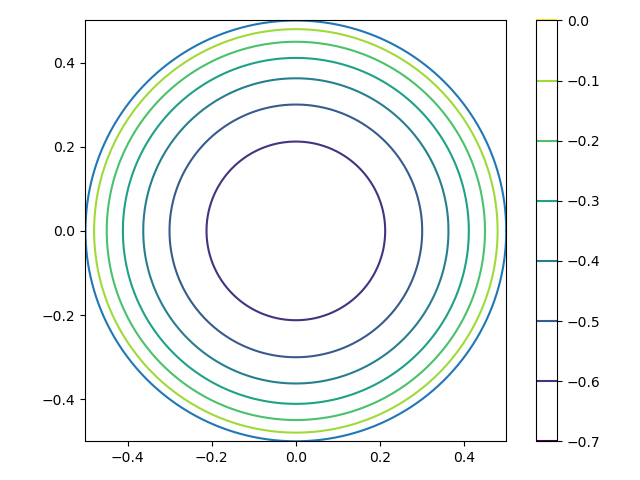}
     \ref{entropy} $\Phi(\cdot)$
    \end{minipage}
    \begin{minipage}{.49\linewidth}
      \centering
      \includegraphics[width=.95\linewidth]{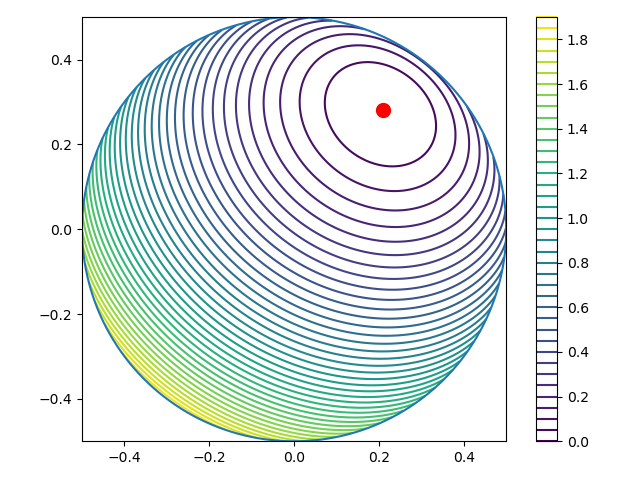}
      Bregman divergence $\HH_{\Phi}(\cdot, y)$
    \label{fig:bregman}
    \end{minipage}
    \caption{Level curves of the \ref{entropy} and the Bregman divergence with respect to  $y = (0.21, 0.28, 0.5)$ over ${\rm SOC}_2 \cap (\tr=1)$.}
    \label{fig:entbreg}
\end{figure}

 \section{Online symmetric-cone optimization}\label{sec:OSCO}
Having defined  the \ref{entropy} $\Phi$ and the corresponding Bregman divergence $\HH_\Phi$ for  symmetric cones, we  now define the following three iterative schemes for online learning over the trace-one slice of any symmetric cone: 

\begin{mdframed}[style=MyFrame]
 {\bf  Symmetric-Cone Multiplicative Weights Update:} 
\begin{equation}
\tag{SCMWU}\label{SCMWU}
p_{t+1}=\frac{ \expl(-\eta \sum_{i=1}^t m_i)} {\tr(\expl(-\eta \sum_{i=1}^t m_i))}
\quad \forall t \geq 0
\end{equation}

 {\bf Follow the Regularized Leader: } 
\begin{equation}
\tag{FTRL}\label{FTRL}
p_{t+1}= \underset{x\in \cc,\tr(x)=1 }{\argmin} \Big\{ \Phi(x)+\eta\sum_{i=1}^t\la m_i, x\ra\Big\}
\quad \forall t \geq 0
\end{equation}
  {\bf Online Mirror Descent:}
 \begin{align*}\tag{OMD}\label{OMD}
   p_1&=\underset{x\in \cc,\tr(x)=1 }{\argmin} \Phi(x)\\
 p_{t+1}&=\underset{x\in \cc,\tr(x)=1 }{\argmin} \Big\{ \eta \la m_t, x\ra+ \HH_\Phi(x,p_t) \Big\}
 \quad \forall t \geq 1 
 \end{align*}
 \end{mdframed}
 
 \medskip 
{\bf Main theorem:} Let $(\cj, \circ )$ be an EJA of rank $r$ and let~$\cc$ be its cone of squares.  We have that:
\begin{itemize}
\item[$(i)$] The  iterates generated by  \ref{SCMWU}, \ref{FTRL} and \ref{OMD} coincide.
\item[$(ii)$]   \ref{SCMWU} has vanishing regret. Specifically,  for any sequence of loss vectors   $m_t$ where   
${-e \preceq_\cc m_t\preceq_\cc e}$ and stepsize $\eta \leq 1$  we have  that $$\sum_{t=1}^{T} \la m_t,p_t-u\ra \le
 \eta T+\frac{1}{\eta}\ln r,$$
for all $ u\in\cc$ such that $\tr(u)=1$.
\end{itemize}

The proof of part $(i)$ is given in Theorem \ref{thm:main},  and that of part~$(ii)$ is given in Theorem \ref{thm:regret}.  The  proof of these three viewpoints' equivalence  is well known  both for the  case of the nonnegative orthant, e.g., see \cite{BOO:H19}, and  the PSD cone, e.g., see \cite{spectral}. Moreover, our  regret bound for the  nonnegative orthant  specializes to the well-known fact that 
the multiplicative weights update method has vanishing regret  \cite{vovk,lwar,ART:AHK12}. Finally,  our regret bound in the case of the~PSD  cone specializes to the fact that the matrix multiplicative weights update method is no-regret \cite{kale}.

\section{Equivalence of the three iterative schemes}\label{sec:equivalence}
In this section, we prove the equivalence of the three different iterative procedures \ref{SCMWU}, \ref{FTRL}, and \ref{OMD}. 
This relies on the well-known  fact that both the FTRL and OMD iterates can be decoupled into two steps: taking an unconstrained $\argmin$ followed by taking a Bregman projection 
(see, e.g., Theorem 18 and 19 in \cite{ART:McMahan} and Section 6.4.2 in \cite{orabona}). In our case, we are concerned with the set $\{x\in\cc,\tr(x)=1\}$. 
In Lemmas \ref{lemma:FTRldecoup} and \ref{lem:OMD} we give the explicit forms of the decoupled iterates in our symmetric cone setting, and use these to prove the equivalence of \ref{SCMWU}, FTRL, and OMD in Theorem \ref{thm:main}.

\begin{lemma}\label{lemma:FTRldecoup}(FTRL decoupling)
The \ref{FTRL} iterates
\[
 p_{t+1}=\underset{x\in \cc,\tr(x)=1 }{\argmin} \Big\{ \Phi(x)+\eta\sum_{i=1}^t\la m_i, x\ra  \Big\}
\]
can be decoupled as 
\begin{align}
y_{t+1}&= \underset{x\in \R^d}{\argmin} \Big\{ \Phi(x)+\eta\sum_{i=1}^t\la m_i, x\ra \Big\}=\exp(-1)\expl(-\eta \sum_{i=1}^tm_i),\label{lem:decoup_ftrl1}\\
p_{t+1}&=\underset{x\in \R^d,\tr(x)=1}{\argmin}\HH_\Phi(x,y_{t+1})
= \frac{ \expl(-\eta \sum_{i=1}^t m_i)} {\tr(\expl(-\eta \sum_{i=1}^t m_i))}
.\label{lem:decoup_ftrl2}
\end{align}

\end{lemma} 

\begin{proof}
    $x^* := (\nabla \Phi)^{-1} \left( - \eta \sum_{i=1}^t m_i\right)$ is the unique minimizer of $\Phi(x) + \eta \sum_{i=1}^t \innerprod{m_i}{x}$ since it satisfies
    \[
        \nabla \left[\Phi(x^*) + \eta \sum_{i=1}^t \innerprod{m_i}{x^*} \right]
        =
        0
    \]
     and $\Phi(x)$ is strictly convex (see Lemma \ref{lem:entropy}). Thus, from \eqref{lem:entropy4} we have that 
$$y_{t+1}=\expl(-\eta \sum_{i=1}^t m_i - e)= \exp(-1)\expl(-\eta \sum_{i=1}^t m_i ),$$
where for the second equality we use Lemma \ref{lem:expe}. Finally, by \eqref{brgprojection} it follows that
\[
    \underset{x\in \R^d,\tr(x)=1}{\argmin}\HH_\Phi(x,y_{t+1})
    =
    \frac{ \expl(-\eta \sum_{i=1}^t m_i)} {\tr(\expl(-\eta \sum_{i=1}^t m_i))}.
\]
\end{proof}

\begin{lemma}\label{lem:OMD}(OMD decoupling)
The \ref{OMD} iterates
$$p_{t+1}=\underset{x\in \cc,\tr(x)=1}{\argmin} \Big\{ \eta \la m_t, x\ra+ \HH_\Phi(x,p_t) \Big\}$$
can be decoupled as
\begin{align}
y_{t+1}&=\underset{y\in \R^d}{\argmin} \Big\{ \eta\la m_t, y\ra+ \HH_\Phi(y,p_t)\Big\}
= \expl(\lnl p_t-\eta m_t),
\label{eqn:MDdecoup1}\\
p_{t+1}&=\underset{x\in \R^d,\tr(x)=1}{\argmin}\HH_\Phi(x,y_{t+1})
= \frac{\expl(\lnl p_t-\eta m_t)}{\tr(\expl(\lnl p_t-\eta m_t))}
.\label{eqn:MDdecoup2}
\end{align}
\end{lemma}

\begin{proof}
    We first note that any $y^*$ satisfying $\nabla_y \HH_\Phi (y^*, p_t) = - \eta m_t$ would be the unique minimizer of $\eta \la m_t, y\ra + \HH_\Phi(y,p_t)$ since it would satisfy
    \[
        \nabla_y \big[\eta \la m_t, y^*\ra + \HH_\Phi(y^*,p_t) \big] = 0
    \]
    and $\HH_{\Phi}(\cdot, \cdot)$ is strictly convex in its first argument (see Lemma \ref{breglemma} \eqref{lem:BregProp1}).
    Thus, from Lemma \ref{breglemma} \eqref{lem:BregProp2} we have that 
$$\eta m_t+\lnl y_{t+1}-\lnl  p_t=0$$
and so 
$$y_{t+1}=\expl(\lnl p_t-\eta m_t).$$
Finally, from \eqref{brgprojection} it follows that
$$\underset{x\in \R^d,\tr(x)=1}{\argmin}\HH_\Phi(x,y_{t+1})
=\frac{\expl(\lnl p_t-\eta m_t)}{\tr(\expl(\lnl p_t-\eta m_t))}.$$
\end{proof}

\begin{theorem}\label{thm:main}
For any  symmetric cone $\cc$ 
the iterates $\{p_t\}_{t}$ generated by \ref{SCMWU}, \ref{FTRL} and \ref{OMD}~coincide.
\end{theorem}

\begin{proof}
    We have already shown in Lemma \ref{lemma:FTRldecoup} that the \ref{SCMWU} and FTRL iterates coincide. To prove that the \ref{SCMWU} and OMD iterates coincide, suppose by induction that they give the same iterate $p_t$ for some  $t \geq 1$. 
     Then by Lemma \ref{lem:OMD} OMD's next iterate is
    \begin{equation} \label{xcsdvf}
        p_{t+1} 
        = 
        \frac{\expl(\lnl p_t-\eta m_t)}{\tr(\expl(\lnl p_t-\eta m_t))}.
    \end{equation} 
    From the induction hypothesis  we have that
\begin{equation}\label{eqn:_viewpoints_p_t} 
p_t=\expl(c_t e-\eta \sum_{i=1}^{t-1} m_i)
\end{equation}
where $c_t=-\ln \tr(-\eta \expl(\sum_{i=1}^{t-1}m_i)).$ Substituting \eqref{eqn:_viewpoints_p_t} in  \eqref{xcsdvf}   gives us that
\[
    p_{t+1}
     =  { \expl(c_t e - \eta \sum_{i=1}^t m_i) \over \tr(\expl(c_t e - \eta \sum_{i=1}^t m_i))}=    \frac{ \expl(-\eta \sum_{i=1}^t m_i)} {\tr(\expl(-\eta \sum_{i=1}^t m_i))},
\]
where for the last equality we use Lemma \ref{lem:expe}. 
\end{proof}

\paragraph{Remark on the per-iteration complexity of \ref{SCMWU}.} We note that each \ref{SCMWU} update amounts to performing a spectral decomposition of the sum of loss vectors (and then exponentiating the eigenvalues). \ref{SCMWU} is thus comparable in per-iteration complexity to Online Gradient Descent (OGD) \cite{INP:Z03}, which is a first-order algorithm for online convex optimization given by the update
\begin{equation}
\label{OGD} \tag{OGD}
    p_{t+1} = \argmin_{p \in \U} \, \norm{p - (p_t - \eta m_t)}_2.
\end{equation} 

The primary computational challenge of \ref{OGD} lies in performing Euclidean projection onto the feasible region. The problem of projecting onto the trace-one slice of a primitive symmetric cone is well studied (see, e.g.,   \cite{duchi2008efficient,chen} for the simplex and \cite{yang2020revisiting}  for density matrices), and the algorithms for this problem have comparable per-iteration complexity to SCMWU.
However, whenever the symmetric cone is not primitive and has a direct-sum structure, 
a convex program needs to be solved in each iteration of \ref{OGD}. On the other hand, \ref{SCMWU}'s per-iteration complexity in this setting increases additively when we add components to the symmetric cone's direct-sum structure. This is due to the fact that performing a spectral decomposition on an EJA element with a given direct-sum structure is equivalent to performing a spectral decomposition on each of the direct-sum components.

\section{\ref{SCMWU} has vanishing regret}\label{sec:regret}

In this section  we present a proof of \ref{SCMWU}'s regret bound which relies  on its equivalence with  \ref{OMD}. An alternative, potential-based proof is given in Appendix \ref{appendix:_altreg}.

\begin{theorem}\label{thm:regret} 
Let $(\cj,\circ)$ be an EJA of rank $r$,  $\cc$ be its cone of squares,  and $\la x,y\ra=\tr(x\circ y)$ be the EJA inner product. For any sequence of loss vectors     
${-e \preceq_\cc m_t\preceq_\cc e},$ the    \ref{SCMWU}  iterates  $\{p_t\}_t$ with stepsize ${\eta\leq 1}$  satisfy
\[
    \sum_{t=1}^{T} \la m_t,p_t-u\ra \le
    \eta \sum_{t=1}^T \innerprod{m_t^2}{p_t} +\frac{1}{\eta}\ln r
\]
for all $ u\in\cc$ such that $\tr(u)=1$. 
\end{theorem}

Before we provide a proof we  make some remarks on Theorem \ref{thm:regret}.   Since $-e \preceq_{\cc} m_t \preceq_{\cc} e$ it follows that $0 \preceq_{\cc} m_t^2 \preceq_{\cc}  e$, and so $\innerprod{ m_t^2 }{ p_t } \leq \innerprod{ e }{ p_t } = 1$. Combining this  with  Theorem \ref{thm:regret} we get that
	\[
		  \sum_{t=1}^{T} \la m_t,p_t-u\ra 
            \leq
		\eta T
		+ \frac{1}{\eta} \ln r.
	\]

Moreover, using the notion of eigenvalues in an EJA and the corresponding variational characterization of $\lambda_{\min}$ (e.g., see Theorem \ref{thm:eig_ineq}) the regret bound becomes:
 $$\sum_{t=1}^{T} \la m_t, p_t\ra -\lam_{\min}\left(\sum_{t=1}^{T}  m_t \right)\le  \eta T+\frac{1}{\eta}\ln r. $$
The stepsize $\eta$ that minimizes the  bound $\eta T+\frac{1}{\eta}\ln r$ is $\sqrt{\ln r/T}$, which gives regret 
\begin{equation}
\label{regbound_opt}
    2\sqrt{T \ln r}.
\end{equation}
Finally,  if there is no a priori knowledge of the time horizon $T$,  a no-regret algorithm can be obtained from \ref{SCMWU} using the well-known ``doubling trick'', e.g., see \cite{shwartz}. Following this approach, time is broken up into epochs of exponentially increasing length $2^i$, where over the $i$-th epoch \ref{SCMWU} is run with optimized stepsize $\eta_i = \sqrt{2^{-i} \ln r}$. At time $T$, the regret of this algorithm is upper bounded by 
\begin{equation}
\label{doub_bound}
\frac{2\sqrt{2}}{\sqrt{2}-1} \sqrt{T \ln r}.
\end{equation}

In order to streamline the presentation of the regret bound proof, we  begin with  a  technical lemma.
\begin{lemma}\label{lem:regret}Let $(\cj,\circ)$ be an EJA and $\cc$ its cone of squares. For any fixed  $\eta>0$ and any sequence of loss vectors      
${-e} \preceq_\cc m_t\preceq_\cc e $, the   decoupled \ref{OMD} iterates   
\begin{align*}
y_{t+1}&=\underset{y\in \R^d}{\argmin} \ \eta\la m_t, y\ra+ \HH_\Phi(y,p_t)\\
p_{t+1}&=\underset{x\in \R^d,\tr(x)=1}{\argmin}\HH_\Phi(x,y_{t+1}),
\end{align*}
initialized with $y_1=p_1=e/r$, satisfy  for  all $t\ge 1$ 
$$ \la p_t - u, \eta m_t\ra \le   \HH_\Phi(u, y_t)- \HH_\Phi(u, y_{t+1}) +\eta^2\tr(p_t\circ m_t^2).$$
\end{lemma}
{\em Proof.}
We upper bound each term $ \la p_t - u, \eta m_t\ra $ using the following sequence of steps that are justified~below:
\begin{align*} 
\la p_t-u, \eta m_t\ra &\overset{S1}{=}\la p_t-u,  \lnl p_t-\lnl y_{t+1}\ra\\
&\overset{S2}{=}\HH_\Phi(u,p_t) - \HH_\Phi(u,y_{t+1}) + \HH_\Phi(p_t,y_{t+1}) \\
&\overset{S3}{\le}\HH_\Phi(u,y_t) - \HH_\Phi(u,y_{t+1}) + \HH_\Phi(p_t,y_{t+1}) \\
&\overset{S4}{=}\HH_\Phi(u,y_t) - \HH_\Phi(u,y_{t+1}) + \eta \tr(p_t \circ  m_t)\\
&\quad\quad-\tr(p_t)+\tr(y_{t+1}) \\
&\overset{S5}{\le} \HH_\Phi(u,y_t) - \HH_\Phi(u,y_{t+1}) + \eta \tr(p_t \circ  m_t)\\
&\quad\quad-\tr(p_t)+\tr(p_t \circ ( e-\eta m_t+\eta^2m_t^2)
) \\
&= \HH_\Phi(u,y_t) - \HH_\Phi(u,y_{t+1}) + \eta^2\tr(p_t\circ m_t^2).
\end{align*}

\noindent {\bf  S1:} The optimality condition for the first step in decoupled OMD \eqref{eqn:MDdecoup1} gives
$$
\eta m_t+\lnl y_{t+1}-\lnl p_t=0
$$
(see the proof of Lemma \ref{lem:OMD}). Thus, we get that 
$$\la p_t-u,\eta m_t\ra=\la p_t-u,  \lnl p_t-\lnl y_{t+1}\ra.$$

\medskip 

\noindent {\bf  S2:} Using   Lemma \ref{breglemma} \eqref{lem:3term} we have  that 
\[
\la p_t-u,  \lnl p_t-\lnl y_{t+1}\ra =\HH_\Phi(u,p_t)+\HH_\Phi(p_t,y_{t+1})-\HH_\Phi(u,y_{t+1}).
\]

\medskip

\noindent {\bf S3:} We show that 
\begin{align*}
\HH_\Phi(u,p_t)\le \HH_\Phi(u,y_t), \text{ for all } u\in \cc \text{ with } \tr(u)=1.
\end{align*}
Indeed,
$$p_{t}=\underset{x\in \R^d, \tr(x)=1}{\argmin}\HH_\Phi(x,y_{t}),$$
so using the fact that $\nabla_{x} \HH_\Phi (x, y_{t}) = \lnl x - \lnl y_{t}$ (see Lemma \ref{breglemma} \eqref{lem:BregProp2}) we have 
$$\la \lnl p_{t}-\lnl y_{t}, p_{t}-u\ra\le0, \   \forall u\in \cc \text{ with }  \tr(u)=1.$$
Then, by Lemma  \ref{breglemma} \eqref{lem:3term}  we~have 
$$ \la \lnl p_{t}-\lnl y_{t}, p_{t}-u\ra = \HH_\Phi(u,p_{t})+\HH_\Phi(p_{t},y_{t})-\HH_\Phi(u,y_{t})$$
which gives $\ \forall u\in \cc \text{ with }  \tr(u)=1,$
$$\HH_\Phi(u,p_{t})+\HH_\Phi(p_{t},y_{t})-\HH_\Phi(u,y_{t})\le 0.$$
Finally, as  the Bregman  divergence is nonnegative we get
$$\HH_\Phi(u,p_{t})\le \HH_\Phi(u,y_{t}), \ \forall u\in \cc \text{ with }  \tr(u)=1.$$

\noindent {\bf S4:} Next, as $ \HH_\Phi(u,y_t)-\HH_\Phi(u,y_{t+1})$ will telescope in the sum,  it remains to show that $\HH_\Phi(p_t,y_{t+1})$ is bounded. As a first step we expand $\HH_\Phi(p_t,y_{t+1})$ as~follows:
\begin{align*}
\HH_\Phi(p_t,y_{t+1}) &=\tr(p_t\circ \lnl p_t)-\tr(p_t\circ \lnl y_{t+1})+\tr(y_{t+1}) - \tr(p_t)\\
&=\tr(p_t\circ (\lnl p_t-\lnl y_{t+1}))+\tr(y_{t+1})-\tr(p_t)\\
&=\eta \tr(p_t \circ  m_t)+\tr(y_{t+1})-\tr(p_t),
\end{align*}
where the last equality follows from the  OMD decoupling \eqref{eqn:MDdecoup1}.

\medskip

\noindent {\bf S5:} We proceed to upper bound $ \tr(y_{t+1})$, i.e., 
$$\tr(y_{t+1})\le \tr(p_t \circ ( e-\eta m_t+\eta^2m_t^2)).$$
First, using \eqref{eqn:MDdecoup1} we get 
$$
\bal
\tr(y_{t+1})&=\tr(\expl(\lnl p_t-\eta m_t))\\
&\le \tr(\expl(\lnl p_t)\circ \expl(-\eta m_t))\\&=\tr(p_t\circ \expl(-\eta m_t)),
\eal
$$
where the inequality  follows  by  the generalized Golden-Thompson inequality  \cite{ART:TWK21}.
Specifically, 
    let $(\cj,\circ)$ be an EJA. Then, for any $x,y\in\cj,$
    $${\rm tr}(\expl(x+y))\le {\rm tr}(\expl(x)\circ \expl(y)),$$
and equality holds if and only if $x$ and $y$ \emph{operator commute}, i.e., $x$ and $y$ share a common Jordan frame.

Finally, we bound $\exp(-\eta m_t)$ in the cone order. 
 By assumption we have  $\lam_i(\eta m_t)\ge -1$. 
 Moreover, from Lemma~\ref{bound:_exp_real_quad} we have that
 $$\exp(-s)\le 1-s+s^2 \quad   \text{ for }  s\ge -1,$$
which together with Lemma \ref{lem:cone_ineq} implies that 
$$\expl(-\eta m_t)\preceq_\cc e-\eta m_t+\eta^2m_t^2.$$ 
Finally, as the cones of squares is self-dual, for any $p_t\in \cc$ we get 
$$\tr(p_t\circ \expl(-\eta m_t))\le \tr(p_t \circ ( e-\eta m_t+\eta^2m_t^2)).\qed$$

We are finally in a position to bound the regret of \ref{SCMWU}:

 \noindent {\bf Proof of Theorem \ref{thm:regret}.} 
In Lemma \ref{lem:regret}, we established that, where $y_t$ are the intermediate \ref{OMD}-decoupled iterates, 
$$ \la m_t, p_t-u\ra \le \frac{1}{ \eta}(\HH_\Phi(u,y_t)-\HH_\Phi(u,y_{t+1})) +\eta\innerprod{m_t^2}{p_t}.$$
Our regret  bound then follows via
\[
\sum_{t=1}^T\la m_t, p_t-u\ra 
\le \frac{1}{ \eta}(\HH_\Phi(u,y_1)-\HH_\Phi(u,y_{T+1}))+\eta\sum_{t=1}^T\innerprod{m_t^2}{p_t}
\le  \frac{1}{\eta}\ln r + \eta\sum_{t=1}^T\innerprod{m_t^2}{p_t},
\]
where the last step comes because $\HH_\Phi(u,y_{T+1}) \geq 0$ and $\HH_\Phi(u,y_1) \leq \ln r$ by Lemma \ref{lem:_entropy_central_bound} since $y_1 = \frac{e}{r}$.\qed

\section{Online optimization  over SOC vs  online optimization over the ball}
\label{sec:_SCMWU-ball}

 Online optimization over the simplex and the set of density matrices is well understood. However, Theorem \ref{thm:regret} expands the scope to include new cases involving the second-order cone and direct sums of primitive cones. In this section, we specifically investigate the relationship between online optimization over the trace-one slice of the second-order cone and online optimization over the unit ball  $\ball^d=\{x \in \R^d : \ \norm{x}_2 \leq 1\}$. 

For this, we introduce SCMWU-ball, an  iterative scheme for online optimization  over the unit ball:
\begin{equation}
\label{SCMWU-ball}\tag{SCMWU-ball}
	b_{t+1} =2 \cdot  \nexp(- \eta \sum_{i=1}^t m_i),
\end{equation}
where the normalized exponential map for a vector $x \in \R^d$  is defined by
\begin{equation*}
\begin{split}
	\nexp: \R^d &\to \itr(\hb), \\
	 x \;  &\mapsto 
			\frac{
			\exp(\norm{x}_2)
			}{
			\exp( \norm{x}_2) + \exp(- \norm{x}_2)
			}
			\frac{x}{2 \norm{x}_2}
			+
			\frac{
			\exp(-\norm{x}_2)
			}{
			\exp( \norm{x}_2) + \exp(- \norm{x}_2)
			}
			\left(-\frac{x}{ 2\norm{x}_2} \right).
\end{split}
\end{equation*}
 
 Our main result in this section  is the following: 

 \begin{theorem}
\label{thm:regret-ball} 
 For any sequence of loss vectors  $m_t \in \R^d$ 
satisfying $\norm{m_t}_2 \leq 1$,  the iterates
$\{b_t\}_t$ of \ref{SCMWU-ball} with stepsize ${\eta \leq 1}$  satisfy
\[
    \sum_{t=1}^{T} {m_t}^\top(b_t-u) \le
    { \eta}  \sum_{t=1}^T \|m_t\|_2^2 +\frac{1}{\eta}\ln 2.
\]
\end{theorem}

\begin{proof} 
Consider a sequence of loss vectors $m_t\in \R^d$ with $\|m_t\|_2\le 1$. Note that 
$$\|m_t\|_2\le 1 \iff -e\preceq_{\rm SOC}   (m_t,0)\preceq_{\rm SOC} e.$$
We apply \ref{SCMWU} for learning over the trace-one slice of the second-order cone  using the sequence of loss vectors $(m_t,0).$ The \ref{SCMWU} iterates are given by 
\begin{equation}\label{scmwusoc}
    p_{t+1}=\frac{ \expl(-\eta \sum_{i=1}^t (m_i,0))} {\tr(\expl(-\eta \sum_{i=1}^t (m_i,0)))}=\frac{ \expl (-\eta \sum_{i=1}^tm_i,0)} {\tr(\expl( -\eta \sum_{i=1}^t m_i,0)))}. 
\end{equation}
Next, we use that 
\begin{equation}\label{xcsvdsfv}
\frac{\expl((x, s))}{\tr(\expl((x, s)))}=\left(\nexp(x), {1\over 2}\right).
\end{equation} 
	Indeed, the two eigenvalues of $(x, s)$ are $s\pm  \norm{x}_2$ with   corresponding  eigenvectors
	\[
		q_{+} = \frac{1}{2} \left( \frac{x}{\norm{x}_2},  \, 1 \right),
		\qquad
		q_{-} = \frac{1}{2} \left( - \frac{x}{\norm{x}_2},  \, 1 \right).
	\]
 Thus, we have that
	\[
		\tr(\expl((x, s)))
		=
		\exp(s + \norm{x}_2) + \exp(s - \norm{x}_2)
		=
		\exp(s) \left( \exp( \norm{x}_2) + \exp(- \norm{x}_2) \right).
	\]
Consequently, 
	\begin{equation*}
	\begin{split}
		\frac{\expl((x, s))}{\tr(\expl((x, s)))}
		&=
		\frac{
			\exp(s + \norm{x}_2) q_{+} + \exp(s - \norm{x}_2) q_{-}
			}{
				\exp(s) \left( \exp( \norm{x}_2) + \exp(- \norm{x}_2) \right)
			} \\
		&=
		\frac{
			\exp(\norm{x}_2)
			}{
			\exp( \norm{x}_2) + \exp(- \norm{x}_2)
			}
			q_{+}
		+
		\frac{
			\exp(-\norm{x}_2)
			}{
			\exp( \norm{x}_2) + \exp(- \norm{x}_2)
			}
			q_{-},
	\end{split}
	\end{equation*}
 which is exactly \eqref{xcsvdsfv}.  Combining \eqref{scmwusoc} and \eqref{xcsvdsfv} we get that 
 $$p_{t+1}=\left(\nexp(-\eta \sum_{i=1}^t m_i ), {1\over 2}\right),
$$
and as $p_{t+1}\in {\rm SOC}_d$ it follows that 
$$b_{t+1}=2\cdot \nexp(-\eta \sum_{i=1}^t m_i )\in \ball^d.$$

Finally, as  $ -e\preceq_{\rm SOC}   (m_t,0)\preceq_{\rm SOC} e$ the regret bound of Theorem \ref{thm:regret} applies, i.e., 
\[
    \sum_{t=1}^{T} \innerprod{(m_t,0)}{ \left({b_t\over 2},{1\over 2} \right)- \left({u\over 2}, {1 \over 2} \right)} \le
    \eta \sum_{t=1}^T \innerprod{(m_t, 0)^2}{\left({b_t\over 2},{1\over 2}\right)} +\frac{1}{\eta}\ln 2
\]
for all $ u\in\ball^d$. Using the fact that 
$$(m_t, 0)^2 = (m_t, 0) \circ (m_t, 0) = (0_d, \norm{m_t}_2^2)$$ and recalling that the EJA inner product for the {\rm SOC} is twice the Euclidean inner product (see equation \eqref{SOC_innerprod}) gives$$\innerprod{(m_t, 0)^2}{\left({b_t\over 2},{1\over 2}\right)} =\norm{m_t}_2^2,$$
while on the left-hand side of the inequality we have that
\[
    \innerprod{(m_t,0)}{ \left({b_t\over 2},{1\over 2} \right)- \left({u\over 2}, {1 \over 2} \right)}
    =
    \frac{1}{2} \innerprod{(m_t, 0)}{(b_t - u, 0)}
    =
    m_t^\top (b_t - u).
\]
Putting everything together we get that 
\[
    \sum_{t=1}^{T} m_t^\top(b_t-u) \le
    { \eta}  \sum_{t=1}^T \|m_t\|_2^2 +\frac{1}{\eta}\ln 2 \quad \text{ for all }u\in \ball^d.
\]
\end{proof}

We now make some remarks on Theorem \ref{thm:regret-ball}.   Since $\norm{m_t}_2 \leq 1 \ \forall t$, we can get from Theorem \ref{thm:regret-ball} that
	\[
		    \sum_{t=1}^{T}m_t^\top(b_t-u) \leq
		  \eta T +\frac{1}{\eta}\ln 2.
	\]

The stepsize $\eta$ that minimizes the  bound  is $\sqrt{ \ln 2/T}$, which gives regret \begin{equation}
\label{regbound_opt_ball}
2 \sqrt{ T  \ln 2  }.
\end{equation}
Finally, if there is no a priori knowledge of the time horizon $T$,  a no-regret algorithm can be obtained from \ref{SCMWU-ball} using the well-known ``doubling trick'', e.g., see \cite{shwartz}. Following this approach, time is broken up into epochs of exponentially increasing length $2^i$, where over the $i$-th epoch \ref{SCMWU-ball} is run with optimized stepsize $\eta_i = \sqrt{2^{-i} \ln 2}$. At time $T$, the regret of this algorithm is upper bounded by 
\begin{equation}
\label{doub_bound_ball}
\frac{\sqrt{2}}{\sqrt{2}-1} \cdot 2 \sqrt{T \ln 2} 
= \frac{2 \sqrt{2}}{\sqrt{2}-1} \sqrt{T \ln 2}.
\end{equation}

\section{Experimental evaluation}\label{sec:experiment2}
\subsection{Regret accrued over symmetric cones of various direct-sum structures}
In our first set of experiments  we  plot  the regret accumulated by \ref{SCMWU} over the trace-one slice of symmetric cones of various direct-sum structures. Figure \ref{fig:dirsum} compares the regret accumulated by \ref{SCMWU} (with doubling trick) run over 100 different sequences of randomized loss vectors against the theoretical upper bound obtained in Equation \eqref{doub_bound}. 

\begin{figure}[!htb]
    \centering
    \begin{minipage}{.24\linewidth}
      \centering 
      
      \includegraphics[width=.95\linewidth]{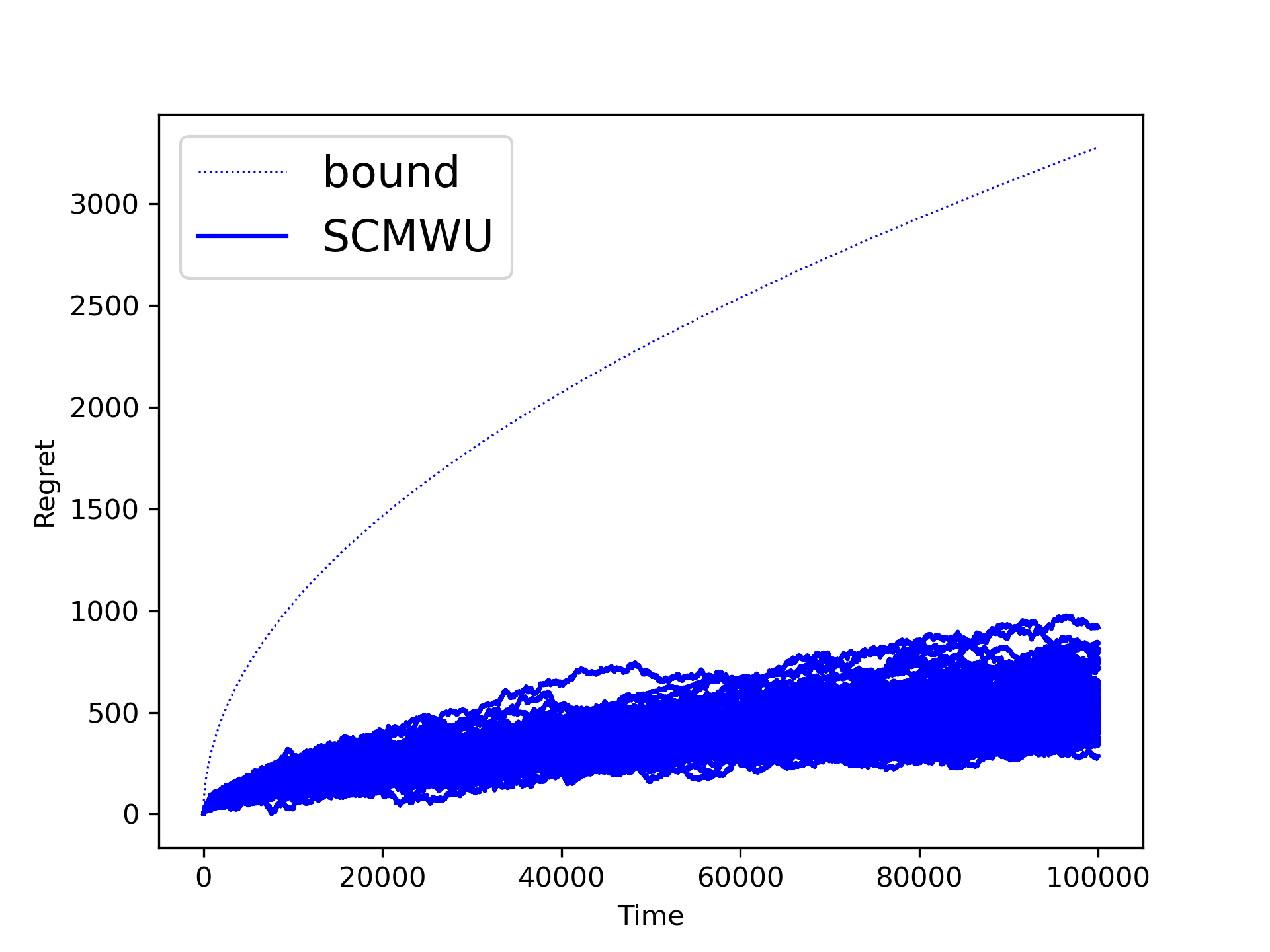}
      $\bigoplus_{d=2}^6 \SOC_d$
    \end{minipage}
    \begin{minipage}{.24\linewidth}
      \centering
      
      \includegraphics[width=.95\linewidth]{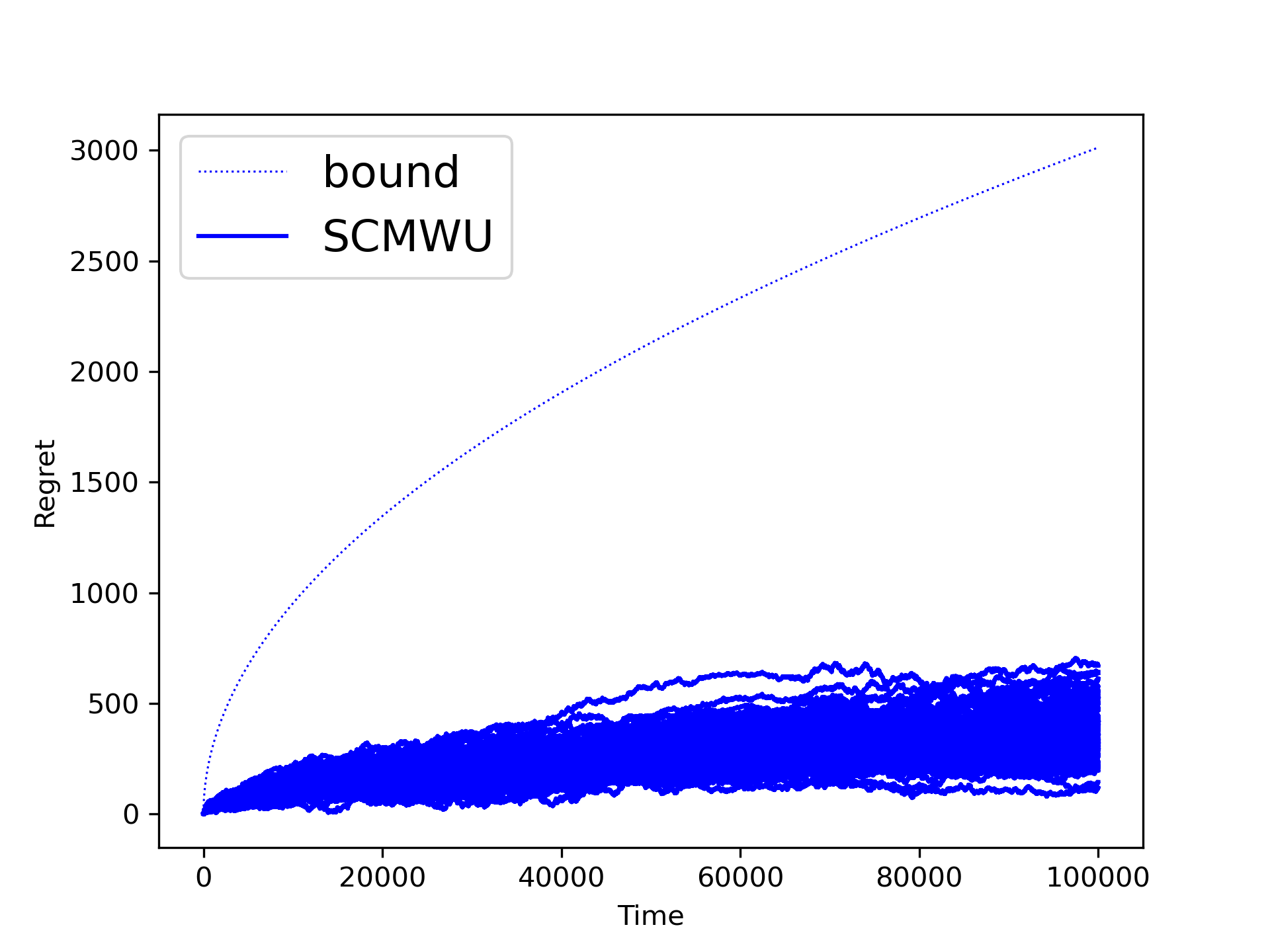}
      $\NO^5 \oplus \SOC_5$
    \end{minipage}
    \begin{minipage}{.24\linewidth}
      \centering  
      
      \includegraphics[width=.95\linewidth]{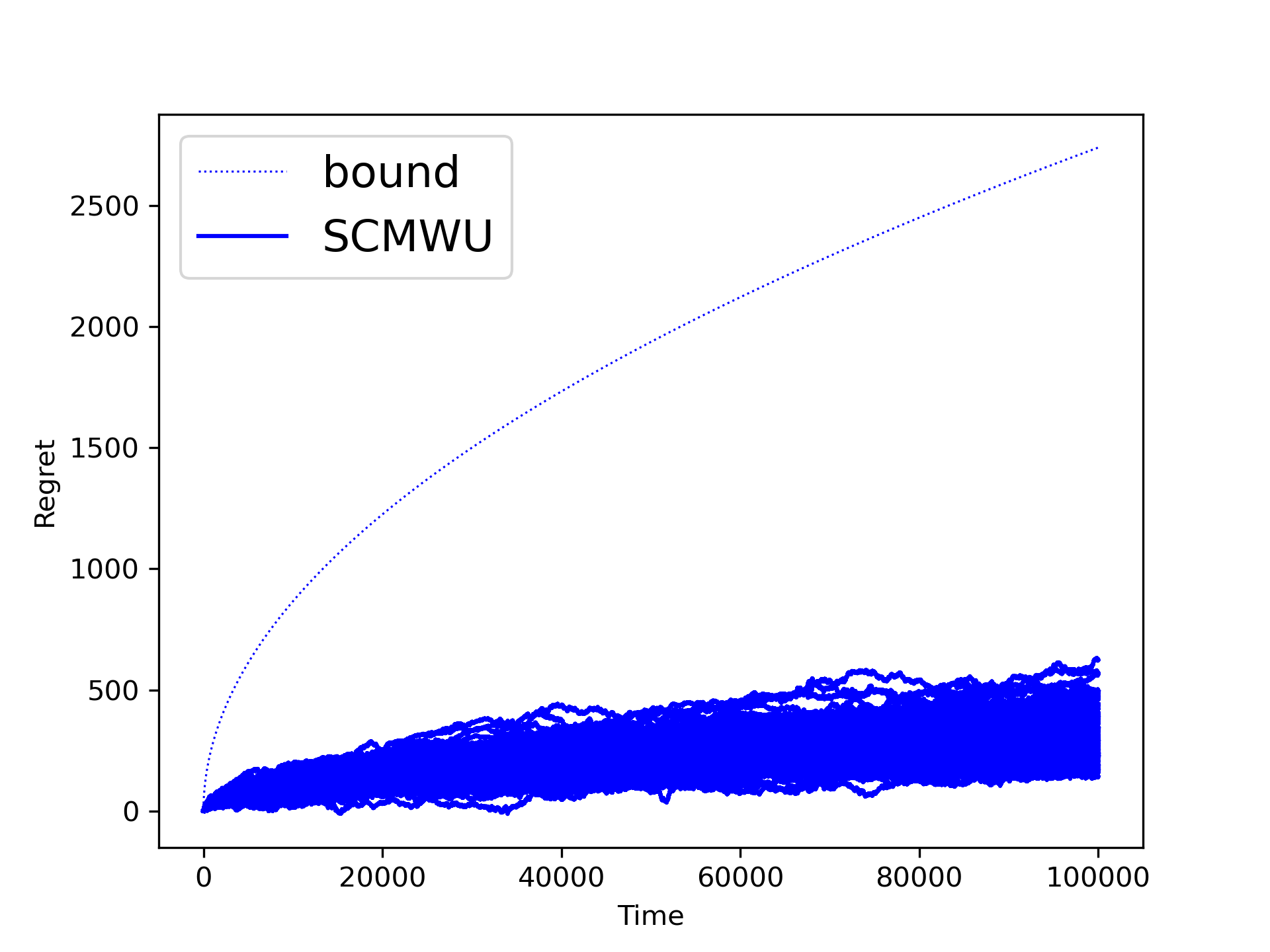}
      $\psdc^3 \oplus \SOC_5$
    \end{minipage}
    \begin{minipage}{.24\linewidth}
      \centering   
      
      \includegraphics[width=.95\linewidth]{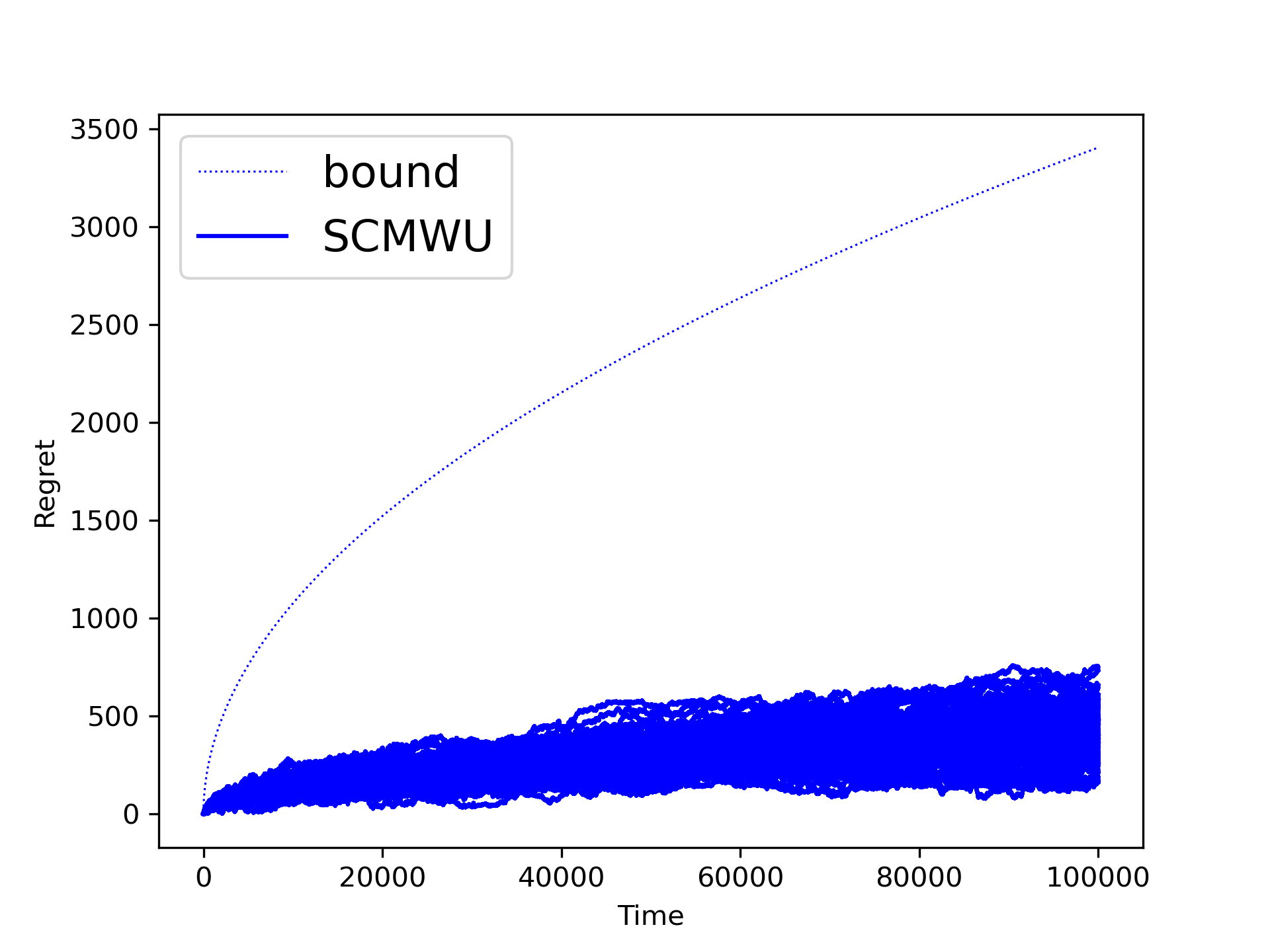}
      $\cc$
    \end{minipage}
    
    \caption{Regret accrued by \ref{SCMWU} on the task of online linear optimization over the $\tr=1$ slice of symmetric cones of various direct-sum structures. The rightmost graph is for the symmetric cone $\cc = \NO^3 \oplus \psdc^2 \oplus \psdc^3 \oplus \SOC_2 \oplus \SOC_3$.}
    \label{fig:dirsum}
\end{figure}

We next perform  numerical experiments to evaluate  the performance of \ref{SCMWU-ball} on the task of online linear optimization over the unit ball. Figure \ref{fig:SCMWU-ball_regbound} compares the regret accumulated by \ref{SCMWU-ball} (with doubling trick) on the task of online linear optimization over unit balls of different dimension against the theoretical upper bound obtained in equation \eqref{doub_bound_ball}, verifying the theoretical regret bound of Theorem \ref{thm:regret-ball}. Since \ref{OGD} is easy to implement in this setting, we also compare in Figure \ref{fig:_vs_OGD_d10_T4} the performance of \ref{SCMWU-ball} against that of \ref{OGD} for the task of online linear optimization over the unit ball, and find that the regret picked up by both algorithms on the same sequences of losses is comparable, with neither algorithm being clearly better.

\begin{figure}[!htb]
    \centering
    \begin{minipage}{.24\linewidth}
      \centering 
      
      \includegraphics[width=.95\linewidth]{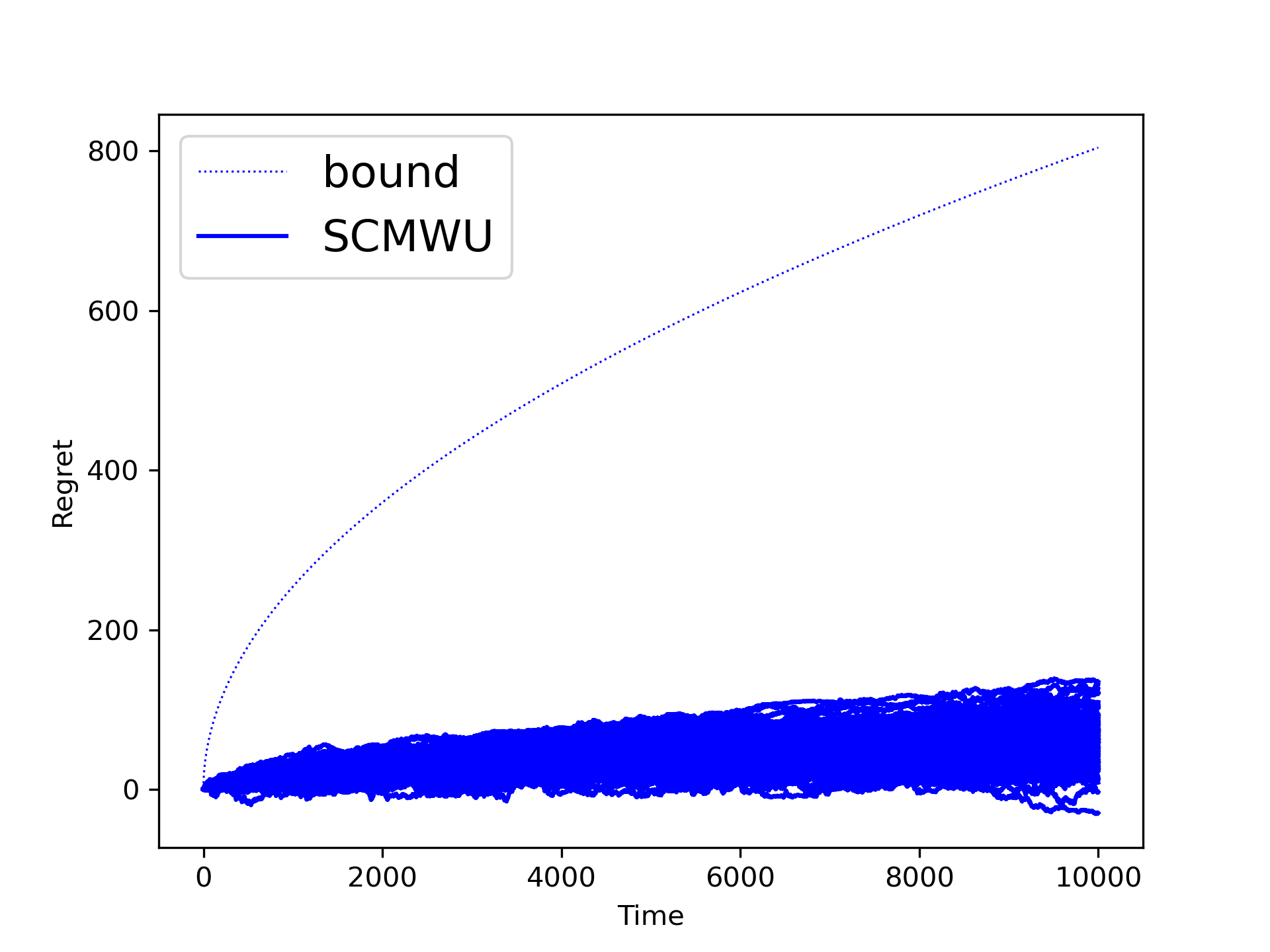}
      $d=2$
    \end{minipage}
    \begin{minipage}{.24\linewidth}
      \centering
      
      \includegraphics[width=.95\linewidth]{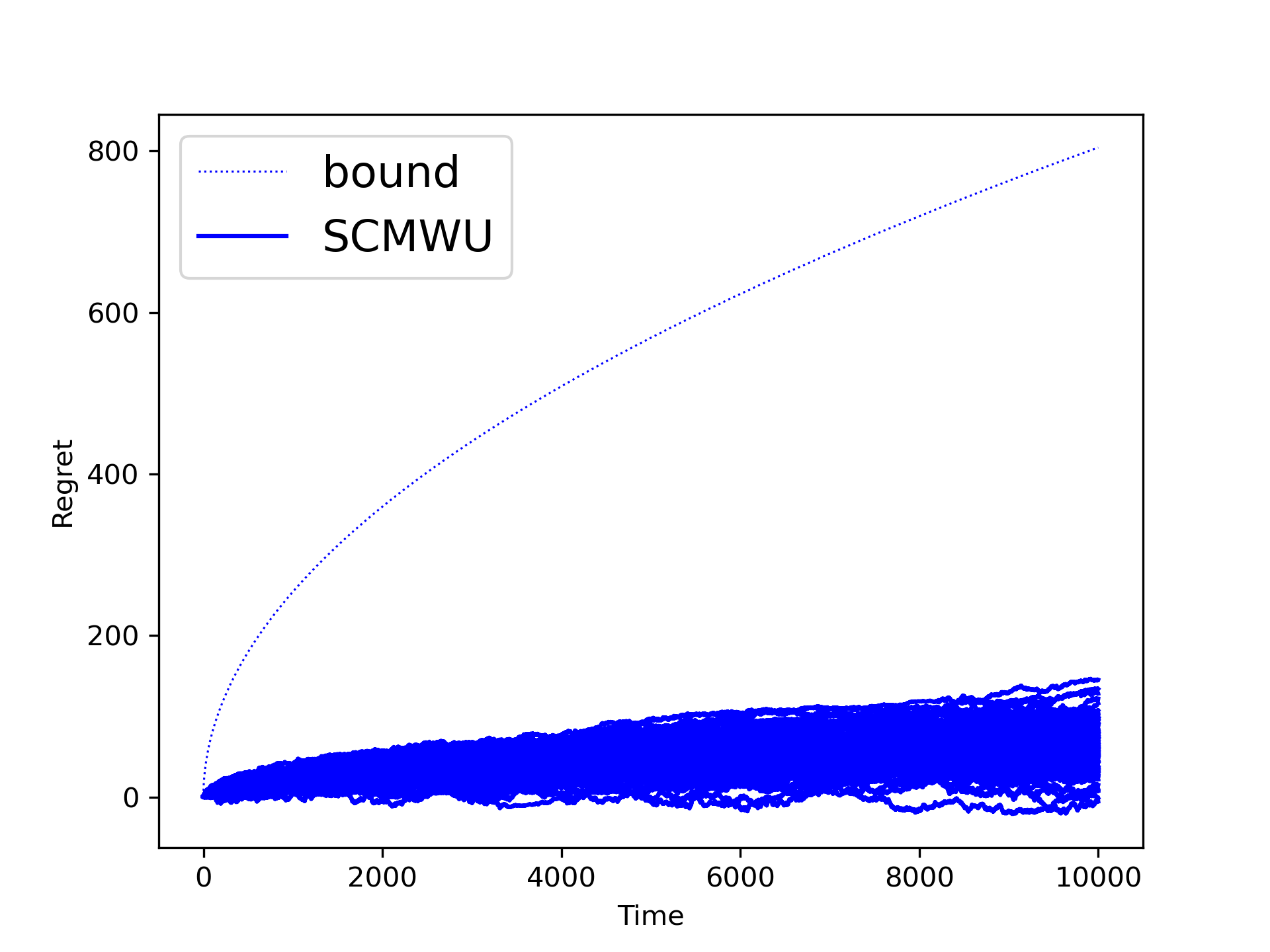}
      $d=3$
    \end{minipage}
    \begin{minipage}{.24\linewidth}
      \centering  
      
      \includegraphics[width=.95\linewidth]{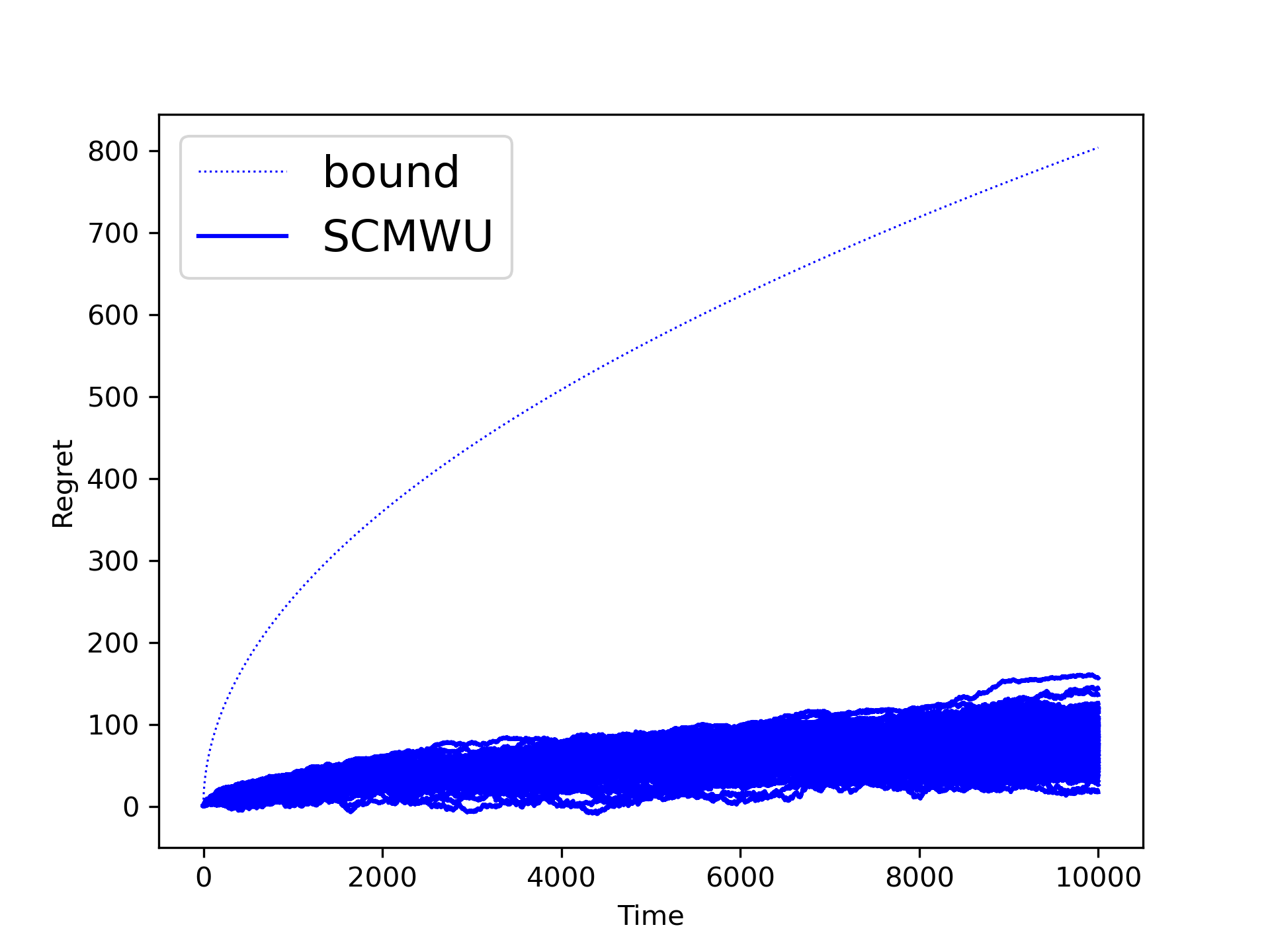}
      $d=5$
    \end{minipage}
    \begin{minipage}{.24\linewidth}
      \centering   
      
      \includegraphics[width=.95\linewidth]{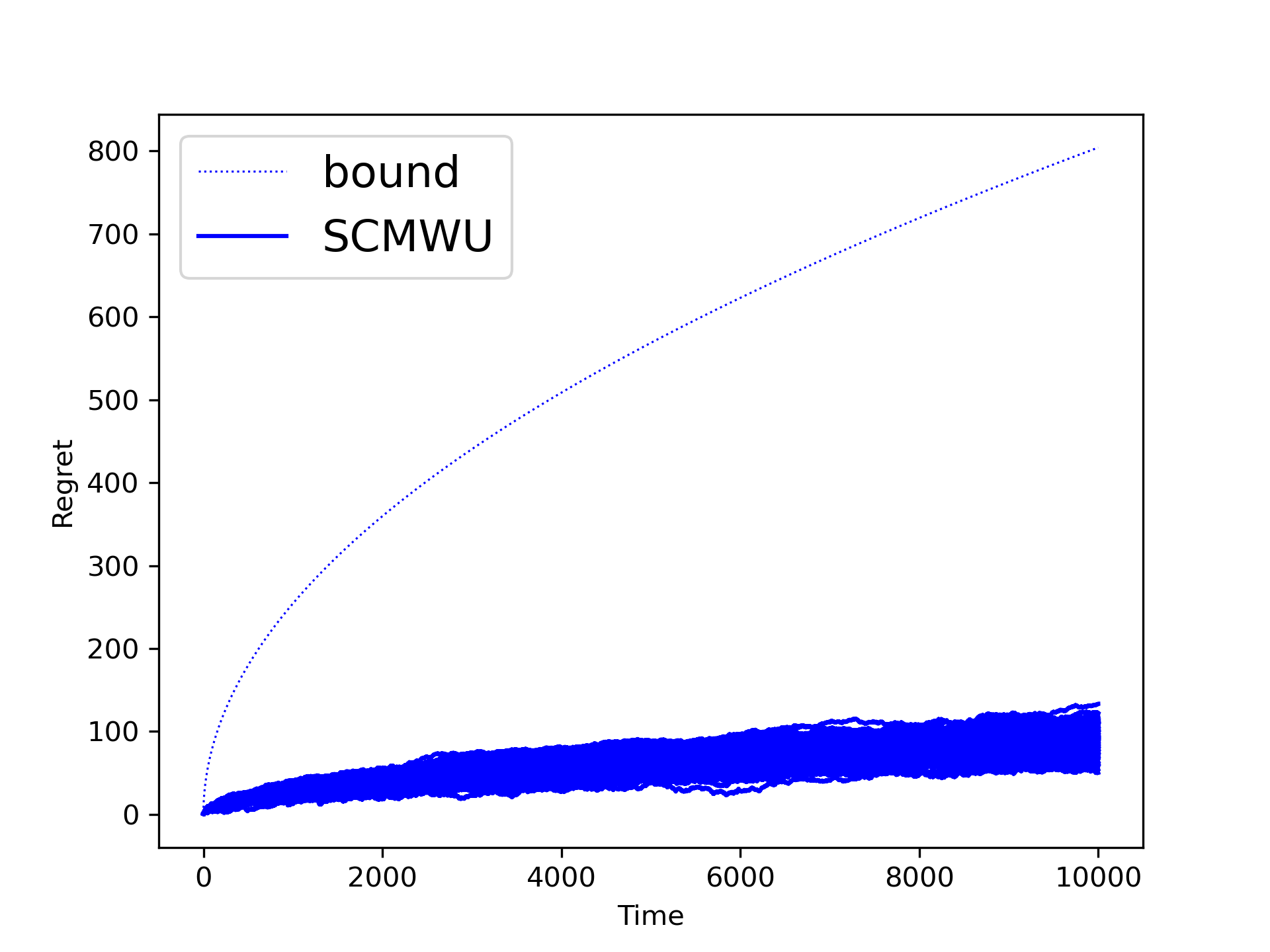}
      $d=10$
    \end{minipage}
    
    \caption{Regret accrued by \ref{SCMWU-ball} on the task of online linear optimization over the $d$-dimensional unit ball, plotted over 100 different sequences of randomized loss vectors for each different dimension $d$.}
    \label{fig:SCMWU-ball_regbound}
\end{figure}

\begin{figure}[!htb]
\centering
    \begin{subfigure}{.19\linewidth}
        \centering
        \includegraphics[width=\linewidth]{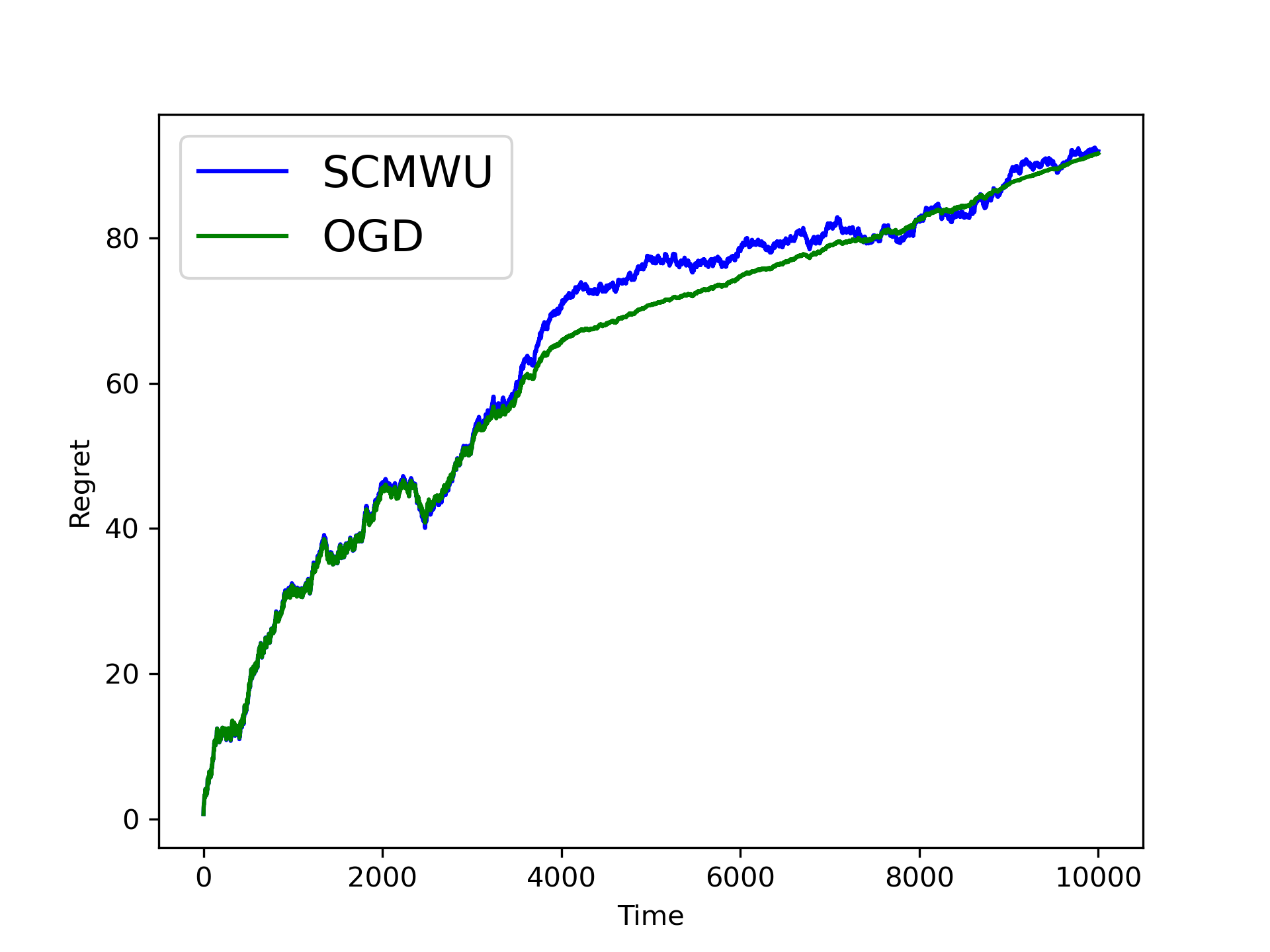}
    \end{subfigure}
    \begin{subfigure}{.19\linewidth}
        \centering
        \includegraphics[width=\linewidth]{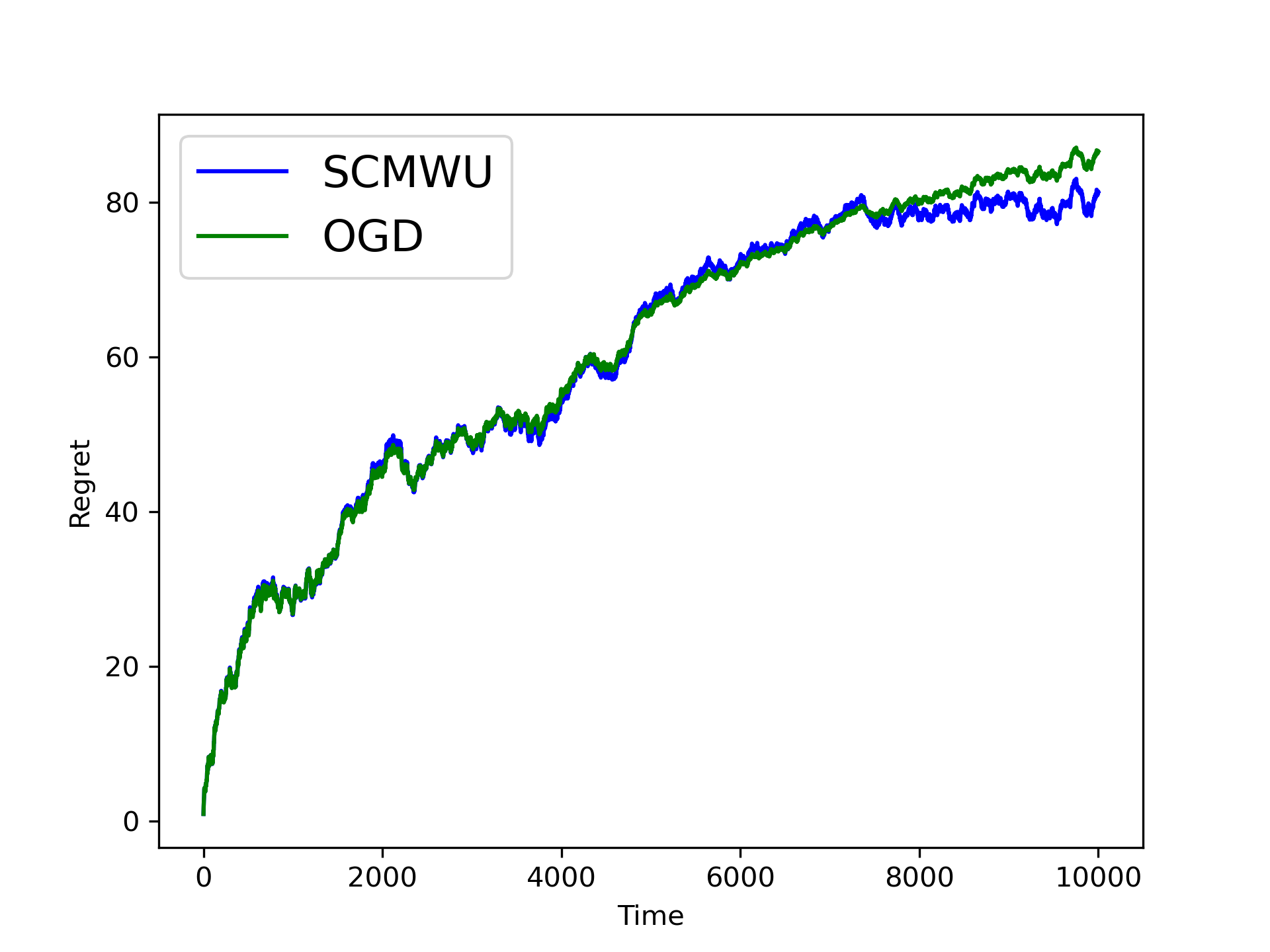}
    \end{subfigure}
    \begin{subfigure}{.19\linewidth}
        \centering
        \includegraphics[width=\linewidth]{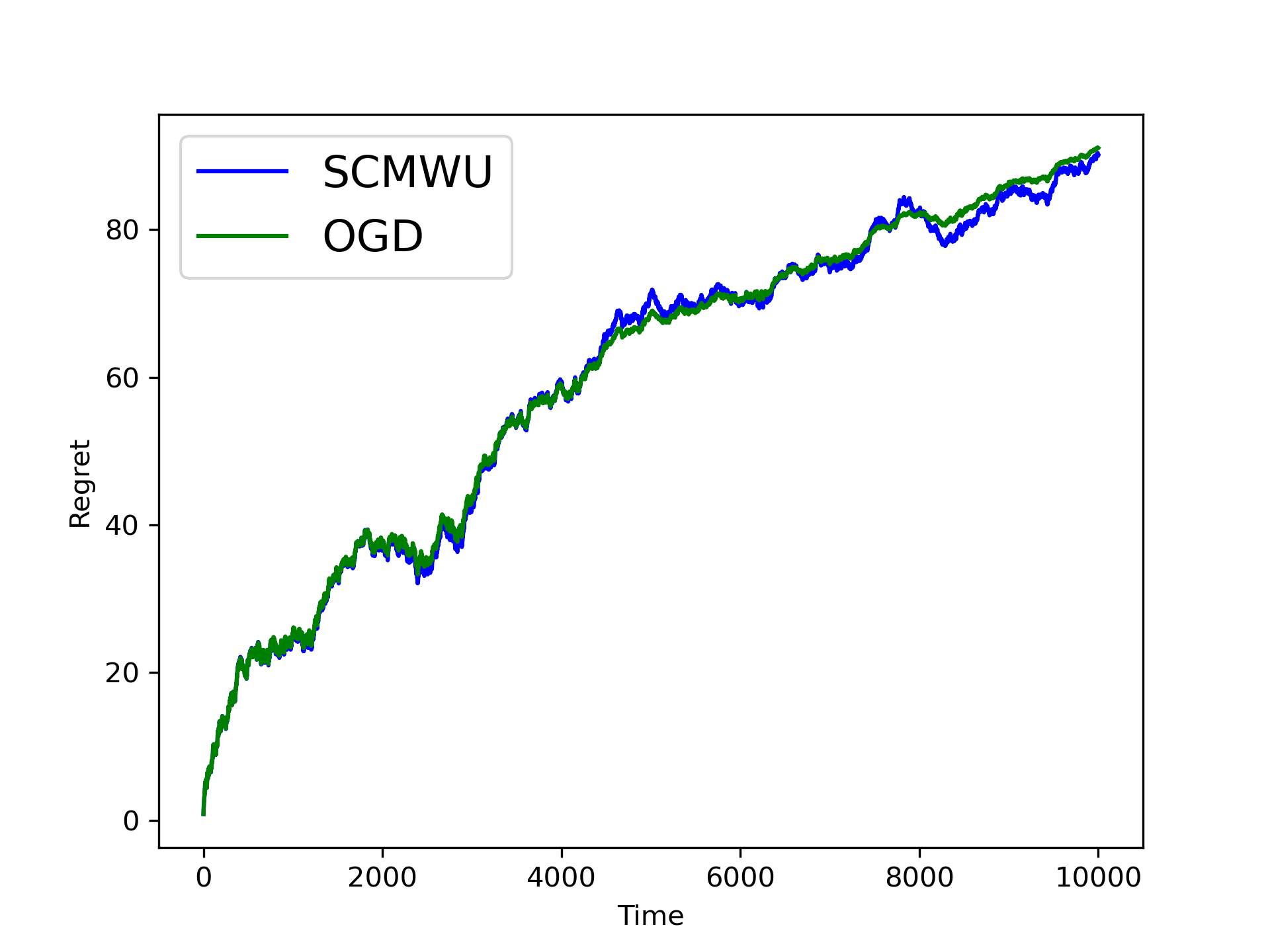}
    \end{subfigure}
    \begin{subfigure}{.19\linewidth}
        \centering
        \includegraphics[width=\linewidth]{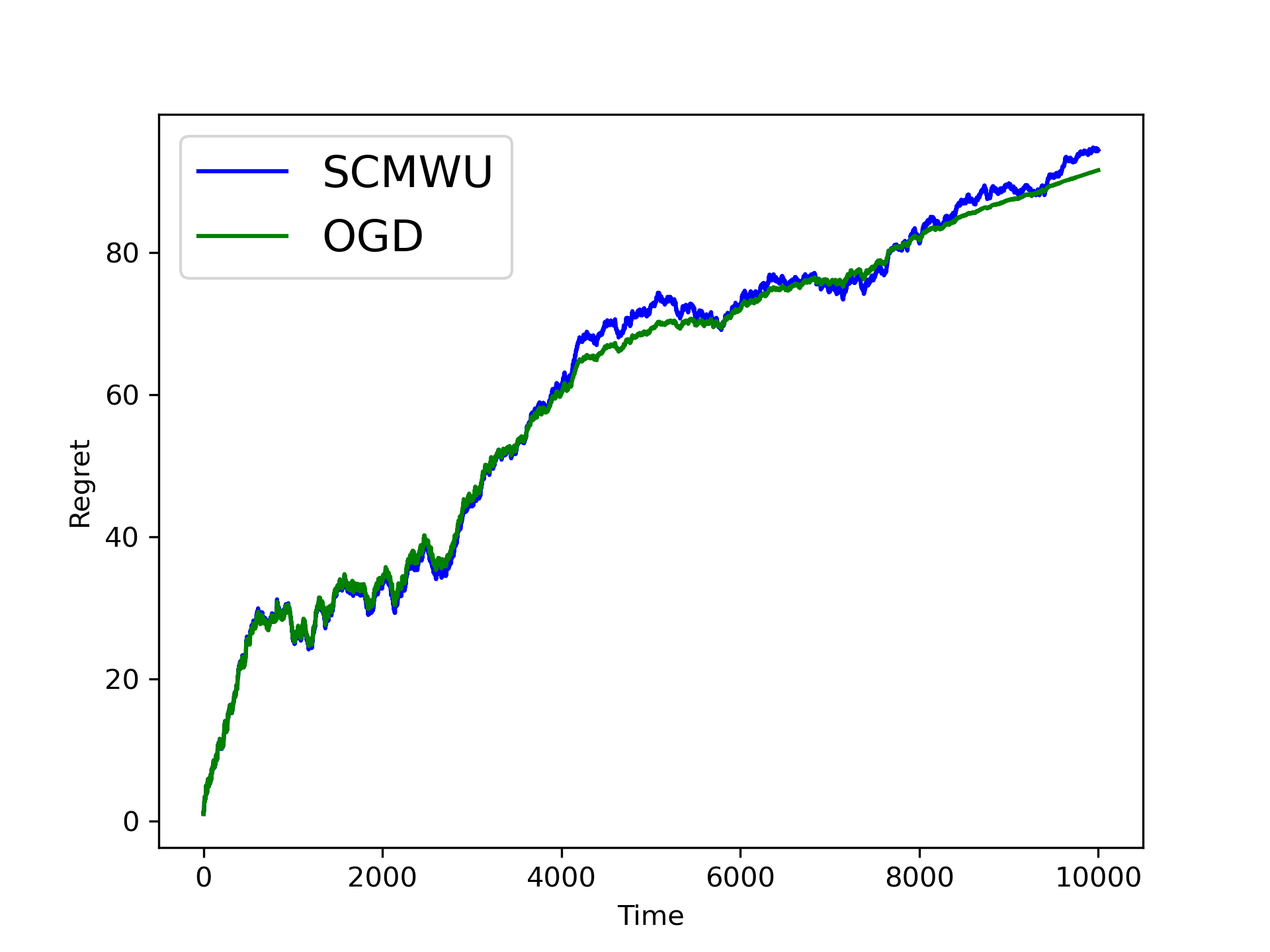}
    \end{subfigure}
    \begin{subfigure}{.19\linewidth}
        \centering
        \includegraphics[width=\linewidth]{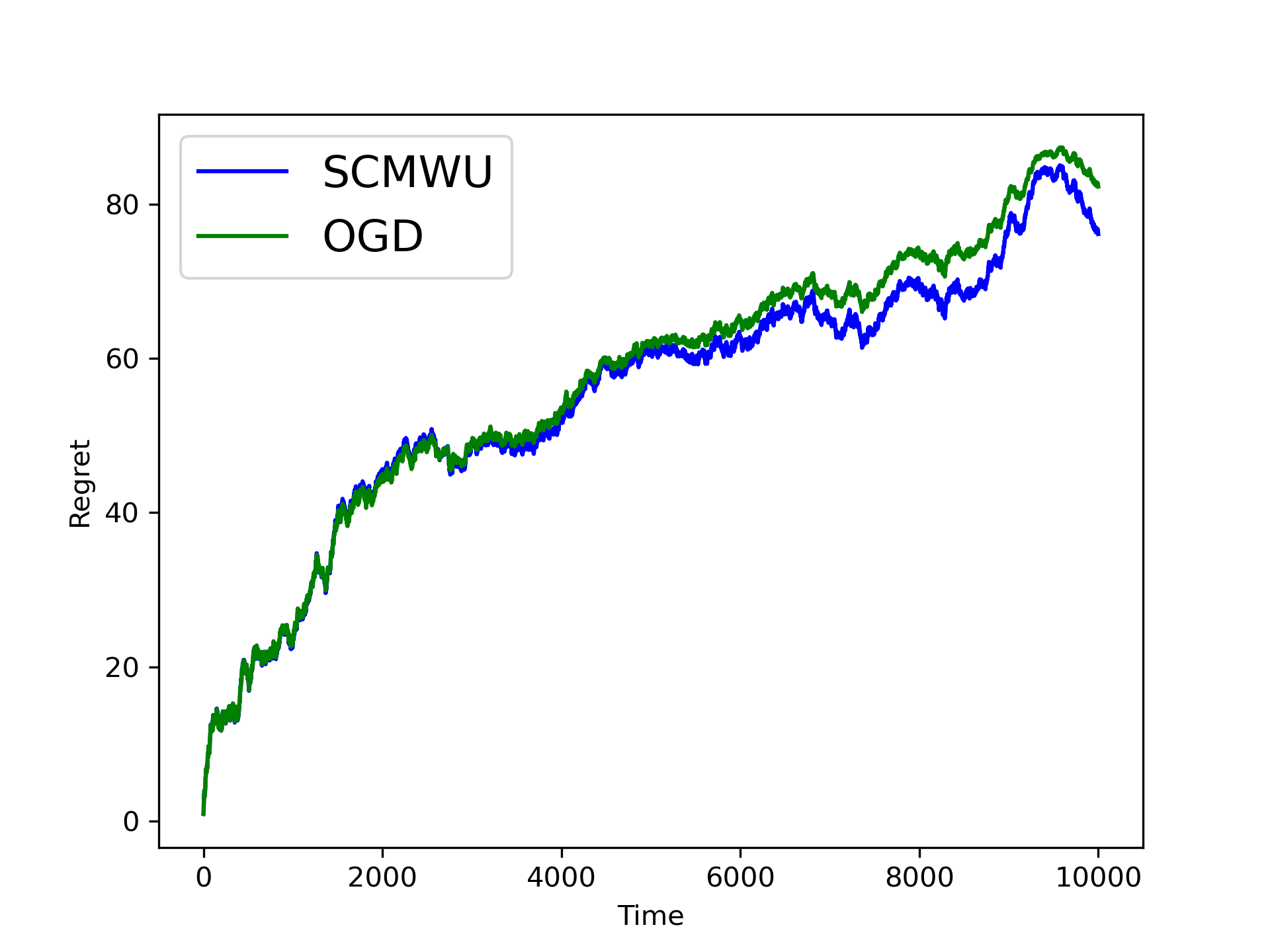}
    \end{subfigure}
    \begin{subfigure}{.19\linewidth}
        \centering
        \includegraphics[width=\linewidth]{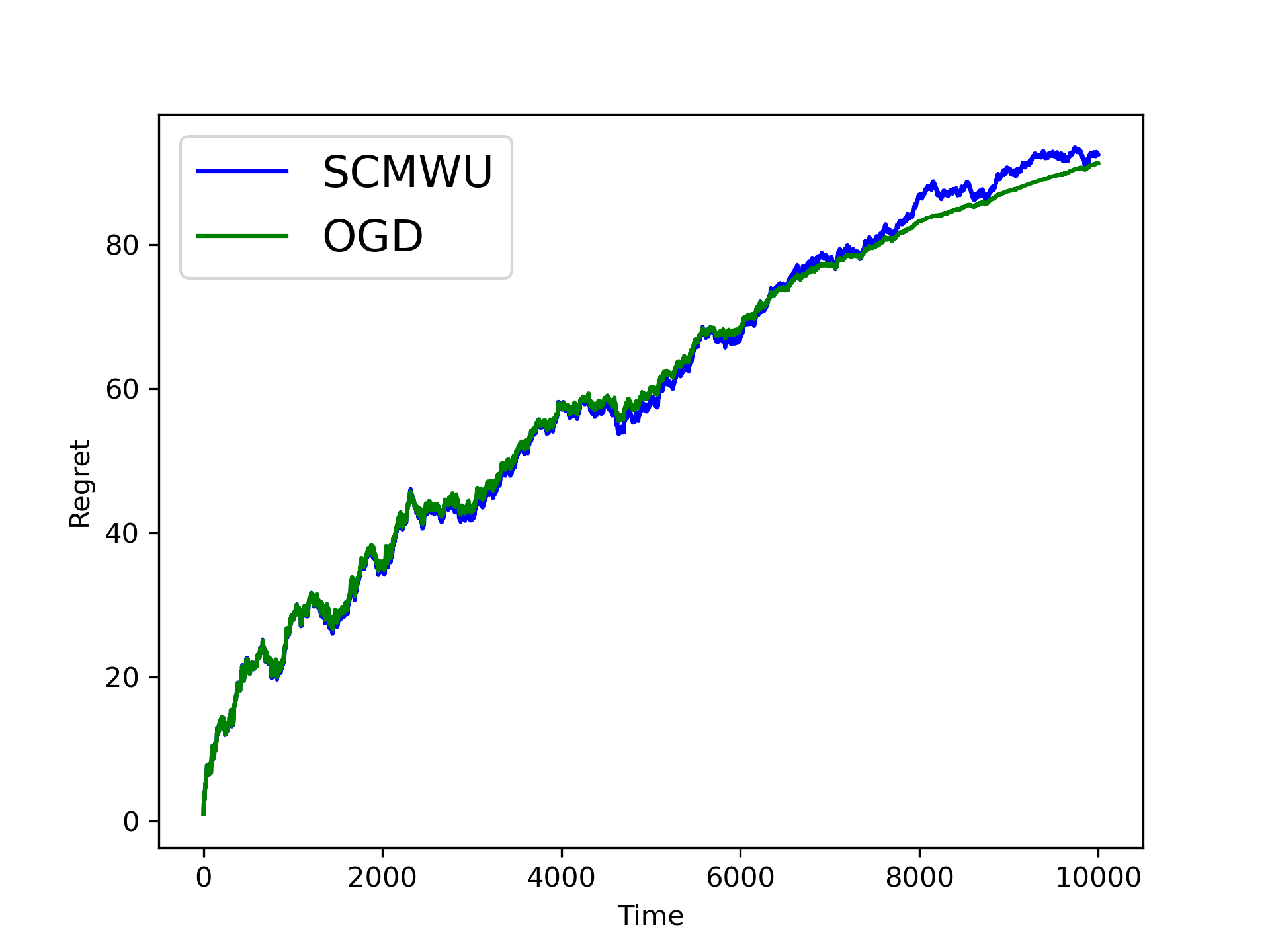}
    \end{subfigure}
    \begin{subfigure}{.19\linewidth}
        \centering
        \includegraphics[width=\linewidth]{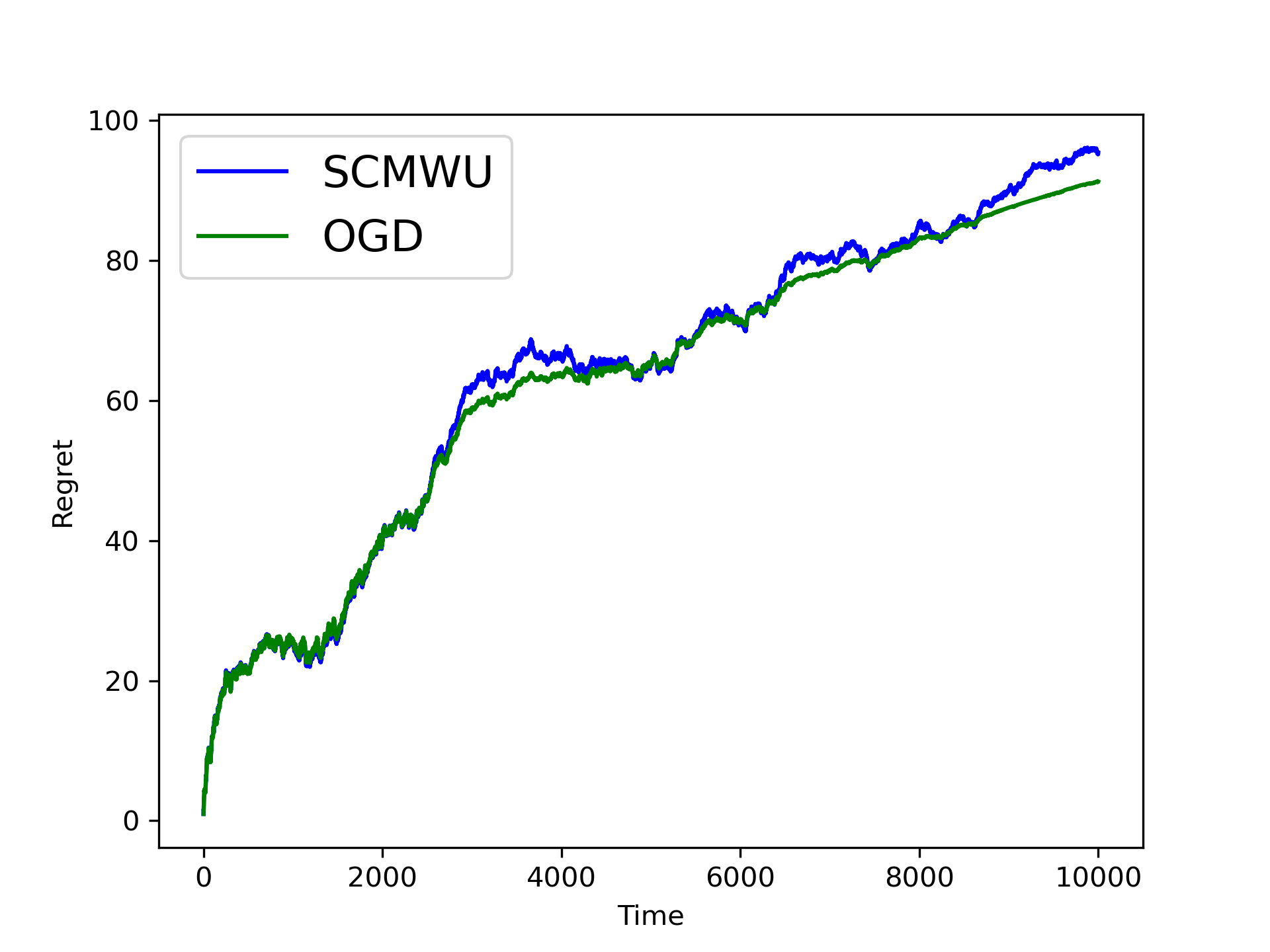}
    \end{subfigure}
    \begin{subfigure}{.19\linewidth}
        \centering
        \includegraphics[width=\linewidth]{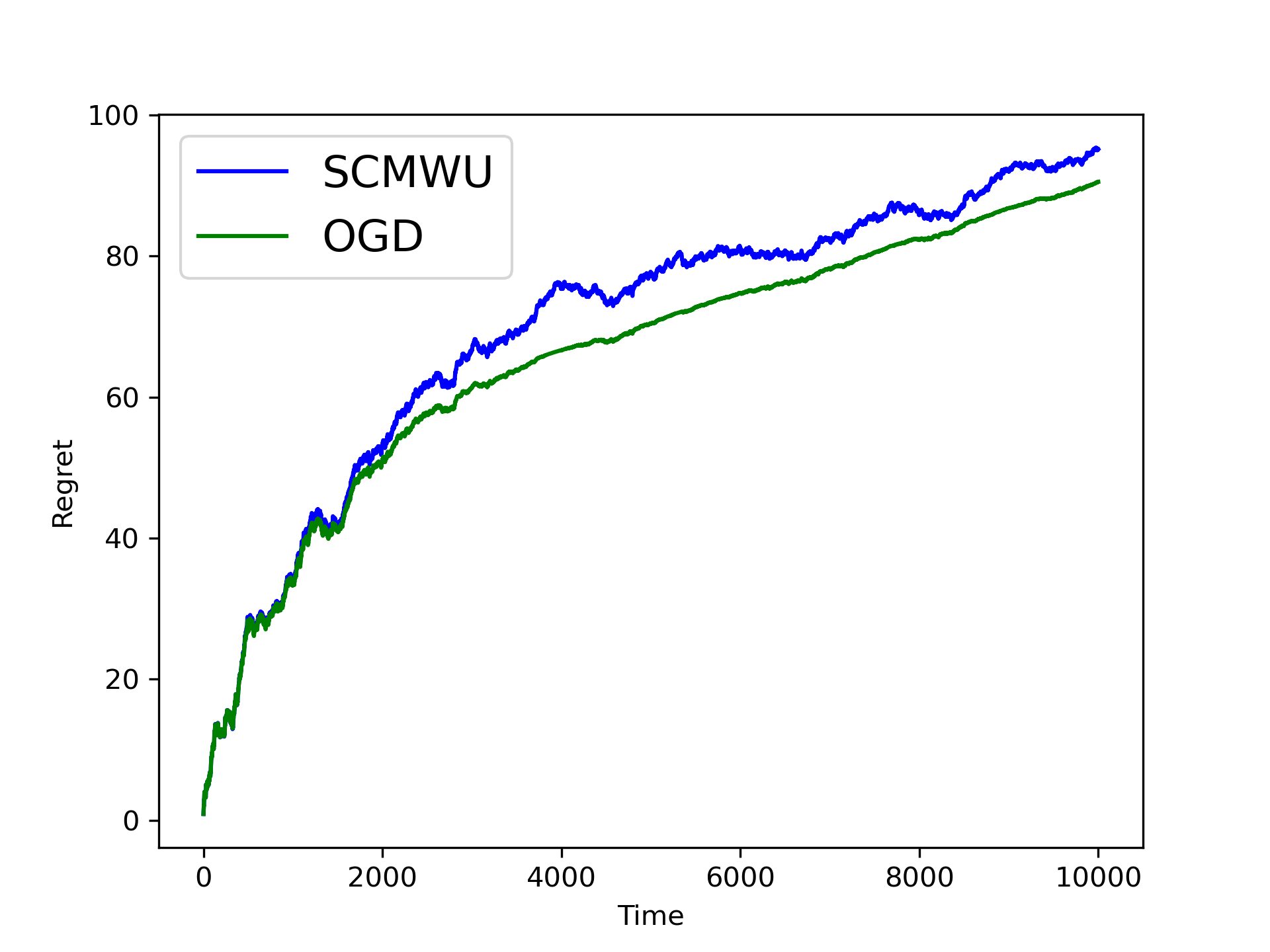}
    \end{subfigure}
    \begin{subfigure}{.19\linewidth}
        \centering
        \includegraphics[width=\linewidth]{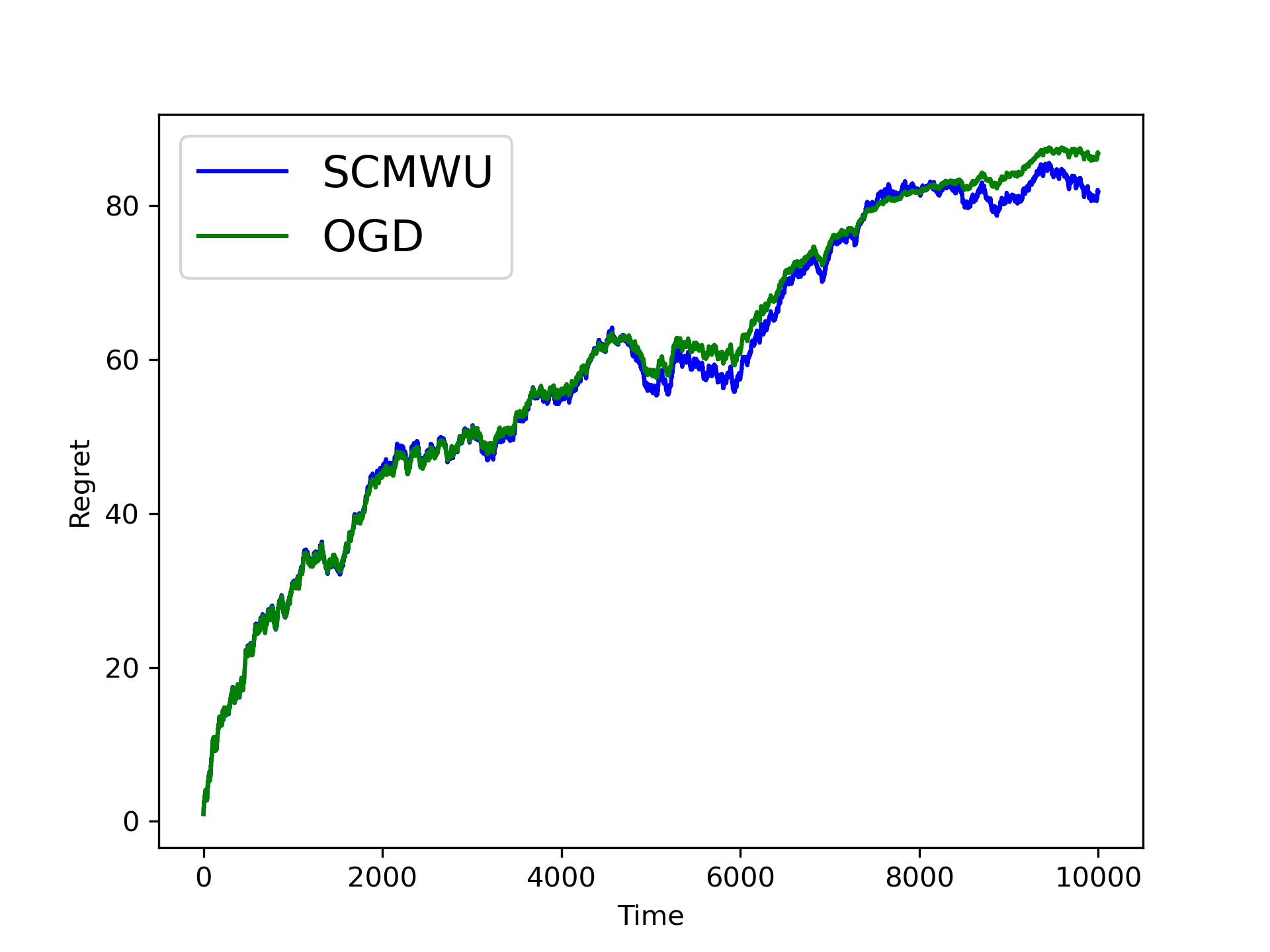}
    \end{subfigure}
    \begin{subfigure}{.19\linewidth}
        \centering
        \includegraphics[width=\linewidth]{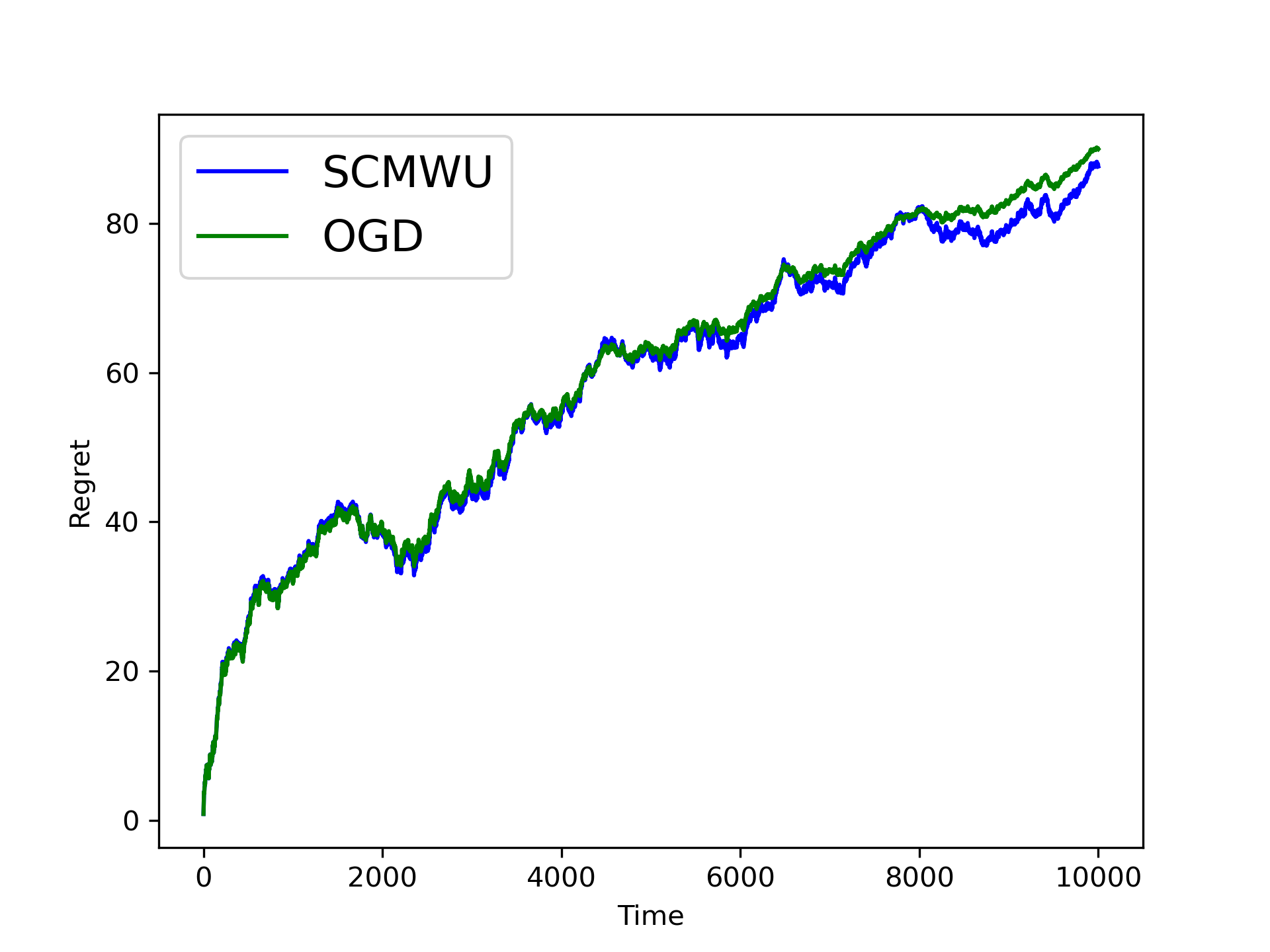}
    \end{subfigure}
    \caption{
    	Regret accrued by \ref{SCMWU-ball} vs \ref{OGD} on the task of online linear optimization over the unit ball in $\R^{10}$. Randomized loss vectors, time horizon of $10^4$, fixed stepsize optimized to time horizon.
	} 
    \label{fig:_vs_OGD_d10_T4}
\end{figure}

\subsection{Online learning in  $l_2$-$l_1$ games}
\label{sec:_SVM_game}

\paragraph{SVM game.} In the Support Vector Machine (SVM) problem we are given $n$ points in $\R^d$   with labels $\pm 1$ and want to find a hyperplane through the origin that separates the two classes  and has the maximum possible margin. Without loss of generality, we can assume that all labels are $+1$, in which case we are trying to find a maximum-margin supporting hyperplane. The SVM dataset can be encoded as the rows of a $n \times d$ matrix $A$. Such a hyperplane is specified by its unit normal vector $x$ satisfying $Ax \geq 0$, and its margin is given by
\[
	m(x) = \min_{1 \leq i \leq n} (Ax)_i= \min_{p \in \triangle} p^\top A x.
\]
Furthermore, the maximum margin attainable is equal to
\[
    \max_{x: \ \norm{x}_2 = 1} m(x) = \max_{x: \norm{x}_2 = 1}  \min_{p \in \triangle} p^\top A x.
\]
Assuming the data is linearly separable (i.e., $\exists \, x \in \R^d$ such that $\min_{p \in \triangle} p^\top A x > 0$), we also have that the maximum attainable margin
\[
    \max_{x: \ \norm{x}_2 = 1} m(x) = \max_{x: \ \norm{x}_2 \leq 1} m(x).
\]
Indeed, equality holds because if $\min_{p \in \triangle} p^\top A x > 0$ for some $x : \ \norm{x}_2 < 1$, then $$\min_{p \in \triangle} p^\top A \frac{x}{\norm{x}_2} > \min_{p \in \triangle} p^\top A x.$$

This naturally leads to the SVM game \cite{clarkson2012sublinear, carmon2019variance}, a zero-sum game played between a \emph{classifier}, who picks a normal vector $x$ in the unit ball (in the $l_2$-norm) and  seeks to find the maximum-margin classifier, and an \emph{adversary}, who picks a distribution $p$ over the data points (i.e., a point in the probability simplex, which is the unit ball in the $l_1$-norm) and seeks to find support vectors of the classifier $x$ in the dataset $A$.  The $x$-player attains utility $p^\top A x$, while the $p$-player attains utility $-p^\top A x$. The value of the SVM game is the maximum margin of a classifier, i.e.,
\[
	v(A) 
 = \max_{x \in \mathbb{B}} \min_{p \in \triangle} p^\top A x 
  = \min_{p \in \triangle} \max_{x \in \mathbb{B}}  p^\top A x,
\]
and any Nash equilibrium of the SVM game is a pair $(x^*, p^*)$  of a maximum-margin classifier $x^*$ and a distribution $p^*$ over support vectors. Using well-known results from the theory of learning in games \cite{cesabianchi}, if each player adjusts their strategies using no-regret algorithms, then their time-averaged strategies  converge to Nash equilibria of the SVM game. Specifically, we have the following result:
\begin{theorem}
\label{thm:SVMgameconvergence}
Consider an SVM instance where all the data points lie in the unit ball  $\ball^d$, i.e., $\norm{A_{i :}}_2 \leq 1 \ \forall i \in [n]$. 
Suppose  we fix the time horizon $T$ and  use \ref{MWU} to learn a distribution  $p \in \triangle^{n}$ over data points and \ref{SCMWU-ball} to learn a  normal vector $x \in \ball^d$ using 
the optimized fixed stepsizes  $ \sqrt{\frac{\ln n}{T}}$ and  $ \sqrt{\frac{ \ln 2}{T}}$ 
respectively. Setting \begin{equation*}
	\epsilon
        = 
	{ 2\over \sqrt{T}} \left(\sqrt{\ln n} + \sqrt{ \ln 2} \right),
\end{equation*}
 the time-averaged strategies  $\overline{p} = \frac{1}{T} \sum_{t=1}^T p_t$ and $\overline{x} = \frac{1}{T} \sum_{t=1}^T x_t$    form  an $\epsilon$-approximate Nash equilibrium of the SVM game  and moreover, $\min_{p \in \triangle} p^\top A \overline{x}$ and $\max_{x \in \ball} \overline{p}^\top A x$ are both within $\epsilon$ of the value of the game.  In particular, the margin $m(\bar{x})$ attained by the classifier $\bar{x}$ is within $\epsilon$ of the maximum  margin. 

\end{theorem}

\begin{proof} From the perspective  of the $x$-player, their  loss function (i.e., the negative of their  utility) at time $t$   is the  linear function $-p_t^\top Ax =x^\top m_t$ where  $m_t=-A^\top p_t$. By the assumption that $\|A_{i:}\|_2\le 1 \; \forall \ i \in [n]$ we have that the vectors $m_t$ satisfy 
$\norm{m_t}_2 = \norm{A^\top p_t}_2 \leq 1.$  Thus, if the $x$-player uses \ref{SCMWU-ball} with stepsize $\sqrt{\frac{ \ln 2}{T}}$ to update their strategy, we get from  \eqref{regbound_opt_ball} that
\[
		\frac{1}{T} \left(\max_{x \in \ball} \sumt p_t^\top A x - \sumt p_t^\top A x_t \right) \leq 2\sqrt{\ln 2 \over T}.
\] 

We then get the following lower bound on the time-averaged empirical utility $\ta p_t^\top A x_t$:
	\begin{equation}
	\label{bound:_EmpUtil_lower}
	\begin{split}
		\ta p_t^\top A x_t
		&\geq
		\max_{x \in \ball} \ta p_t^\top A x - 2\sqrt{\ln 2 \over T} \\
		&=
		\max_{x \in \ball} \overline{p}^\top A x  - 2\sqrt{\ln 2 \over T} \\
		&\geq
		\min_{p \in \triangle} \max_{x \in \ball} p^\top A x - 2\sqrt{\ln 2 \over T} \\
		&= v(A) - 2\sqrt{\ln 2 \over T},
	\end{split}
	\end{equation}
	where $v(A) = \max_{x \in \ball} \min_{p \in \triangle} p^\top A x =  \min_{p \in \triangle} \max_{x \in \ball} p^\top A x$ is the value of the game. 
 
 On the other hand, from the perspective of the $p$-player, their loss function at time $t$ is the linear function $p^\top A x_t = p^\top \tilde{m}_t$ where $\tilde{m}_t = Ax_t$. By the assumption that $\|A_{i:}\|_2\le 1 \; \forall \ i \in [n]$ we have, since $x_t \in \ball$, that the loss vectors $\tilde{m}_t$ satisfy
 \[
    \abs{(\tilde{m}_t)_i} = 
	\abs{(Ax_t)_i}
	=
	\abs{ \langle A_{i :},  x_t\rangle }
	\leq
	\norm{A_{i :}}_2 \norm{x_t}_2
	\leq
	1
    \qquad 
    \forall \ t = 1, 2, \ldots, \;
    \forall \ i \in [n].
\]
Thus, if the $p$-player uses \ref{MWU} with stepsize $\sqrt{\frac{\ln n}{T}}$ to update their strategy, we get from the standard regret bound for \ref{MWU} (see, e.g., \cite{ART:AHK12}) that 
	\[
		\frac{1}{T} \left( \sumt p_t^\top A x_t - \min_{p \in \triangle} \sumt p^\top A x_t \right) \leq 2 \sqrt{\frac{\ln n}{T}}.
	\]
	We then get the following upper bound on the time-averaged empirical utility:
	\begin{equation}
	\label{bound:_EmpUtil_upper}
	\begin{split}
		\ta p_t^\top A x_t
		&\leq
		\min_{p \in \triangle} \ta p^\top A x_t + 2 \sqrt{\frac{\ln n}{T}} \\
		&=
		\min_{p \in \triangle} p^\top A \overline{x} + 2 \sqrt{\frac{\ln n}{T}} \\
		&\leq
		\max_{x \in \ball} \min_{p \in \triangle} p^\top A x + 2 \sqrt{\frac{\ln n}{T}} \\
		&=
		v(A) + 2 \sqrt{\frac{\ln n}{T}}.
	\end{split}
	\end{equation}
	Combining \eqref{bound:_EmpUtil_lower} and \eqref{bound:_EmpUtil_upper} gives us the following chain of inequalities
	\begin{equation}
	\label{zerosum_chain}
		v(A) - 2\sqrt{\ln 2 \over T}
		\leq
		\max_{x \in \ball} \overline{p}^\top A x - 2\sqrt{\ln 2 \over T}
		\leq
		\ta p_t^\top A x_t
		\leq
		\min_{p \in \triangle} p^\top A \overline{x} + 2 \sqrt{\frac{\ln n}{T}}
		\leq
		v(A) + 2 \sqrt{\frac{\ln n}{T}},
	\end{equation}	
	from which we conclude that
	\[
		\max_{x \in \ball} \overline{p}^\top A x \leq v(A) + \frac{2(\sqrt{\ln n} + \sqrt{\ln 2})}{\sqrt{T}}
		\qquad
		\text{ and }
		\qquad
		\min_{p \in \triangle} p^\top A \overline{x} \geq v(A) - \frac{2(\sqrt{\ln n} + \sqrt{\ln 2})}{\sqrt{T}},
	\]
	i.e., that $\max_{x \in \ball} \overline{p}^\top A x$ and $\min_{p \in \triangle} p^\top A \overline{x}$ are within $ \frac{2(\sqrt{\ln n} + \sqrt{\ln 2})}{\sqrt{T}} $ of the value of the game. (It is always true that $\max_{x \in \ball} \overline{p}^\top A x \geq \min_{p \in \triangle} \max_{x \in \ball} \overline{p}^\top A x = v(A)$ and $\min_{p \in \triangle} p^\top A \overline{x} \leq \max_{x \in \ball} \min_{p \in \triangle} p^\top A \overline{x} = v(A)$.) Finally, since
	\[
		\min_{p \in \triangle} p^\top A \overline{x}
		\leq
		\overline{p}^\top A \overline{x}
		\leq
		\max_{x \in \ball} \overline{p}^\top A x,
	\]
	we also get from \eqref{zerosum_chain} that
	\[
		\max_{x \in \ball} \overline{p}^\top A x - \frac{2(\sqrt{\ln n} + \sqrt{\ln 2})}{\sqrt{T}}
		\leq
		\overline{p}^\top A \overline{x}
		\leq
		\min_{p \in \triangle} p^\top A \overline{x} + \frac{2(\sqrt{\ln n} + \sqrt{\ln 2})}{\sqrt{T}},
	\]
	i.e., that $(\overline{p}, \overline{x})$ is an $\frac{2(\sqrt{\ln n} + \sqrt{\ln 2})}{\sqrt{T}}$-approximate Nash equilibrium.
\end{proof}

In other words, for any $\epsilon > 0$, if we run (optimized) fixed-stepsize \ref{SCMWU-ball} against MWU for a time horizon of at least
\begin{equation}
\label{SVM_T}
	T = 4 \left(\sqrt{\ln n} + \sqrt{2 \ln 2} \right)^2 \epsilon^{-2} 
\end{equation}
then we are guaranteed that the time-averaged strategies form an $\epsilon$-Nash equilibrium and that $\min_{p \in \triangle} p^\top A \overline{x}$ and $\max_{x \in \ball} \overline{p}^\top A x$ are both within $\epsilon$ of the value of the game. In particular, if the data is linearly separable 
then
the margin attained by $\overline{x}$ is within $\epsilon$ of the maximum margin (i.e., the margin attained by the best possible classifier).

    \paragraph{Experiments (SVM).}
    
    We ran \ref{SCMWU-ball} against \ref{MWU} in the SVM game  using data of varying dimensions ($d=2, 3, 5, 10$) and different time horizons ($T = 10^2, 10^3$). Our analysis focused on the margins achieved by the classifier $\overline{x}$. We performed these experiments on 100 randomized instances, each consisting of $n=10^3$ linearly separable data points in $\mathbb{R}^d$ with a Euclidean norm of at most 1 and a separation margin of at least  0.1. 

Table \ref{table:SVMratio} presents the mean and worst-case ratios of the margin achieved by the learned classifier compared to the \emph{generated margin}, which is the margin of the data points along the direction of separation that we used to generate the data (and thus an approximation to the maximum margin). The results demonstrate that the learned classifier approaches very close to the generated margin in multiplicative error even when allowed to learn for only a small time horizon. For instance, when $d=10$ and the time horizon is $T = 10^2$, the learned classifier achieves a margin that is, on average, 0.908 times the margin of the generated data.

Furthermore, Table \ref{table:SVMerror} reports the mean and worst-case additive errors ($\hat{\epsilon}$) in the margin achieved by the learned classifier compared to the generated margin. In all cases, we observe that the margin of the learned classifier has an additive error $\hat{\epsilon}$ well within 0.02 of the generated margin. Our experiments thus validate the theoretical results and suggest that convergence to a maximum-margin classifier may occur much faster in practice than the worst-case theoretical guarantee, since to guarantee an additive error of $<0.02$ from the maximum margin the theoretically required time horizon according to equation \eqref{SVM_T} is $T = \Big\lceil \left( \sqrt{\ln 10^3} + \sqrt{2 \ln 2}\right)^2 (0.02)^{-2} \Big\rceil = 144832 >> 10^2$.

For visualization purposes, we showcase in Figures \ref{fig:_SVM_viz_T100}, \ref{fig:_SVM_viz_T1000}, and \ref{fig:_SVM_viz_T10000} how the classifier obtained from running \ref{SCMWU-ball} against \ref{MWU} in the SVM game (over the same 10 instances of $n = 10^3$ data points in $\R^2$) converges closer to the optimal classifier as the time horizon increases from $10^2$, to $10^3$, and finally to $10^4$ respectively.

\begin{table*}[ht]
\centering
\begin{NiceTabular}{|c|c|c|c|c|}
\hline
\RowStyle[bold]{}
    $T \; \backslash \; d$
    & 2
    & 3
    & 5
    & 10
    \\
\hline
    $\mathbf{10^2}$
    & 0.983 (0.929)
    & 0.967 (0.924)
    & 0.954 (0.909)
    & 0.908 (0.859)
    \\
\hline
    $\mathbf{10^3}$
    & 0.993 (0.964)
    & 0.996 (0.977)
    & 0.990 (0.979)
    & 0.980 (0.963)
    \\
\hline
\end{NiceTabular}
\caption{Mean ratio (and worst-case ratio in parentheses) of learnt classifier's margin to generated margin. Classifier learnt by running \ref{SCMWU-ball} against \ref{MWU} in the SVM game on $n=10^3$ data points in $\R^d$, for different dimensions $d$ and different time horizons $T$ (100 instances each).}
\label{table:SVMratio}
\end{table*}

\begin{table*}[ht]
\centering
\begin{NiceTabular}{|c|c|c|c|c|}
\hline
\RowStyle[bold]{}
    $T \; \backslash \; d$
    & 2
    & 3
    & 5
    & 10
    \\
\hline
    $\mathbf{10^2}$
    & $1.73 \scnexp{3}$ ($7.13 \scnexp{3}$)
    & $3.32 \scnexp{3}$ ($7.62 \scnexp{3}$)
    & $4.64 \scnexp{3}$ ($9.06 \scnexp{3}$)
    & $9.23 \scnexp{3}$ ($1.41 \scnexp{2}$)
    \\
\hline
    $\mathbf{10^3}$
    & $7.37 \scnexp{4}$ ($3.58 \scnexp{3}$)
    & $4.29 \scnexp{4}$ ($2.31 \scnexp{3}$)
    & $1.01 \scnexp{3}$ ($2.06 \scnexp{3}$)
    & $1.99 \scnexp{3}$ ($3.69 \scnexp{3}$)
    \\
\hline
\end{NiceTabular}
\caption{Mean additive error $\hat{\epsilon}$ (and worst-case error in parentheses) of learnt classifier's margin from generated margin. Classifier learnt by running \ref{SCMWU-ball} against \ref{MWU} in the SVM game on $n=10^3$ data points in $\R^d$, for different dimensions $d$ and different time horizons $T$ (100 instances each).}
\label{table:SVMerror}
\end{table*}

\begin{figure}[H]
\centering
    \begin{subfigure}{.19\linewidth}
        \centering
        \includegraphics[width=\linewidth]{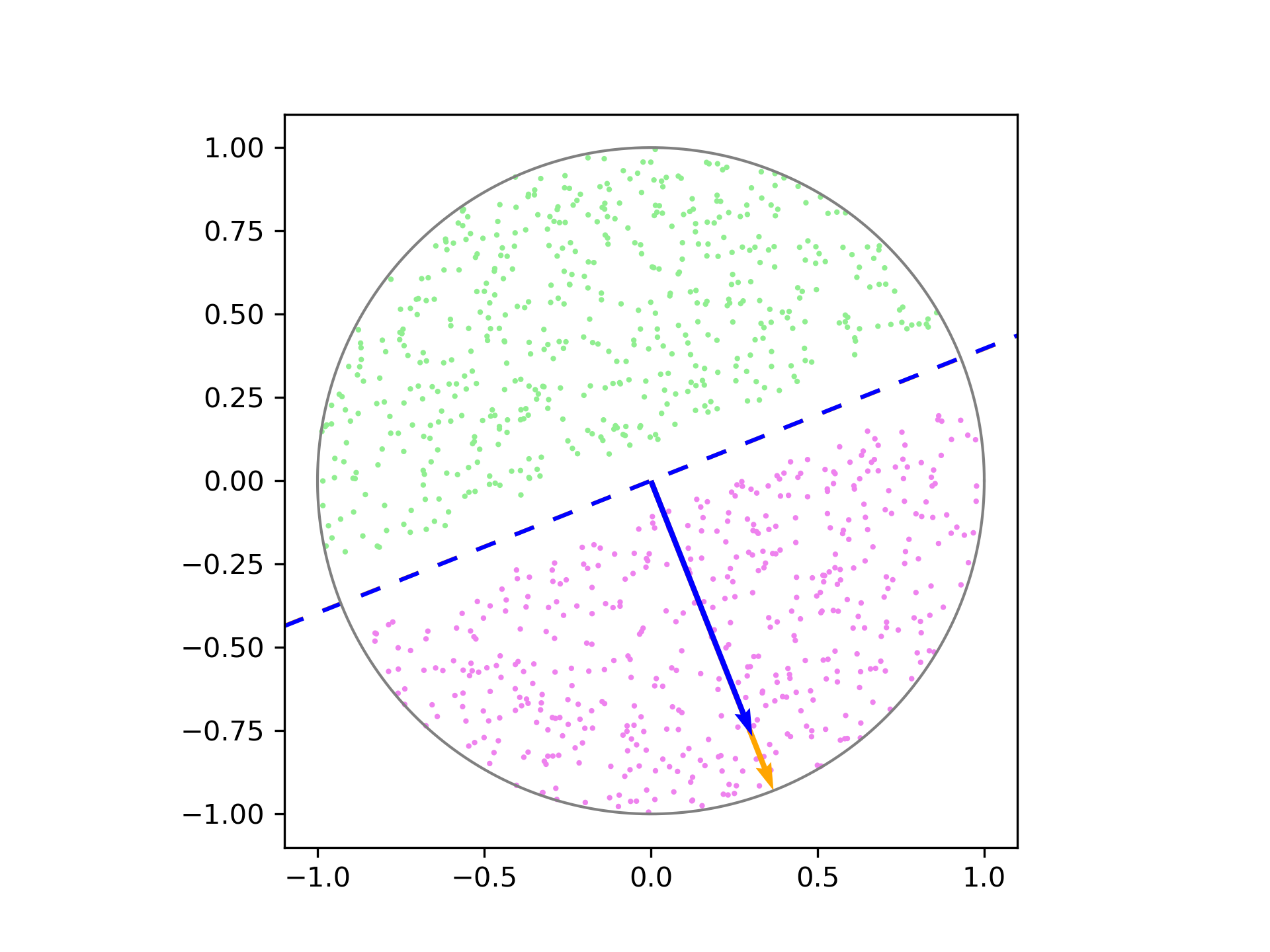}
    \end{subfigure}
    \begin{subfigure}{.19\linewidth}
        \centering
        \includegraphics[width=\linewidth]{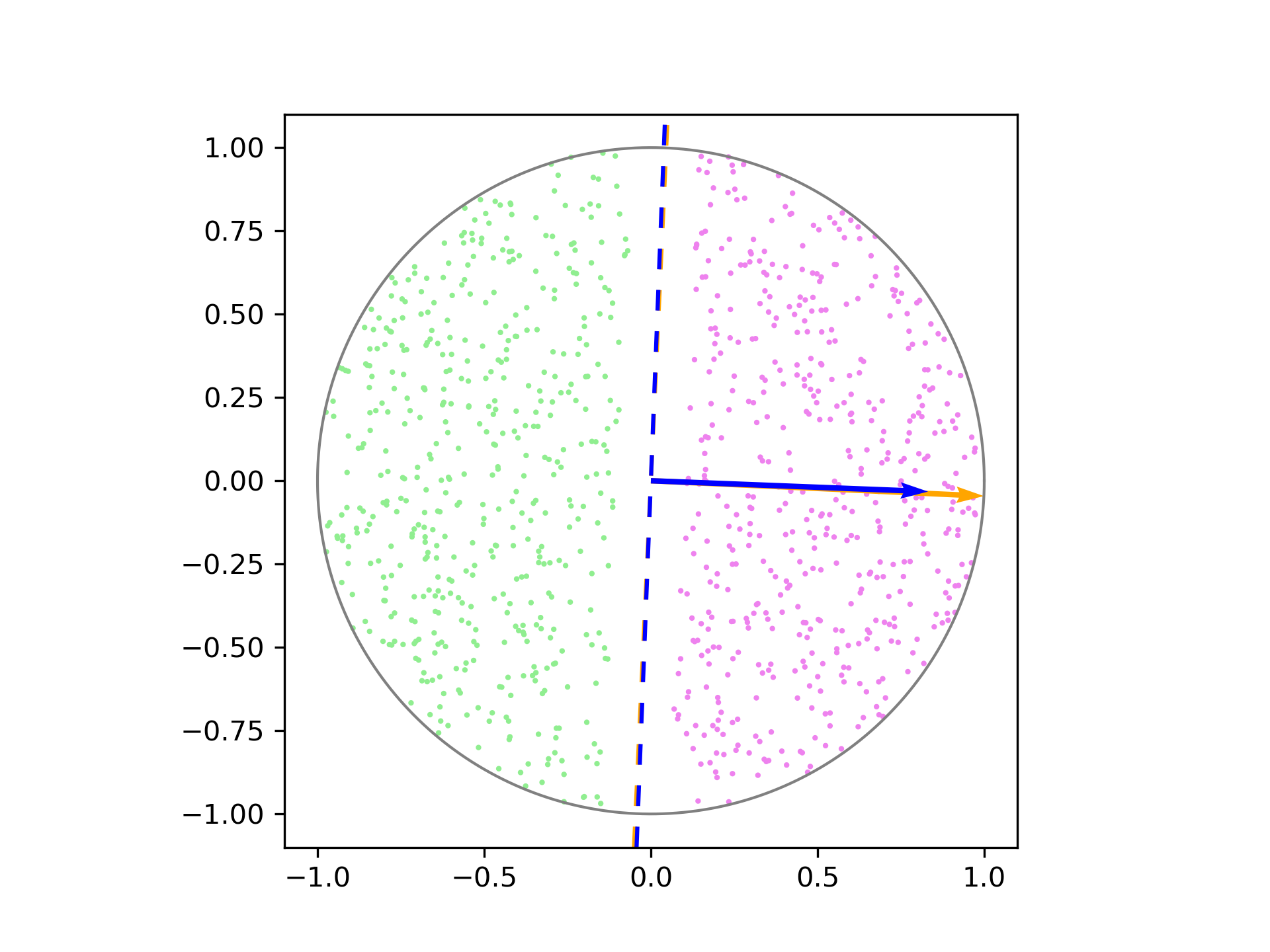}
    \end{subfigure}
    \begin{subfigure}{.19\linewidth}
        \centering
        \includegraphics[width=\linewidth]{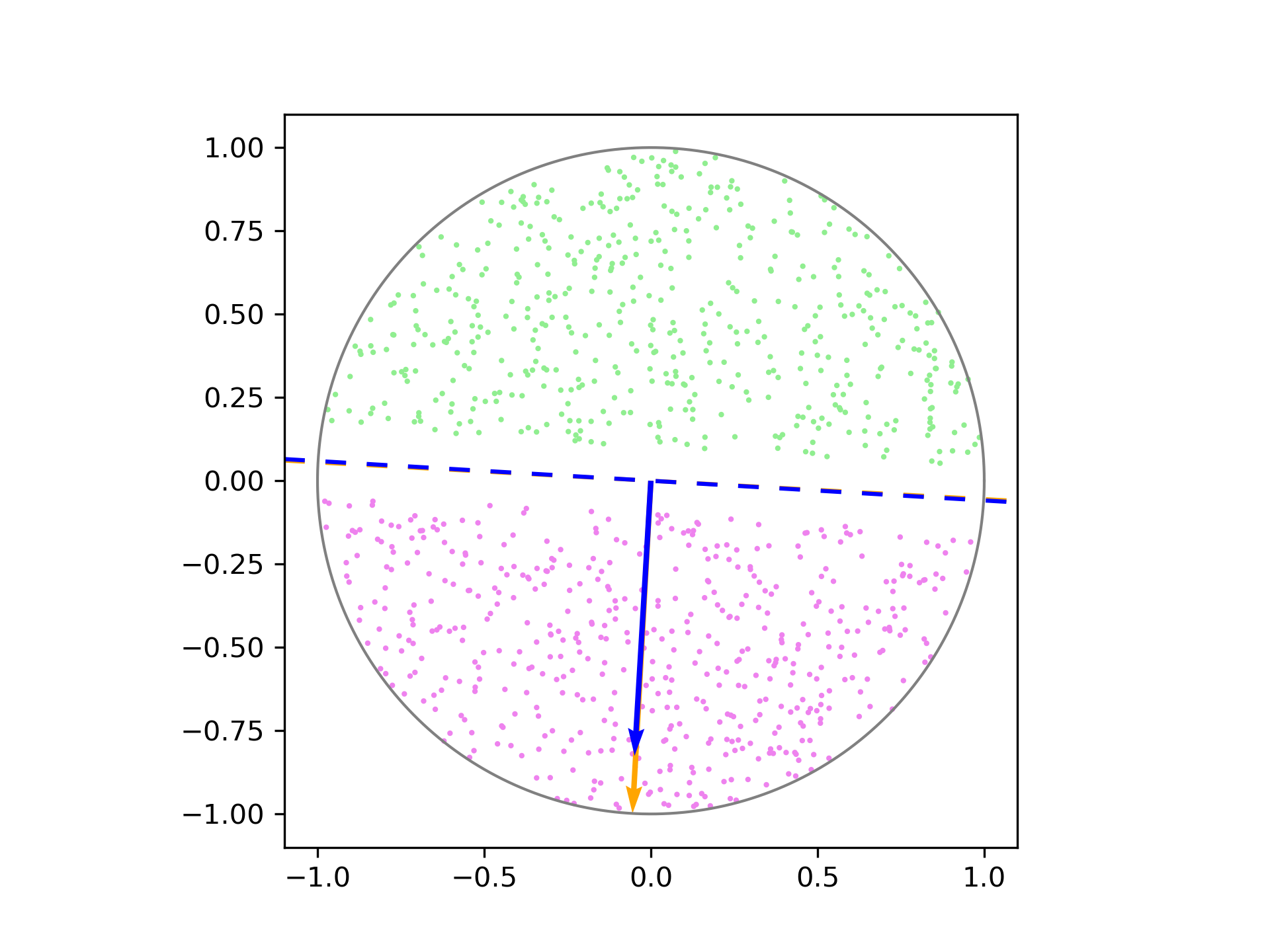}
    \end{subfigure}
    \begin{subfigure}{.19\linewidth}
        \centering
        \includegraphics[width=\linewidth]{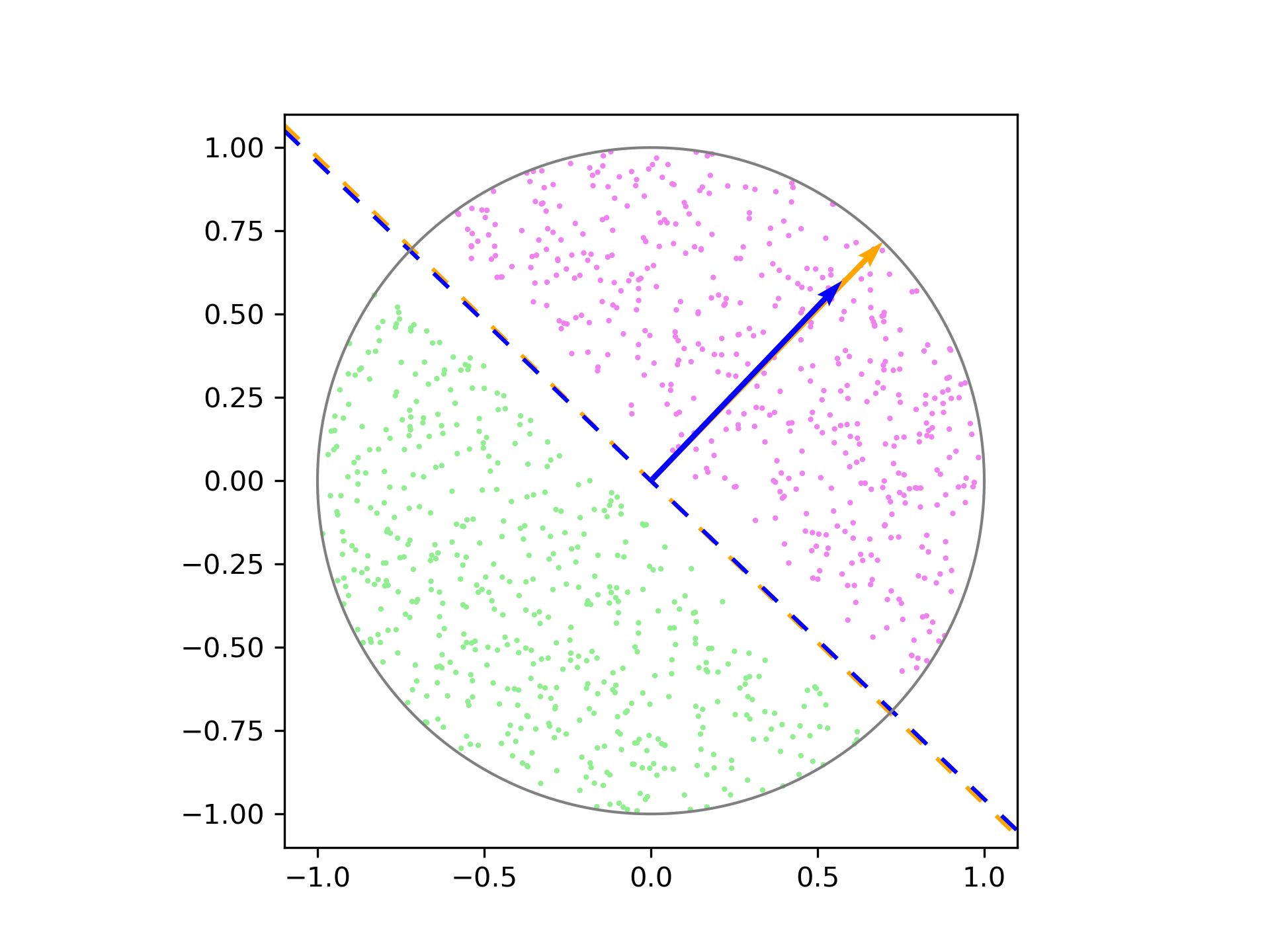}
    \end{subfigure}
    \begin{subfigure}{.19\linewidth}
        \centering
        \includegraphics[width=\linewidth]{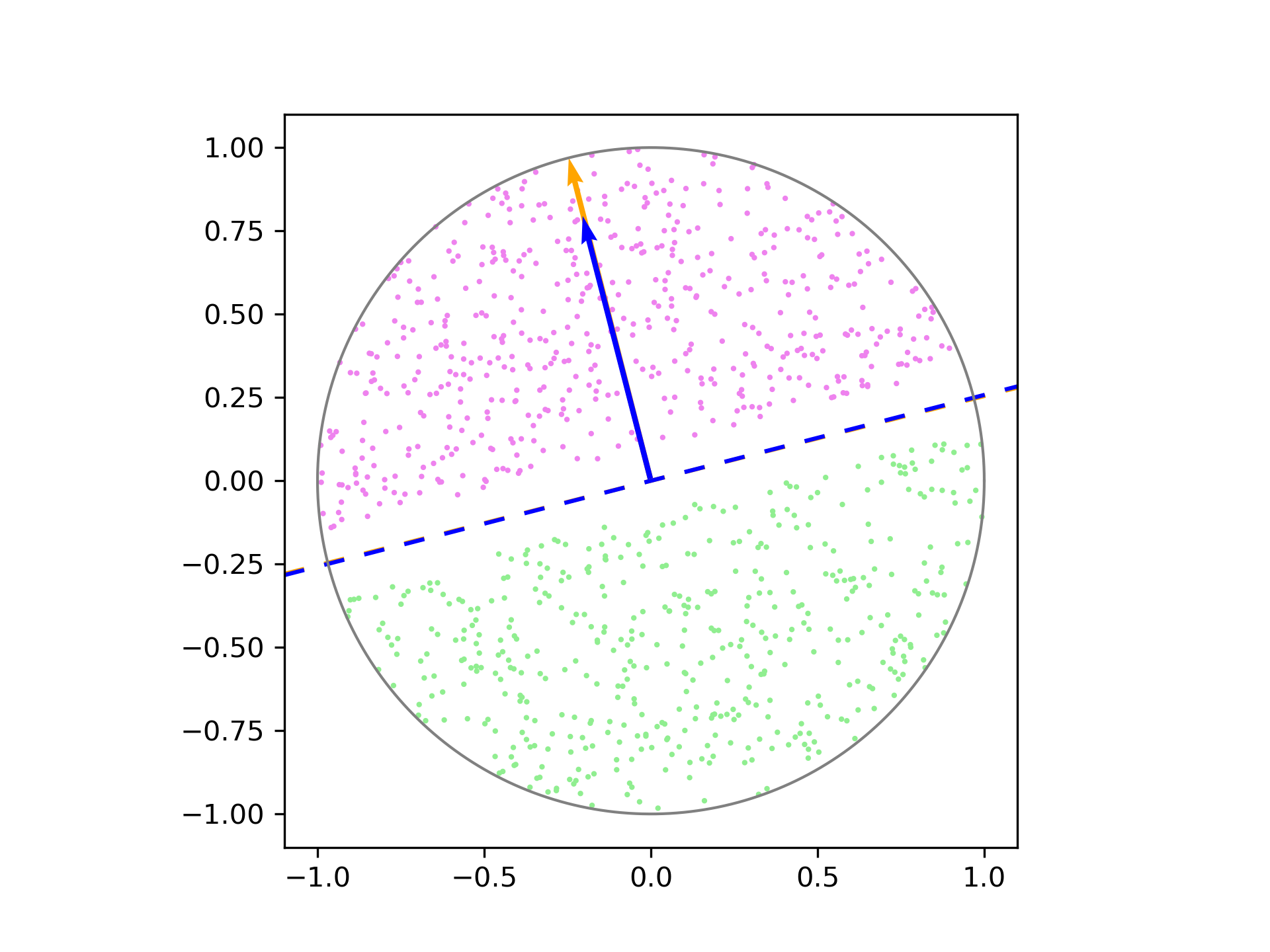}
    \end{subfigure}
    
    \begin{subfigure}{.19\linewidth}
        \centering
        \includegraphics[width=\linewidth]{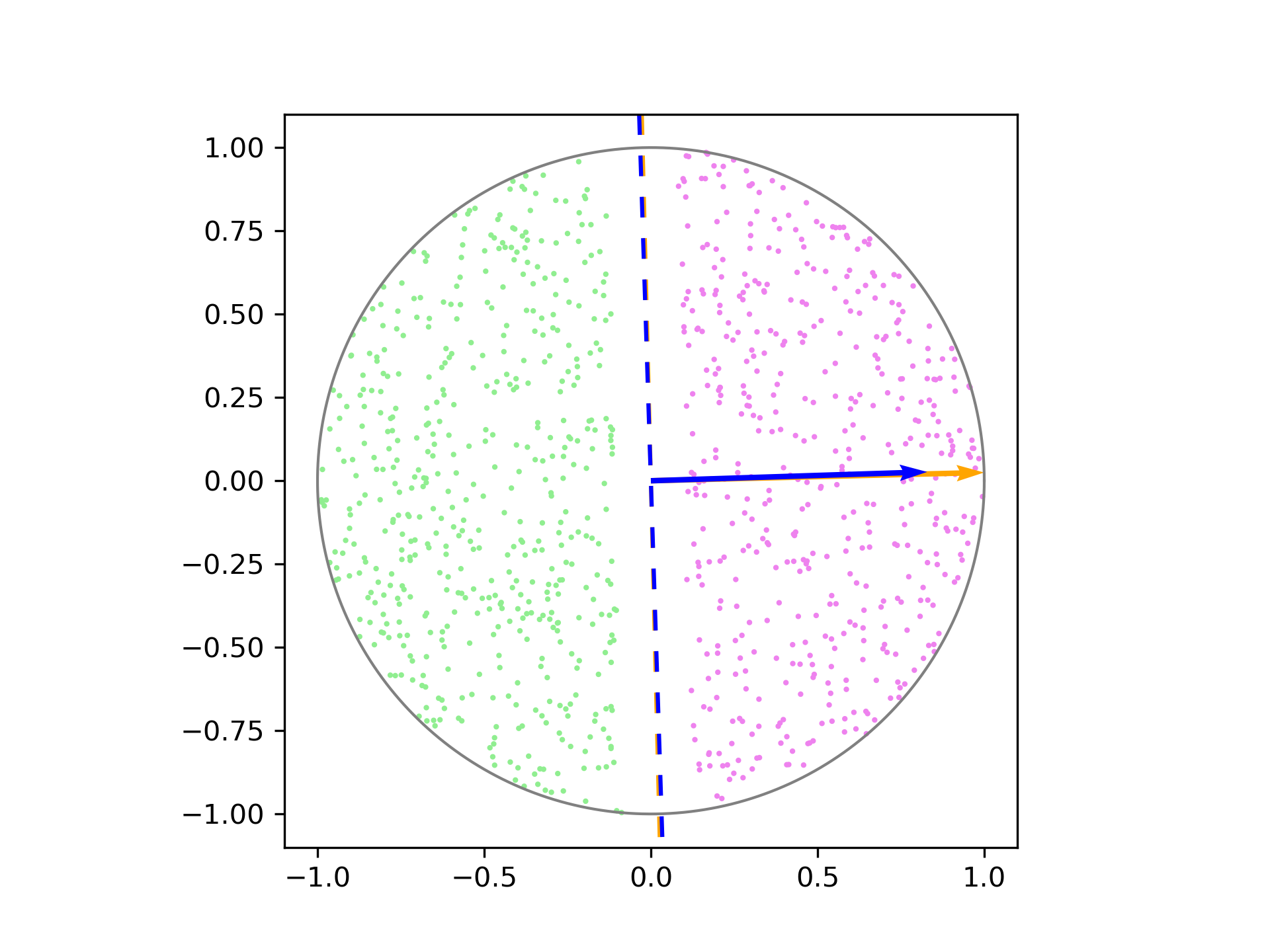}
    \end{subfigure}
    \begin{subfigure}{.19\linewidth}
        \centering
        \includegraphics[width=\linewidth]{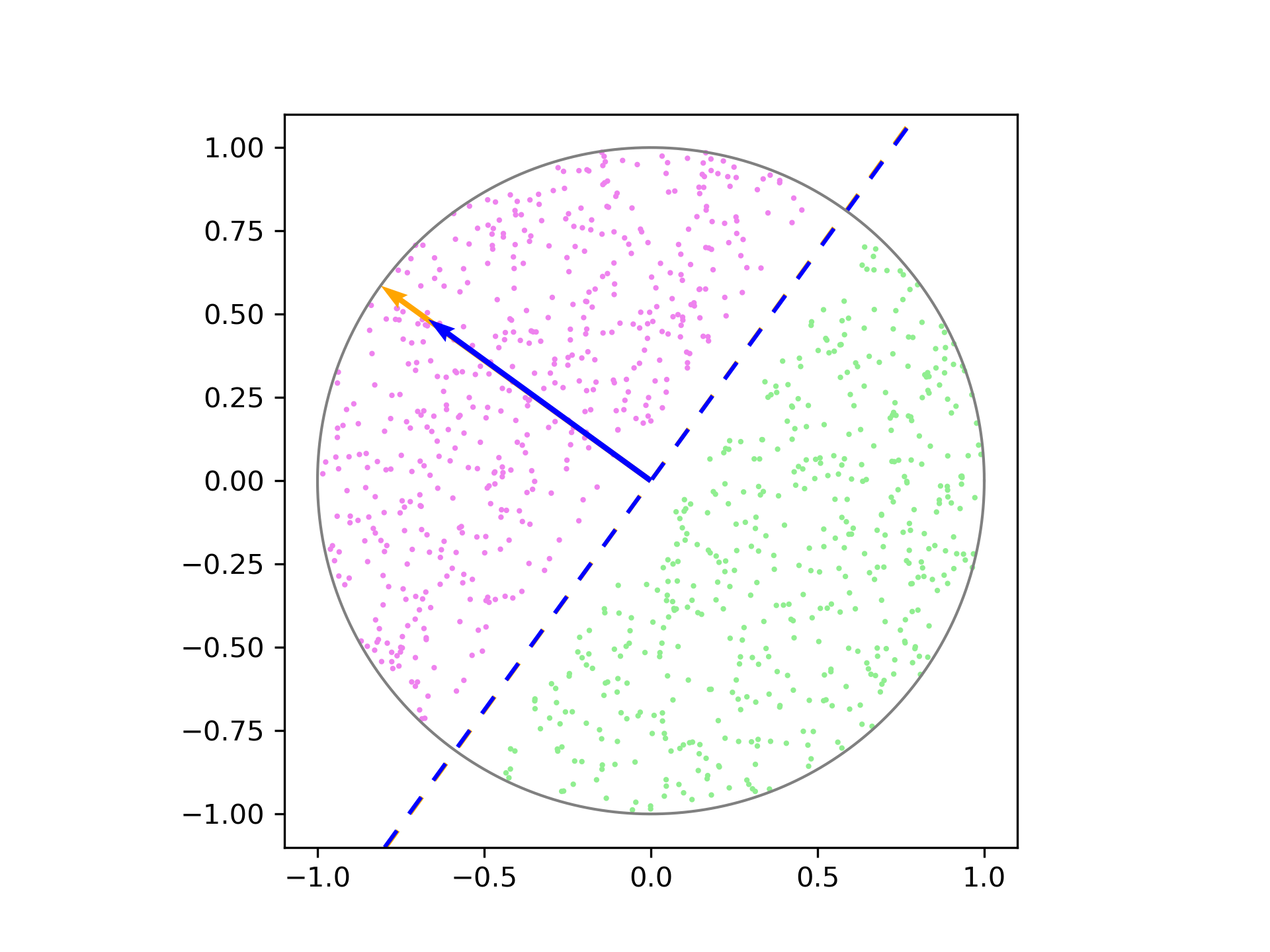}
    \end{subfigure}
    \begin{subfigure}{.19\linewidth}
        \centering
        \includegraphics[width=\linewidth]{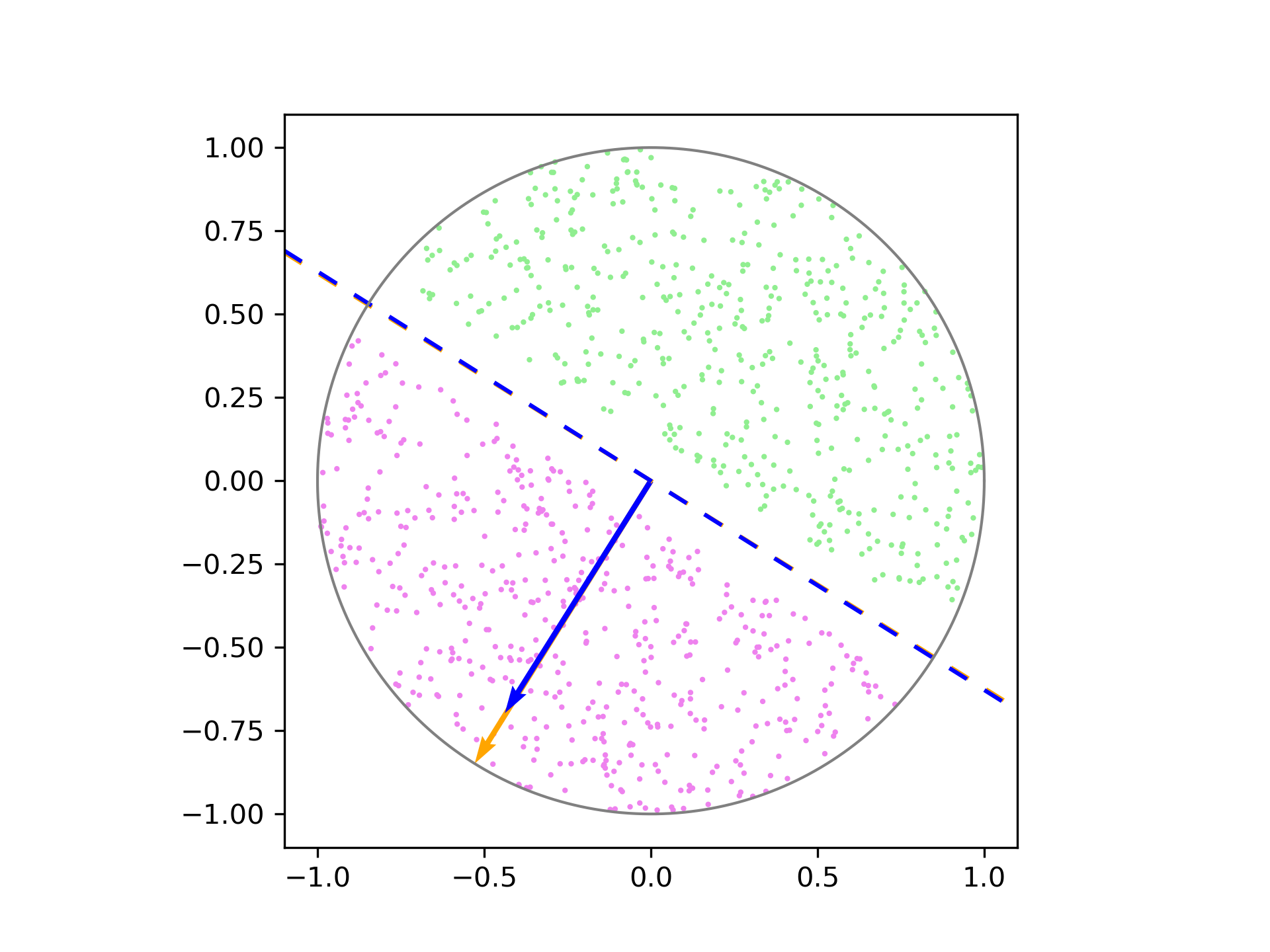}
    \end{subfigure}
    \begin{subfigure}{.19\linewidth}
        \centering
        \includegraphics[width=\linewidth]{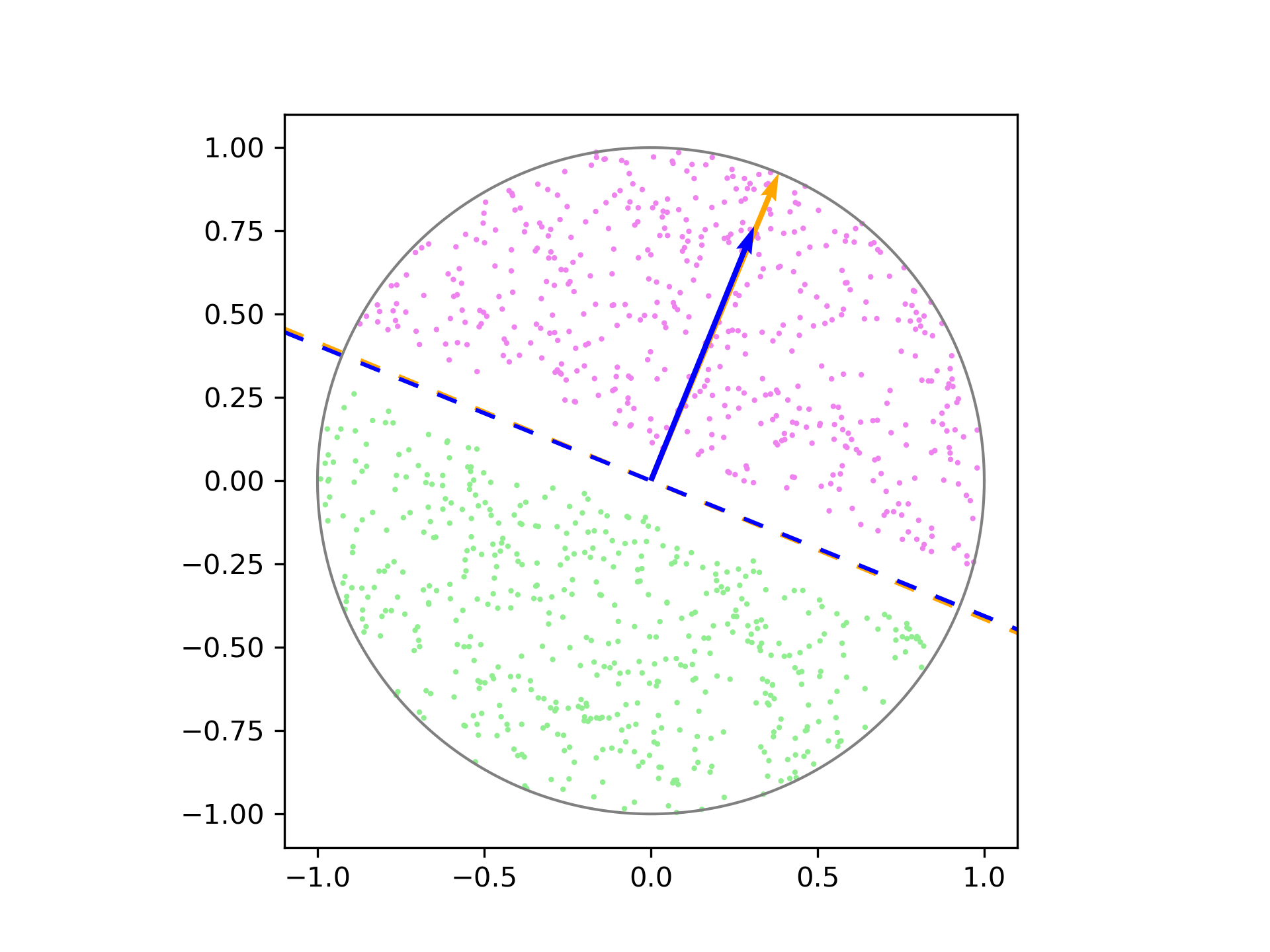}
    \end{subfigure}
    \begin{subfigure}{.19\linewidth}
        \centering
        \includegraphics[width=\linewidth]{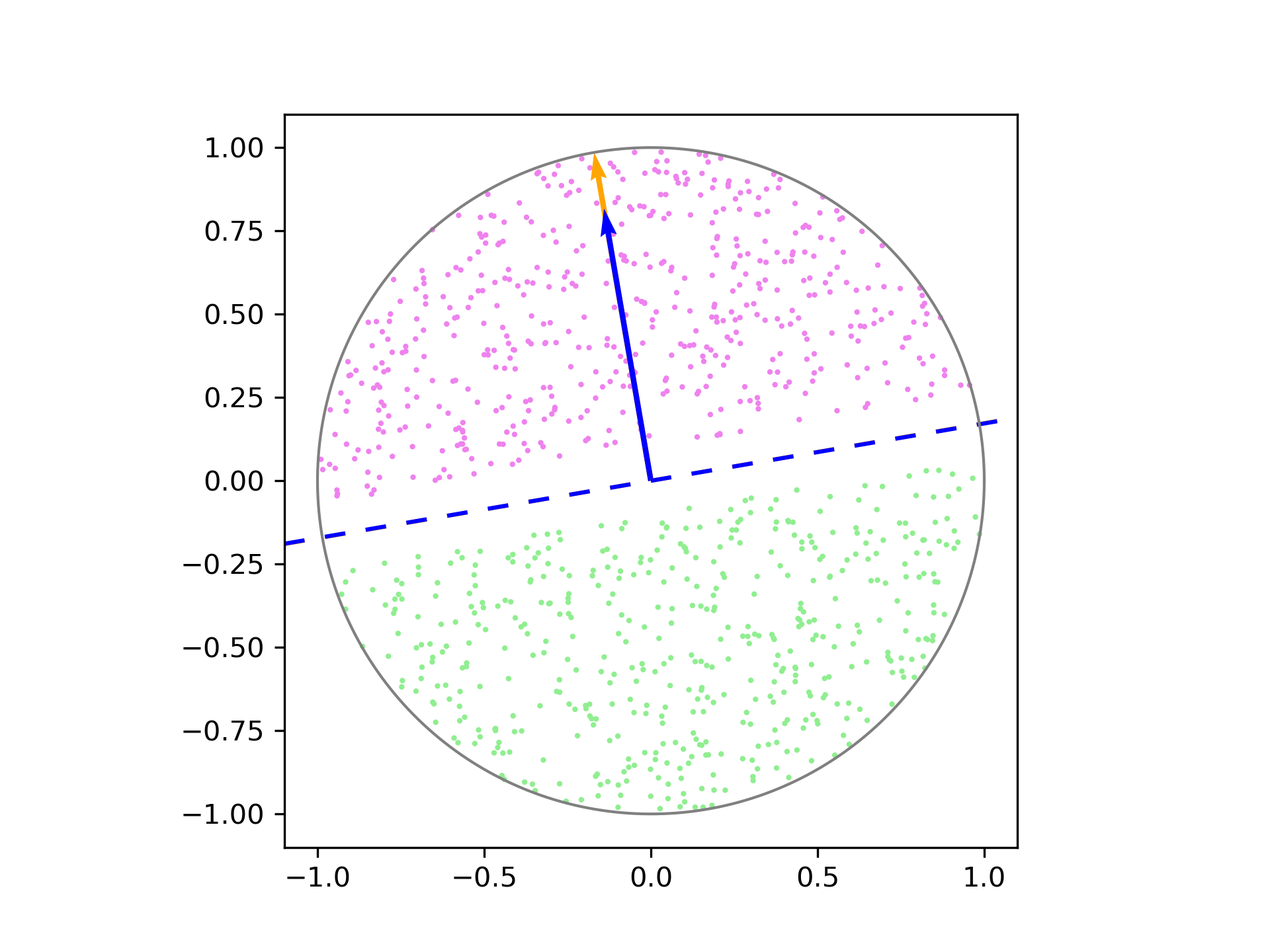}
    \end{subfigure}
    \caption{
    	Classifiers achieved by running \ref{SCMWU-ball} against \ref{MWU} in the SVM game (on $n = 10^3$ data points in $\R^2$) for time horizon $T = 10^2$. ({\textcolor{blue} {Blue}}: classifier $\overline{x}$ achieved by running \ref{SCMWU-ball} vs. \ref{MWU} in SVM game. {\textcolor{orange} {Orange}}: classifier $w$ used to generate data.)
	} 
    \label{fig:_SVM_viz_T100}
\end{figure}

\begin{figure}[H]
\centering
    \begin{subfigure}{.19\linewidth}
        \centering
        \includegraphics[width=\linewidth]{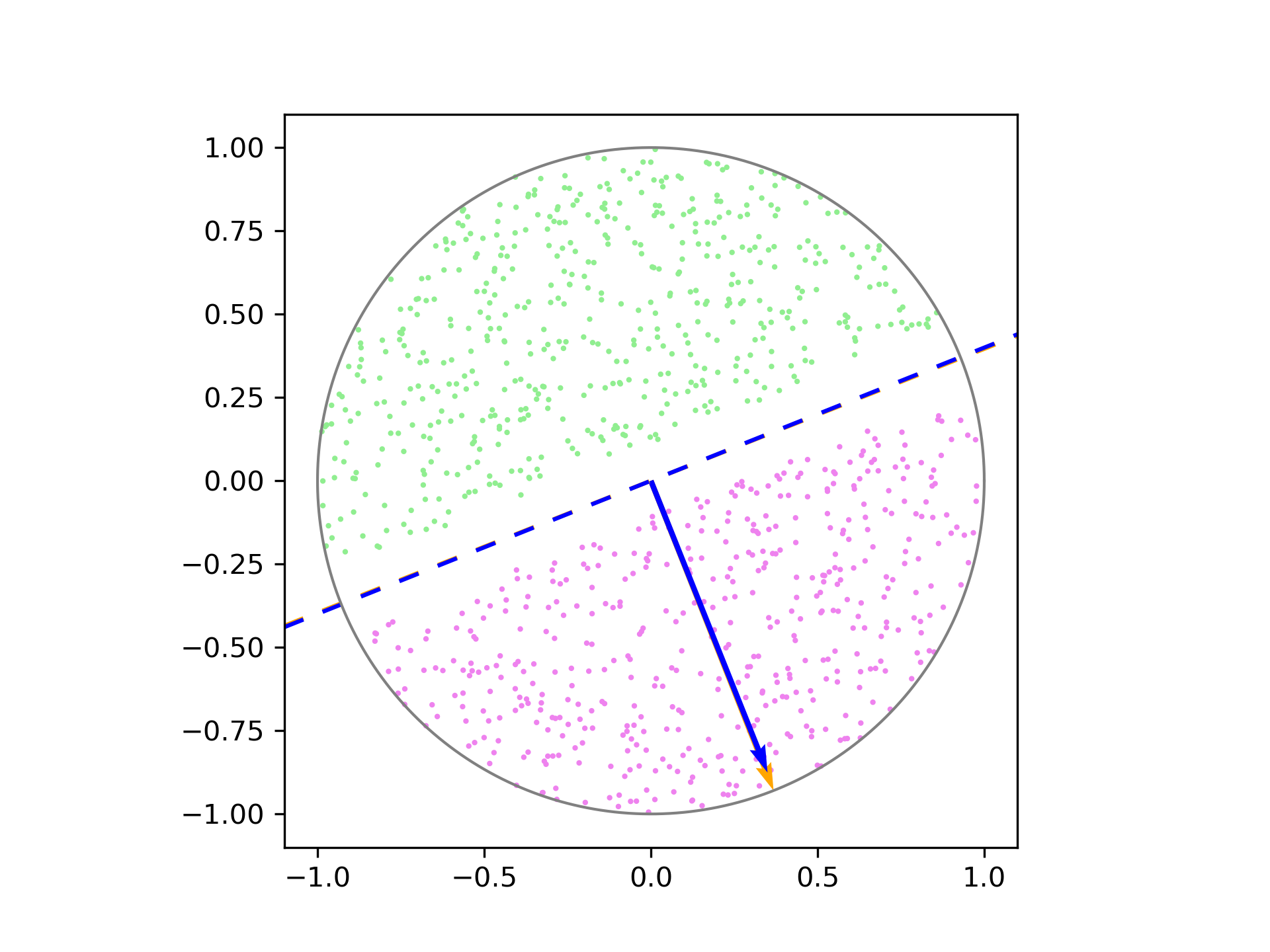}
    \end{subfigure}
    \begin{subfigure}{.19\linewidth}
        \centering
        \includegraphics[width=\linewidth]{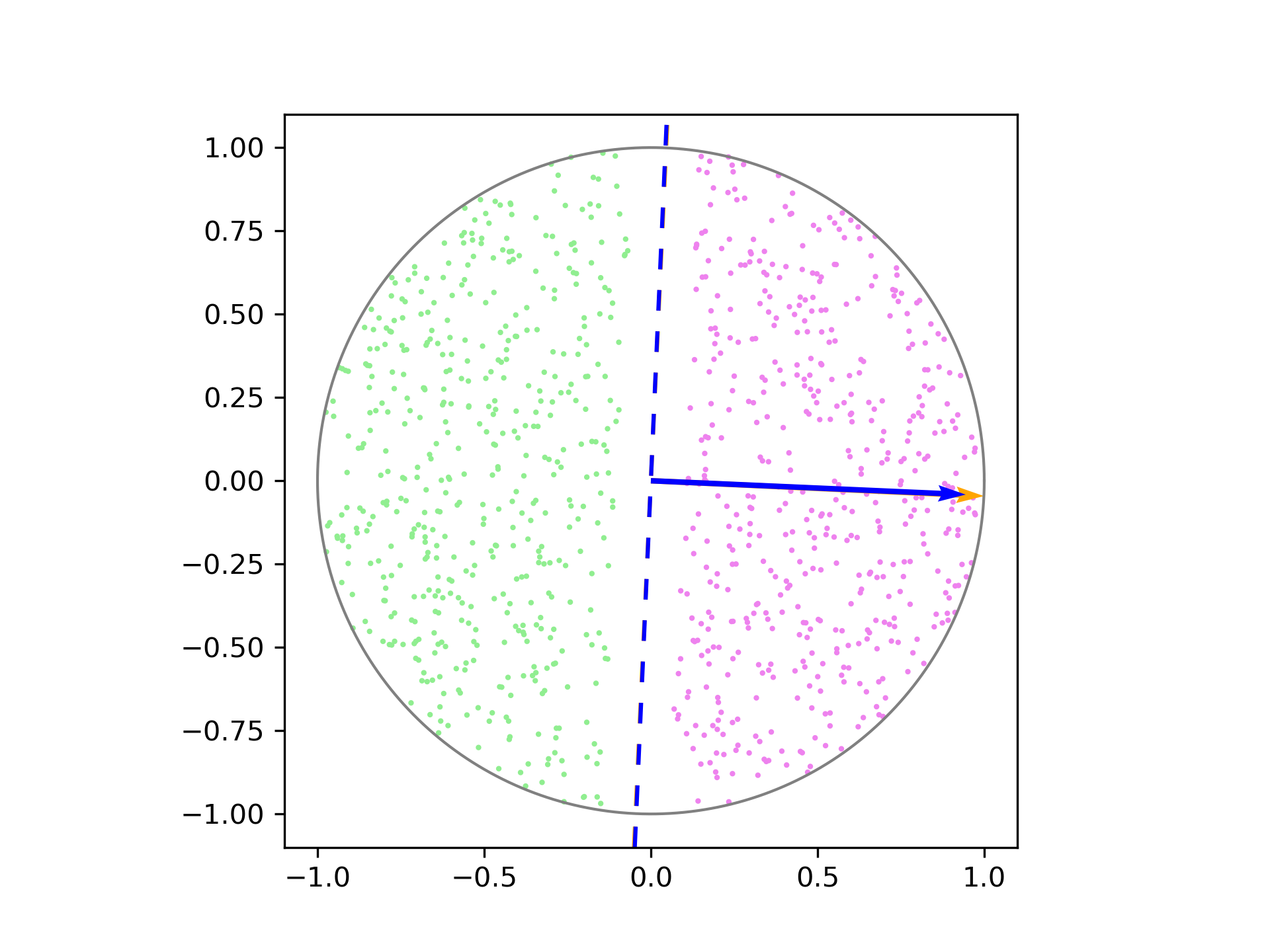}
    \end{subfigure}
    \begin{subfigure}{.19\linewidth}
        \centering
        \includegraphics[width=\linewidth]{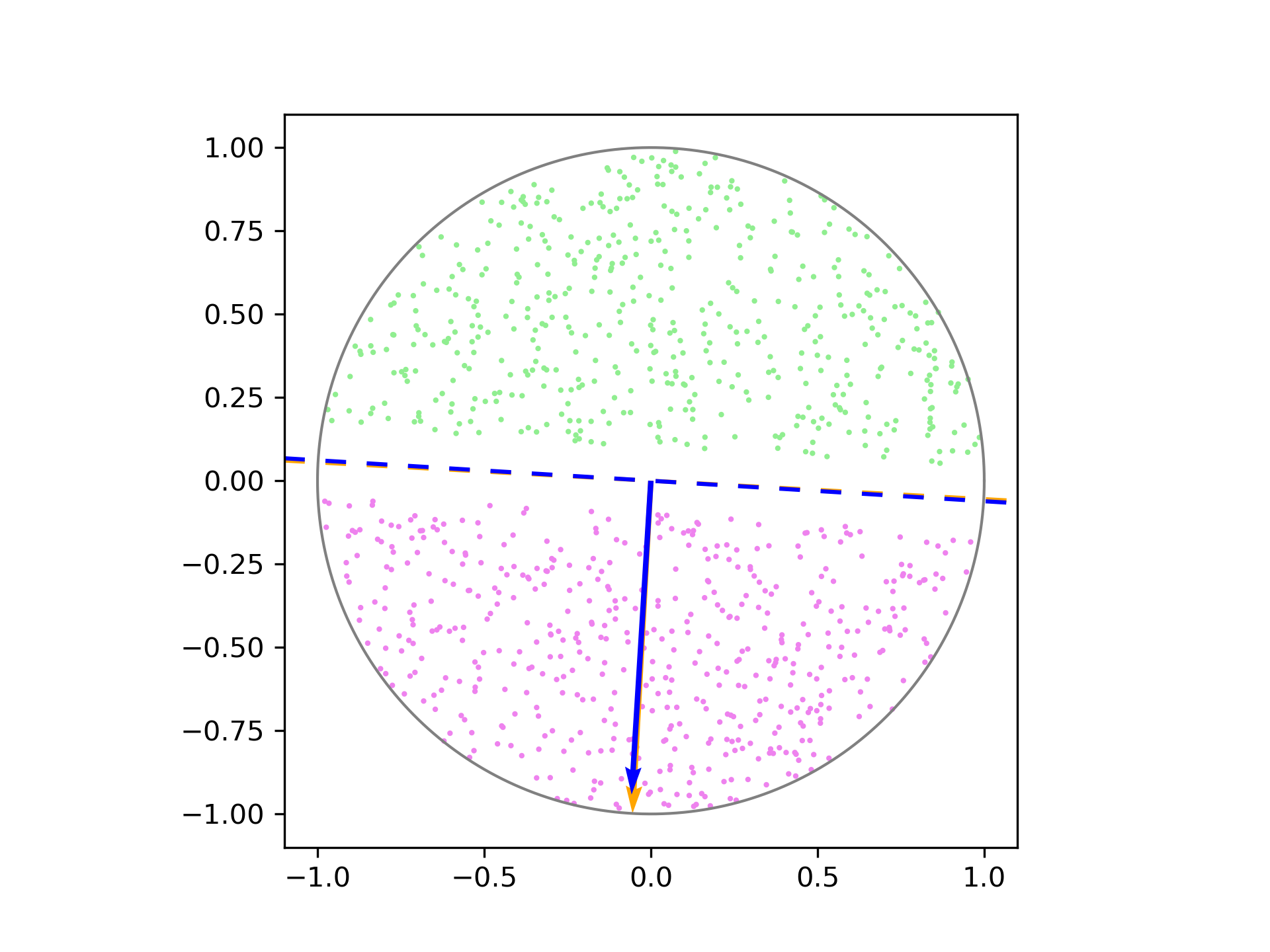}
    \end{subfigure}
    \begin{subfigure}{.19\linewidth}
        \centering
        \includegraphics[width=\linewidth]{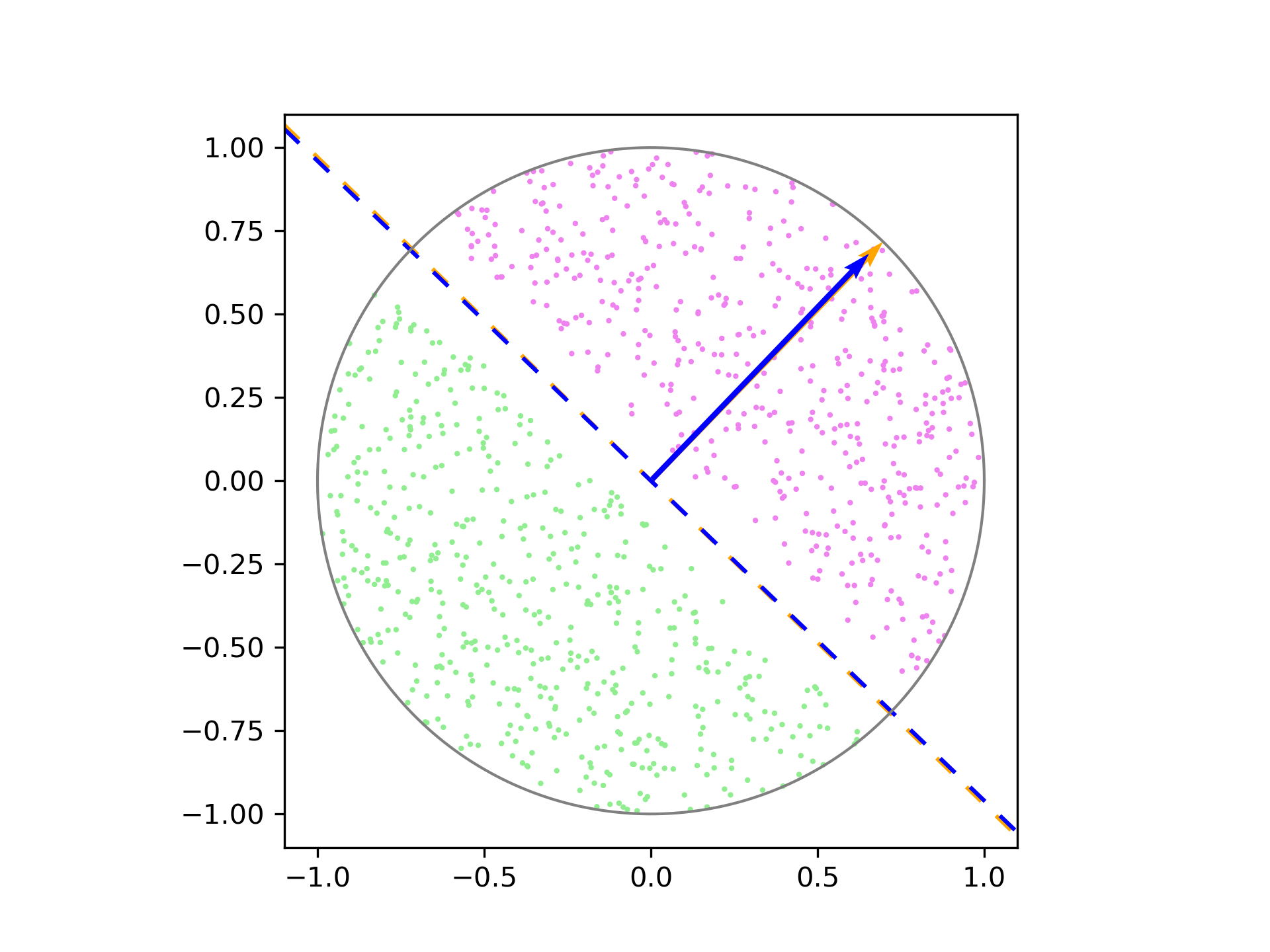}
    \end{subfigure}
    \begin{subfigure}{.19\linewidth}
        \centering
        \includegraphics[width=\linewidth]{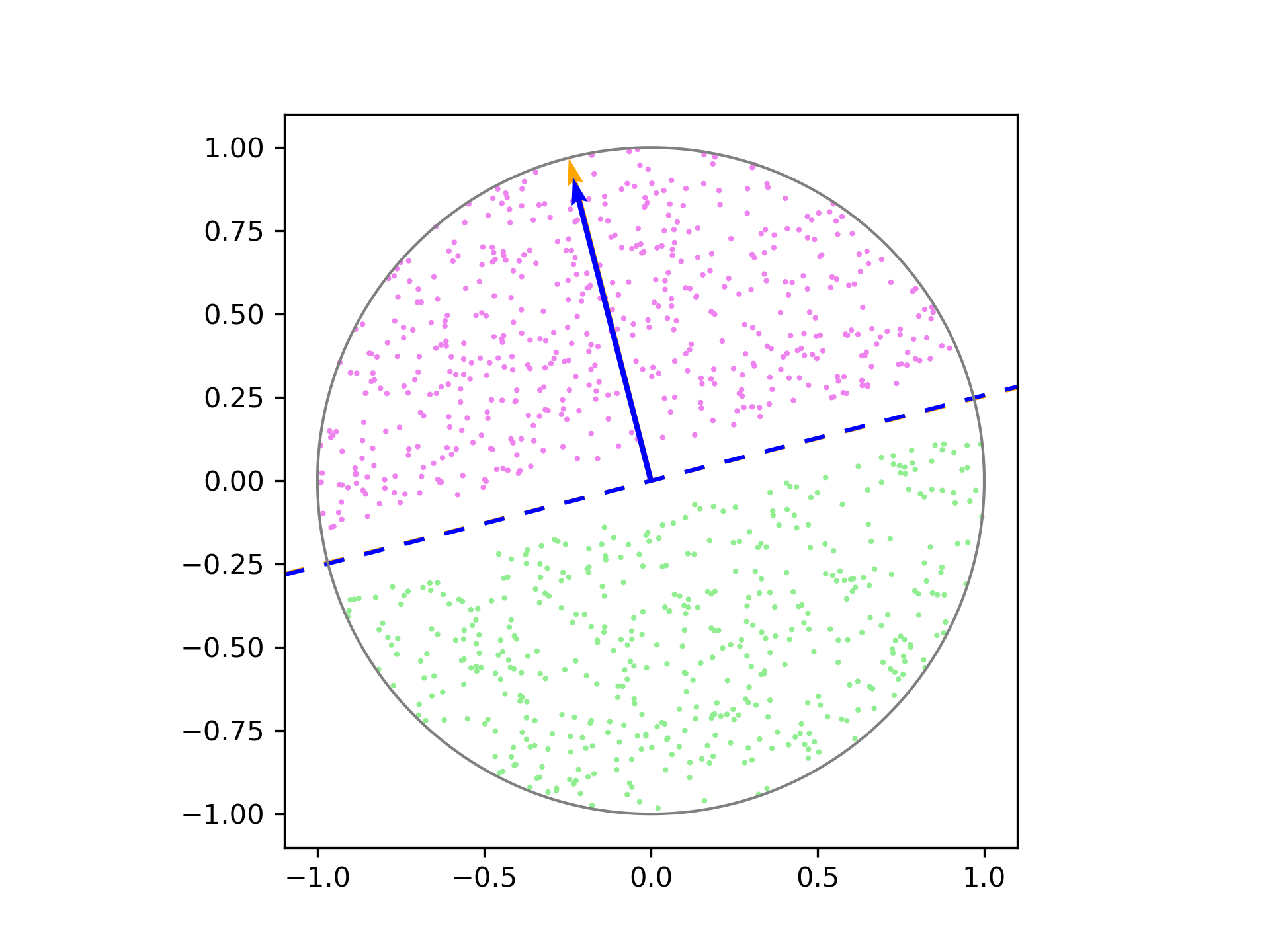}
    \end{subfigure}
    
    \begin{subfigure}{.19\linewidth}
        \centering
        \includegraphics[width=\linewidth]{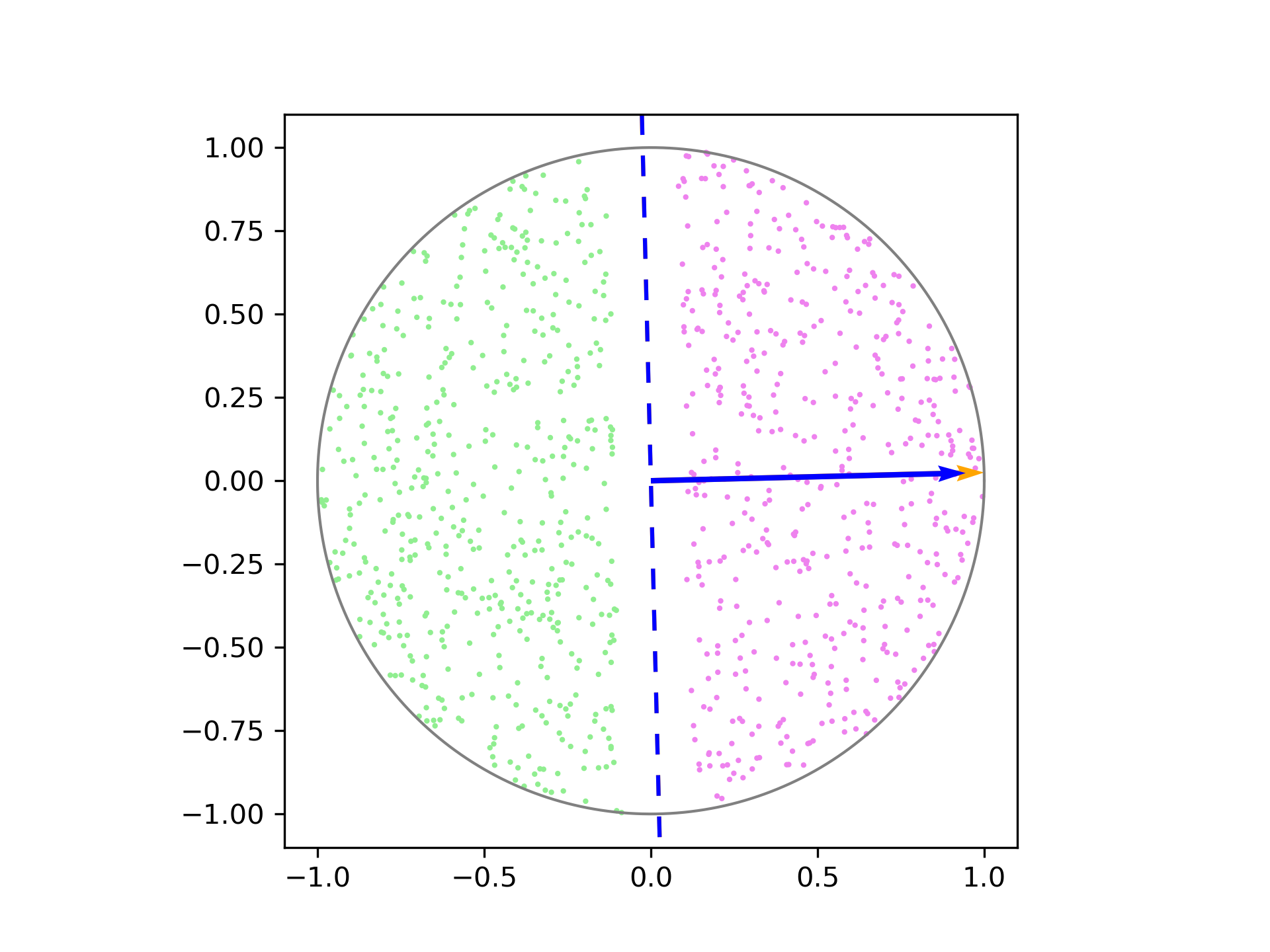}
    \end{subfigure}
    \begin{subfigure}{.19\linewidth}
        \centering
        \includegraphics[width=\linewidth]{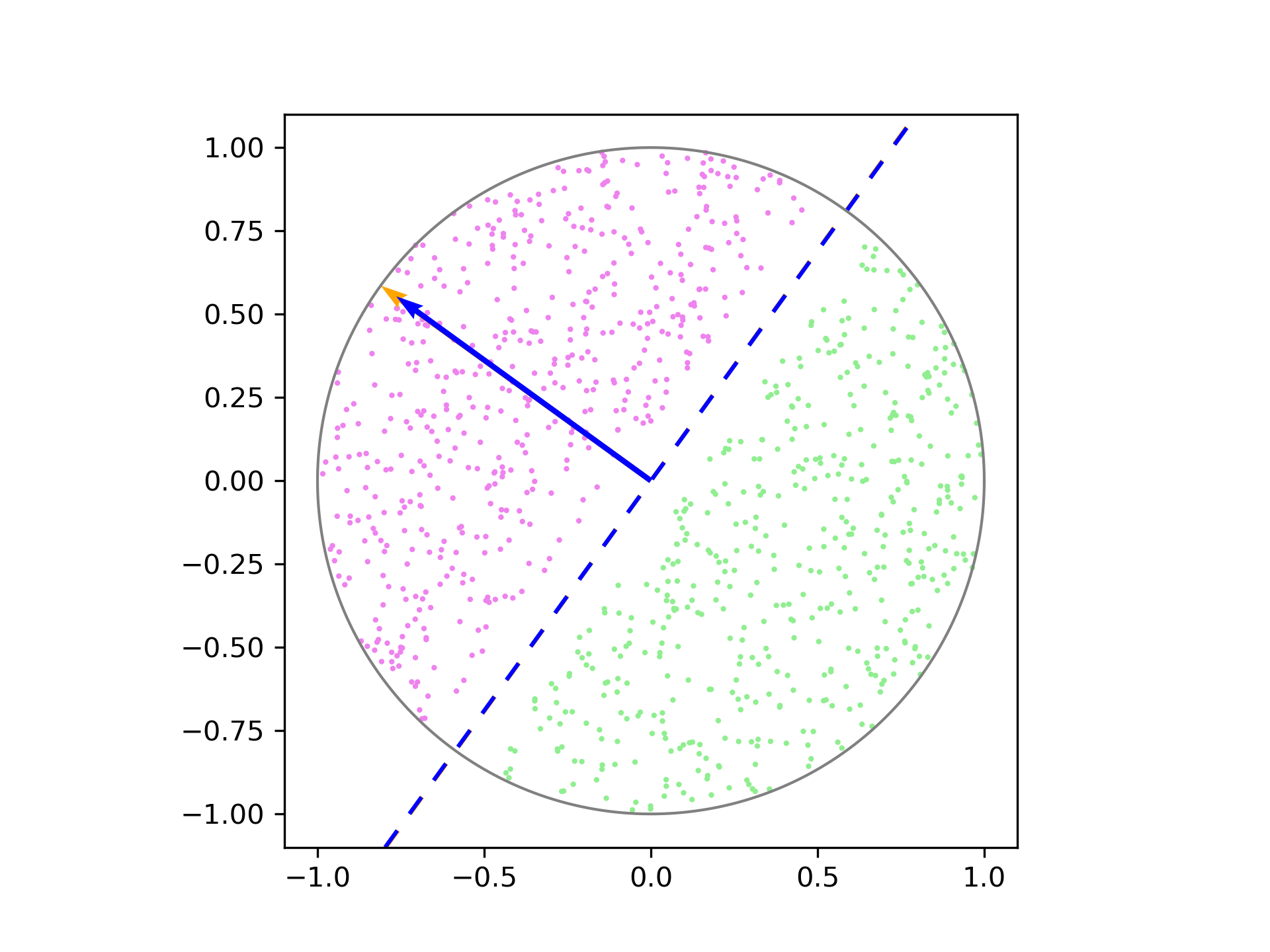}
    \end{subfigure}
    \begin{subfigure}{.19\linewidth}
        \centering
        \includegraphics[width=\linewidth]{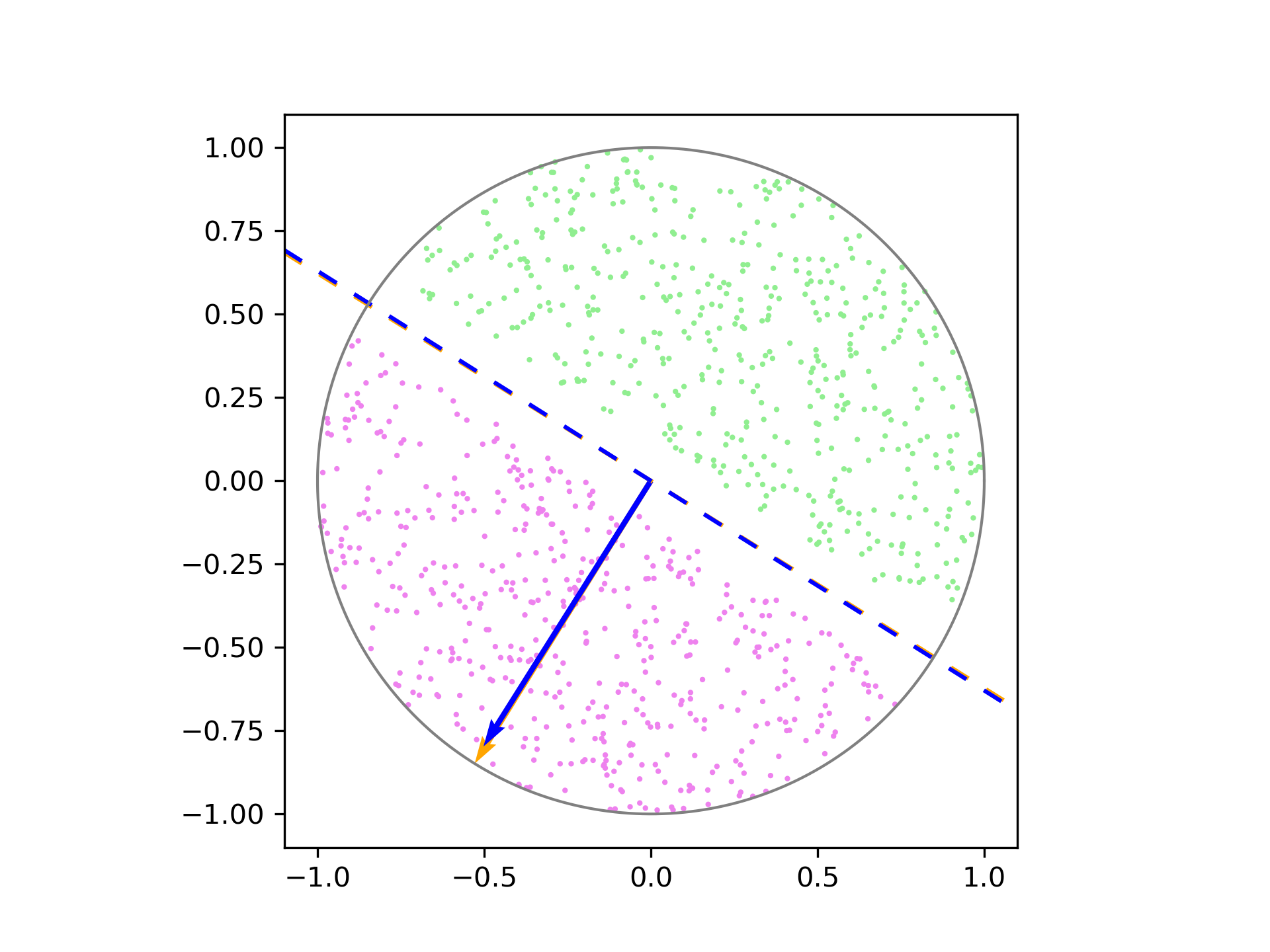}
    \end{subfigure}
    \begin{subfigure}{.19\linewidth}
        \centering
        \includegraphics[width=\linewidth]{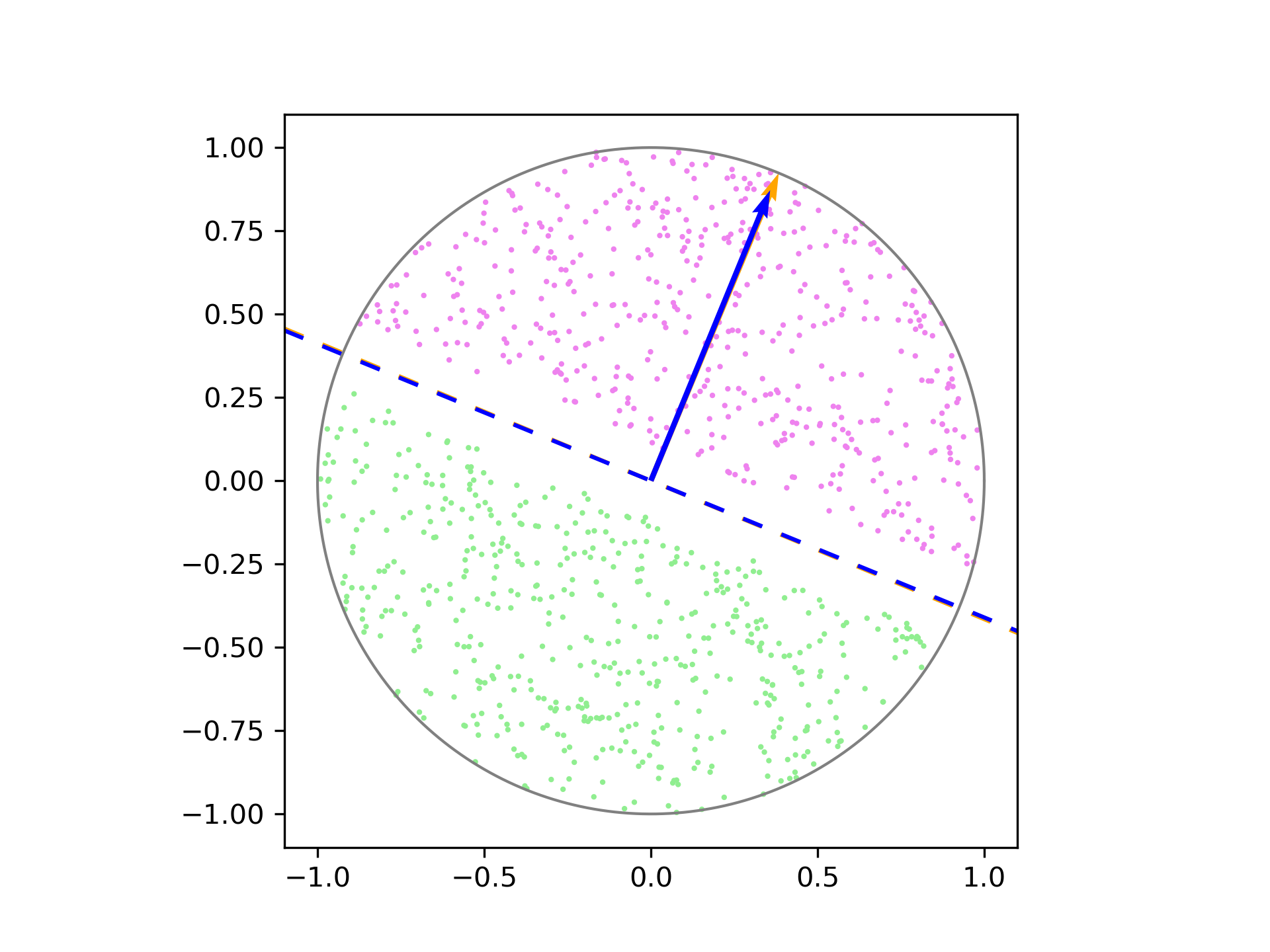}
    \end{subfigure}
    \begin{subfigure}{.19\linewidth}
        \centering
        \includegraphics[width=\linewidth]{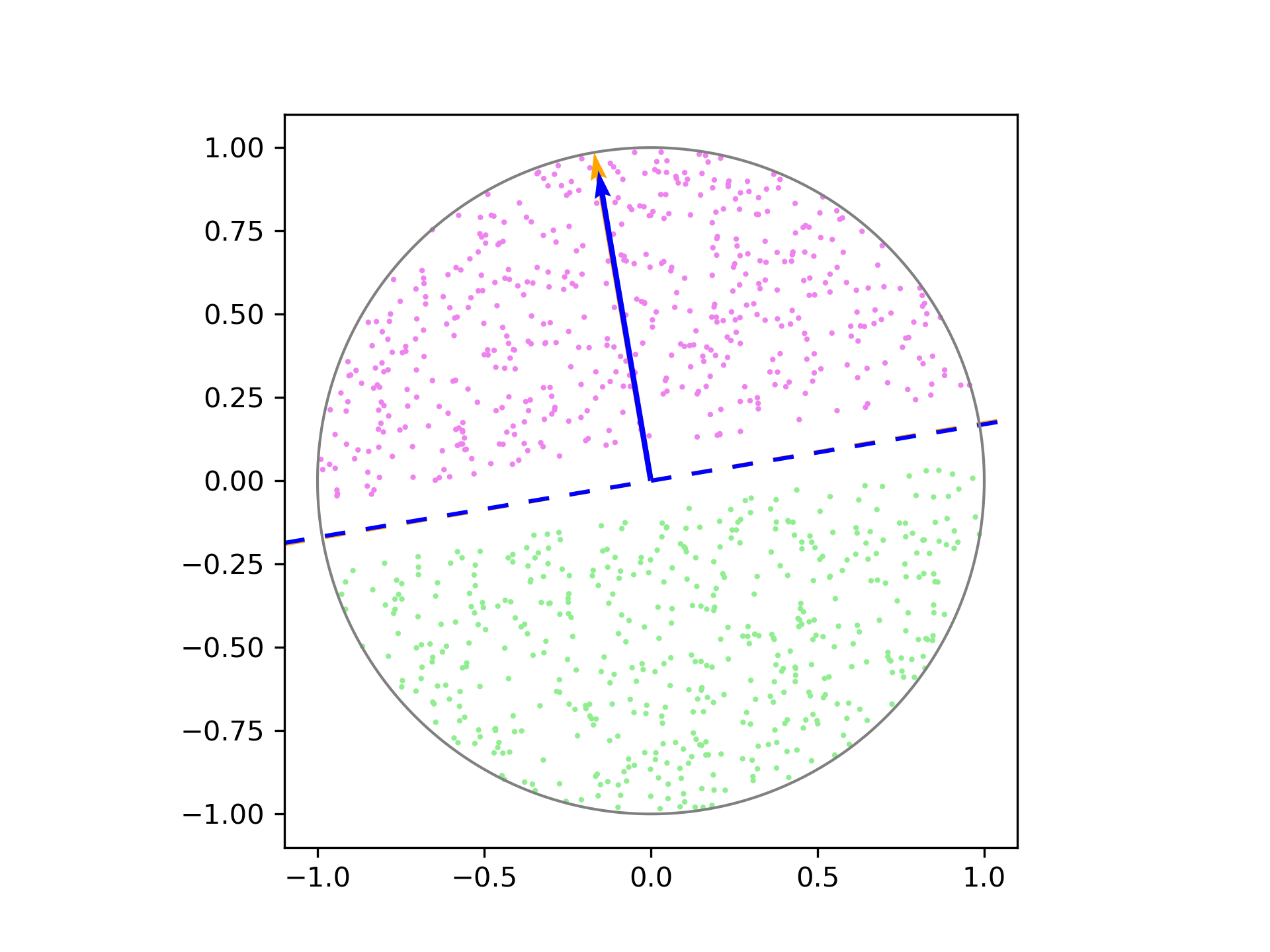}
    \end{subfigure}
    \caption{
    	Classifiers achieved by running \ref{SCMWU-ball} against \ref{MWU} in the SVM game (on $n = 10^3$ data points in $\R^2$) for time horizon $T = 10^3$. ({\textcolor{blue} {Blue}}: classifier $\overline{x}$ achieved by running \ref{SCMWU-ball} vs. \ref{MWU} in SVM game. {\textcolor{orange} {Orange}}: classifier $w$ used to generate data.)
	} 
    \label{fig:_SVM_viz_T1000}
\end{figure}

\begin{figure}[H]
\centering
    \begin{subfigure}{.19\linewidth}
        \centering
        \includegraphics[width=\linewidth]{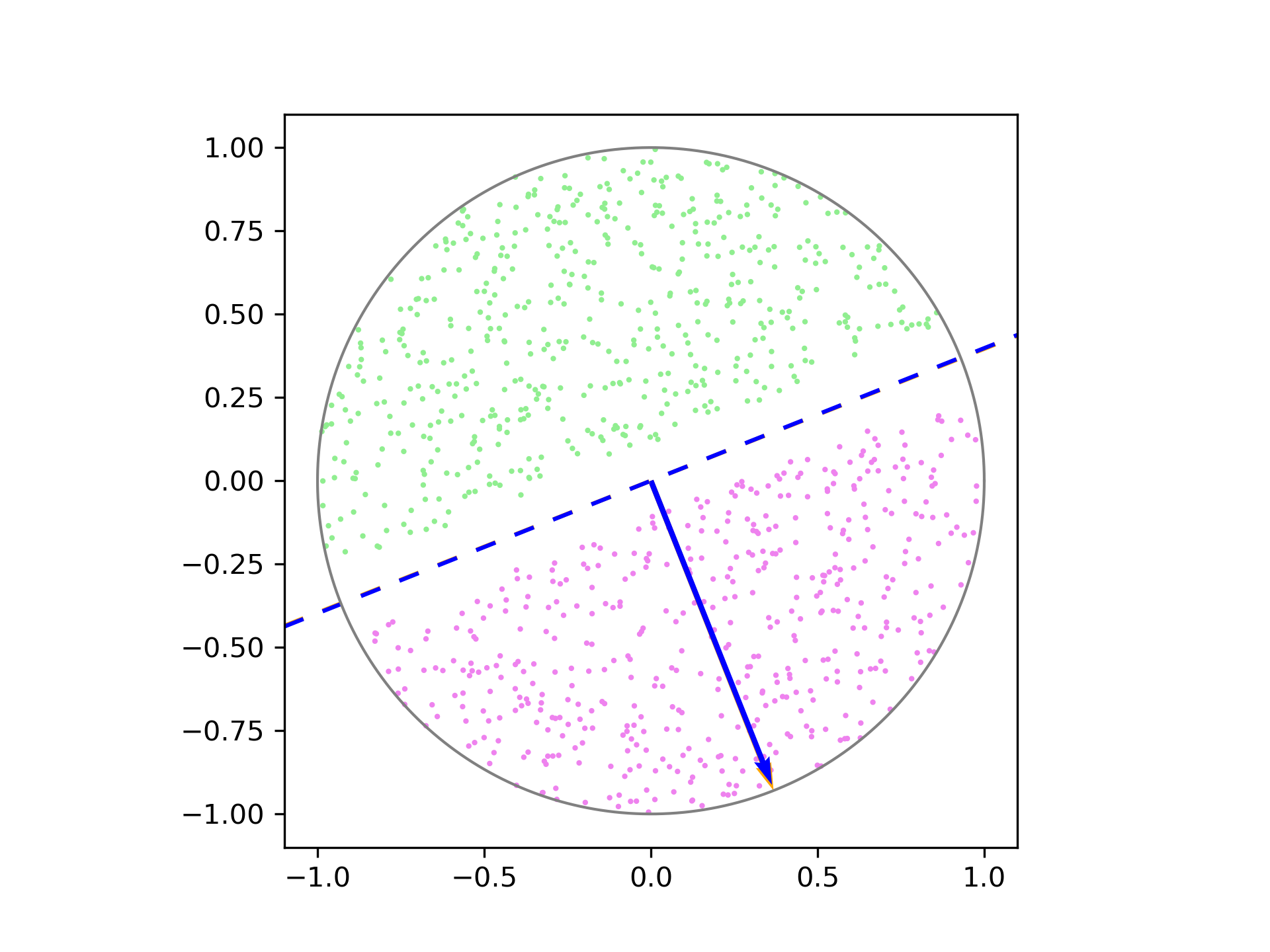}
    \end{subfigure}
    \begin{subfigure}{.19\linewidth}
        \centering
        \includegraphics[width=\linewidth]{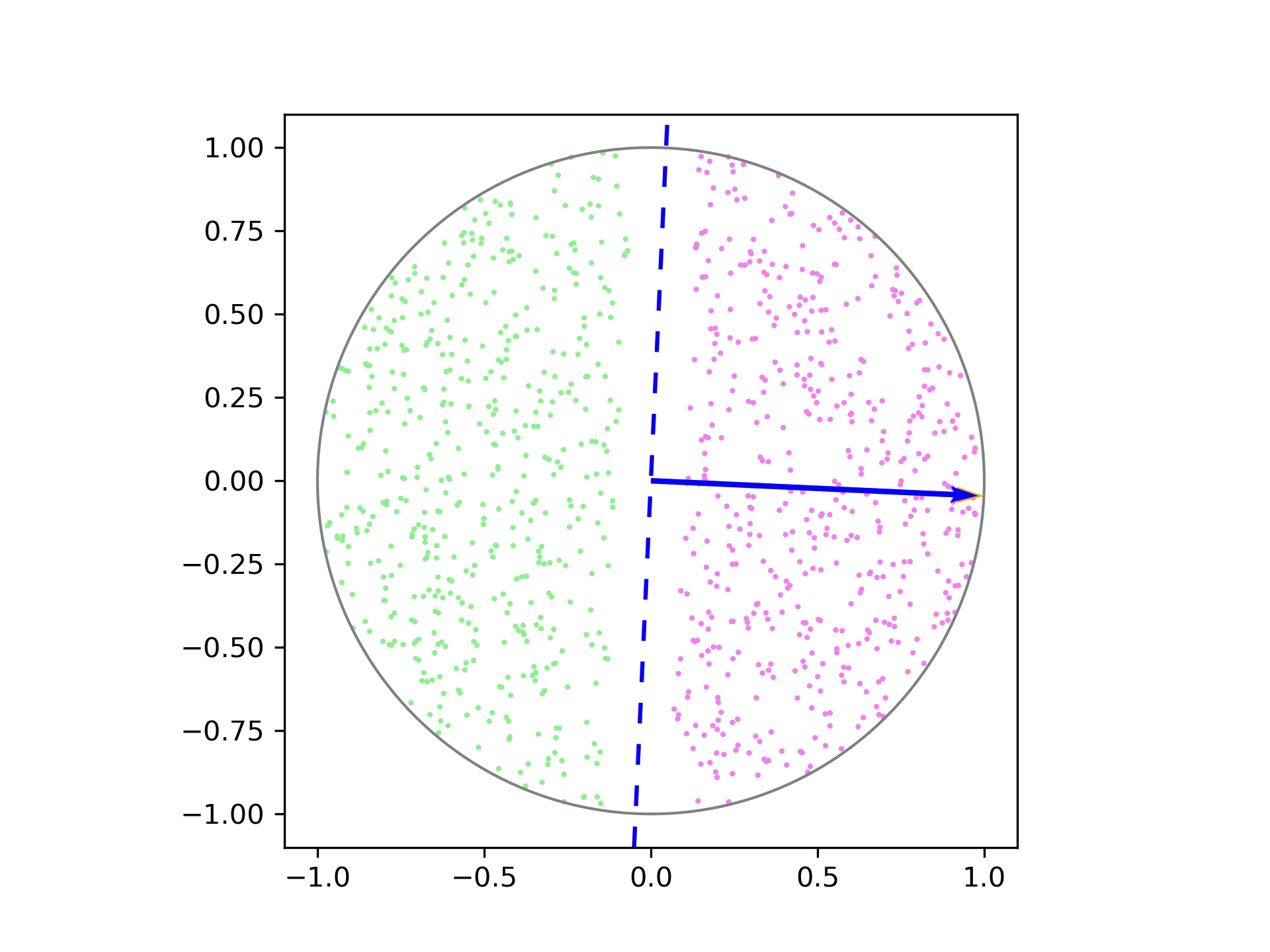}
    \end{subfigure}
    \begin{subfigure}{.19\linewidth}
        \centering
        \includegraphics[width=\linewidth]{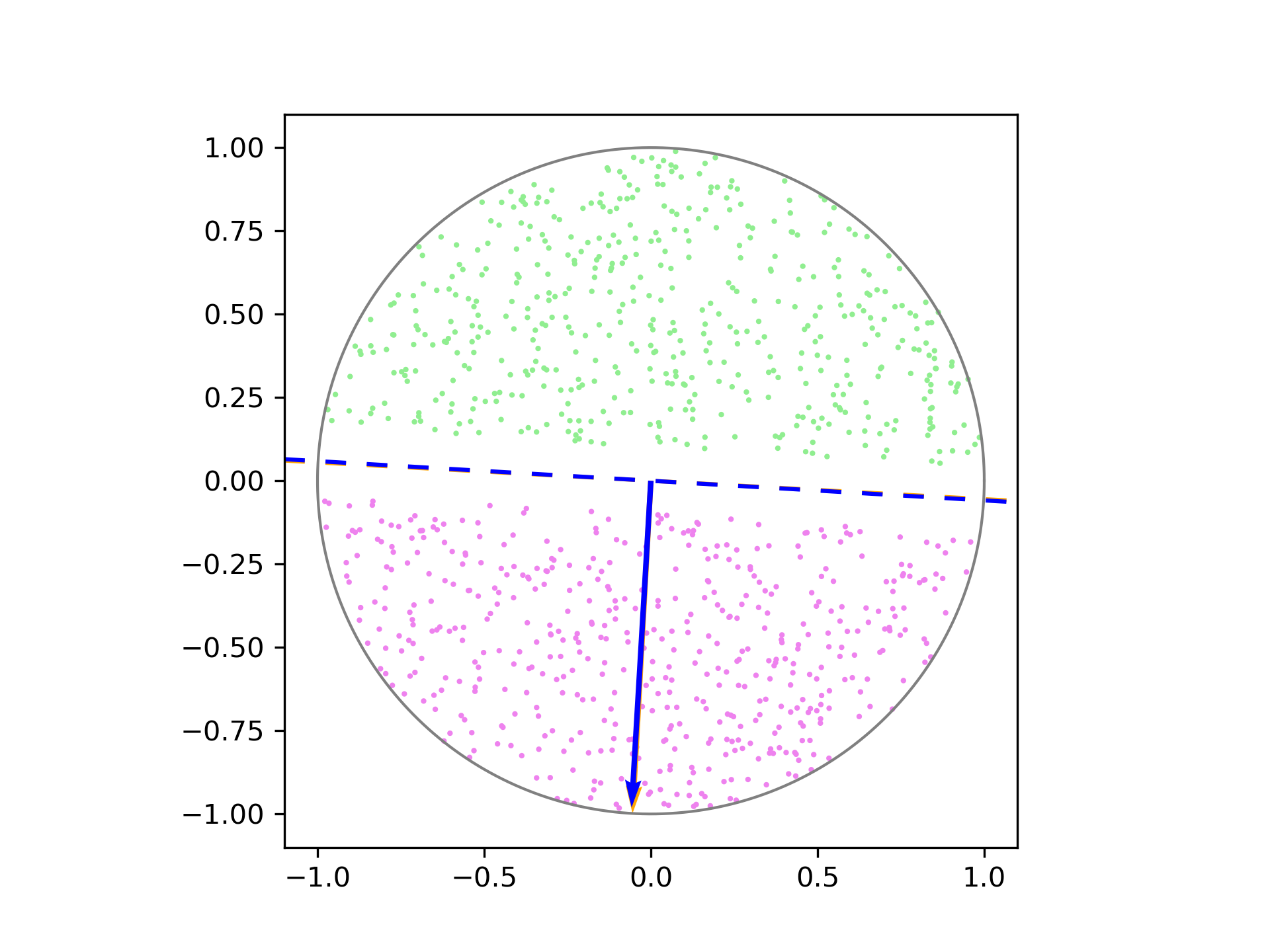}
    \end{subfigure}
    \begin{subfigure}{.19\linewidth}
        \centering
        \includegraphics[width=\linewidth]{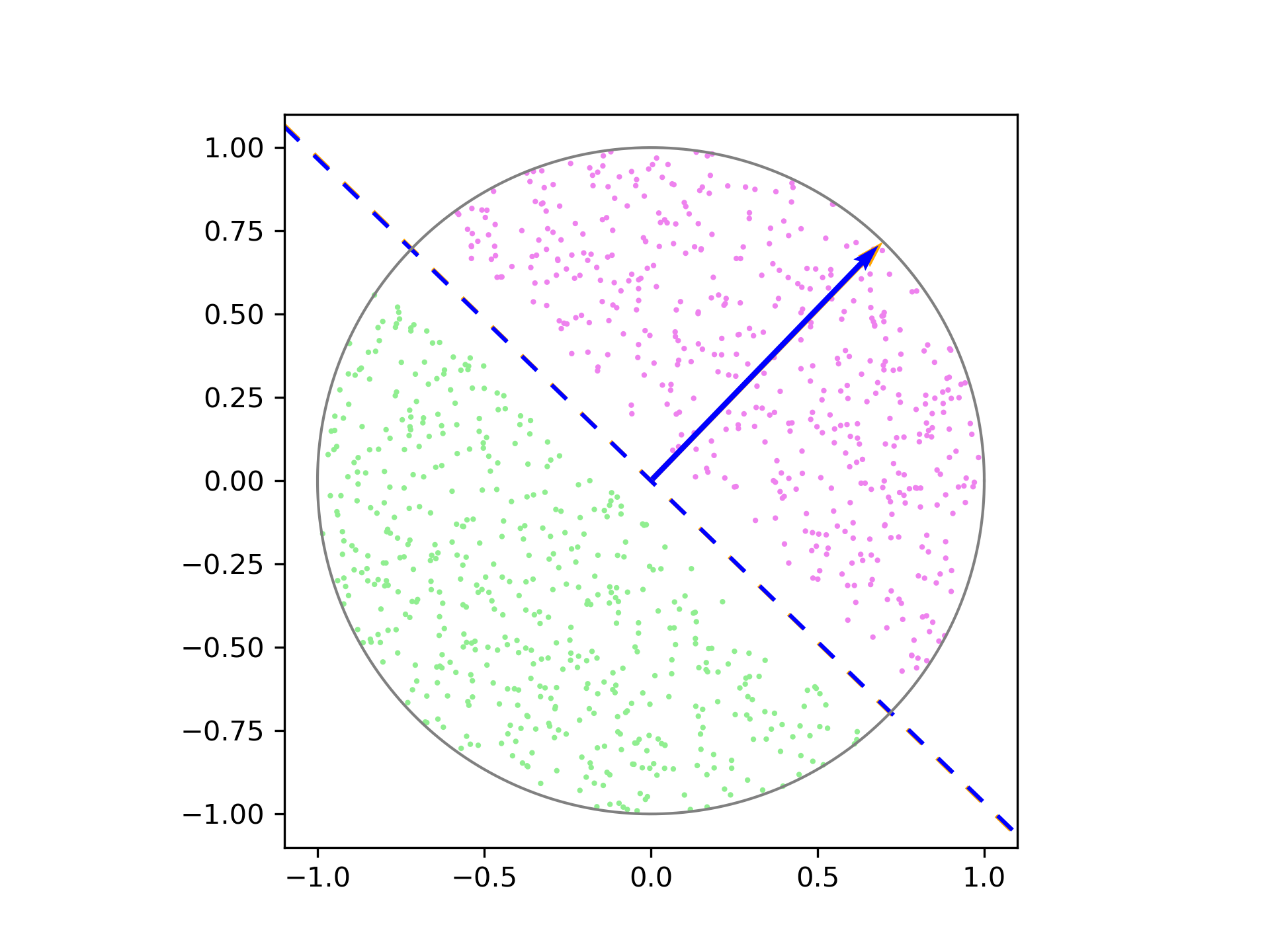}
    \end{subfigure}
    \begin{subfigure}{.19\linewidth}
        \centering
        \includegraphics[width=\linewidth]{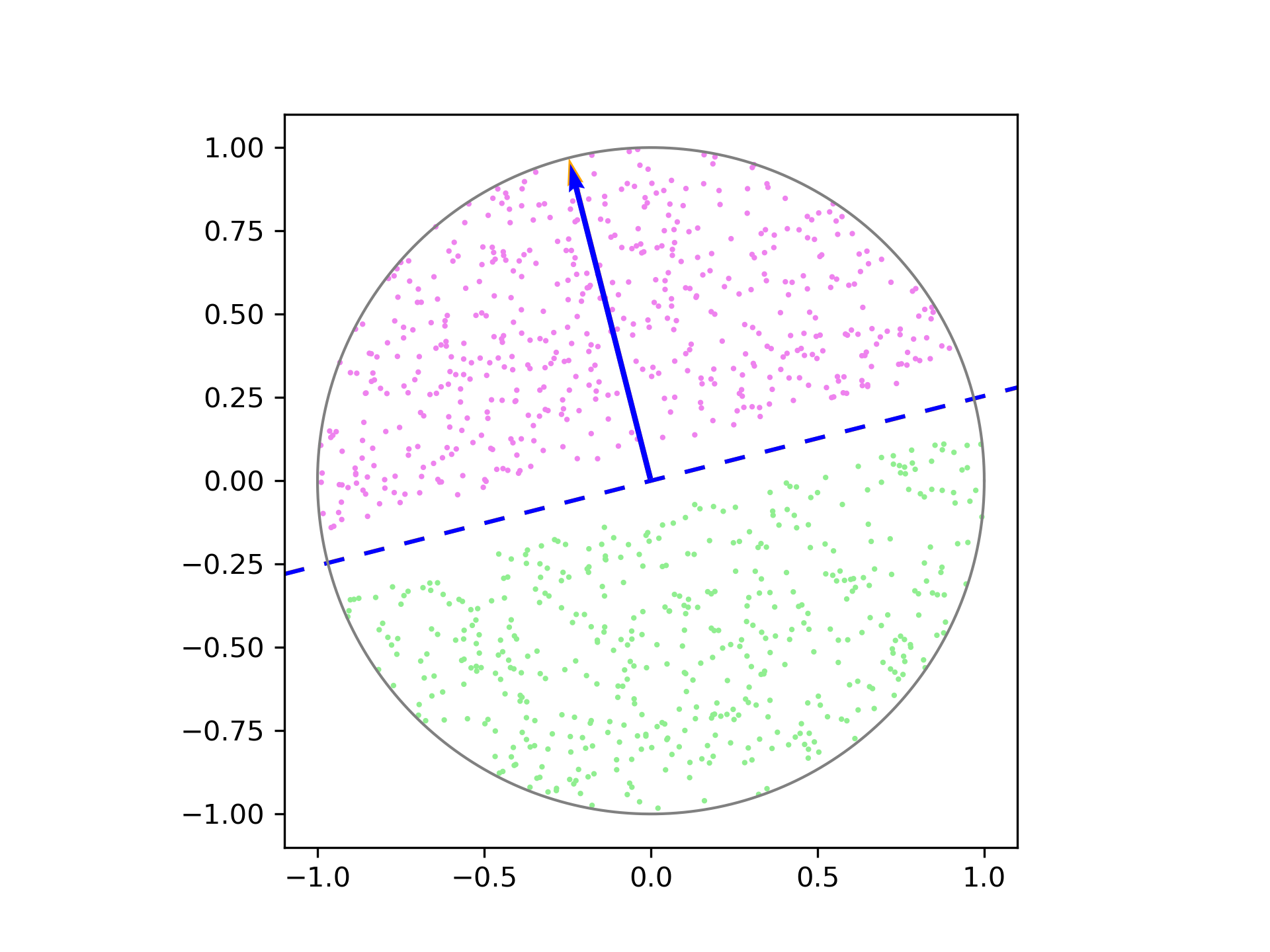}
    \end{subfigure}
    
    \begin{subfigure}{.19\linewidth}
        \centering
        \includegraphics[width=\linewidth]{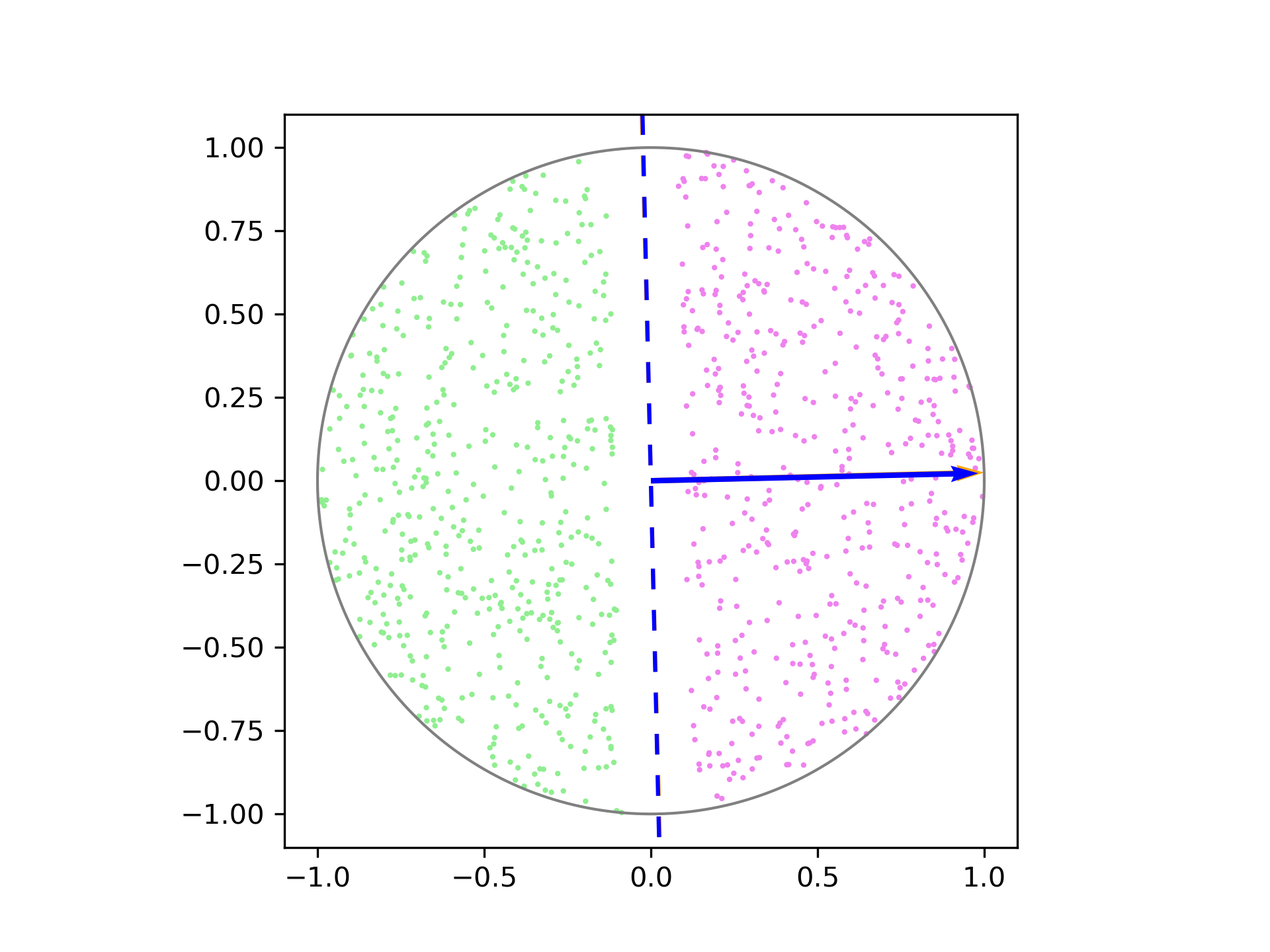}
    \end{subfigure}
    \begin{subfigure}{.19\linewidth}
        \centering
        \includegraphics[width=\linewidth]{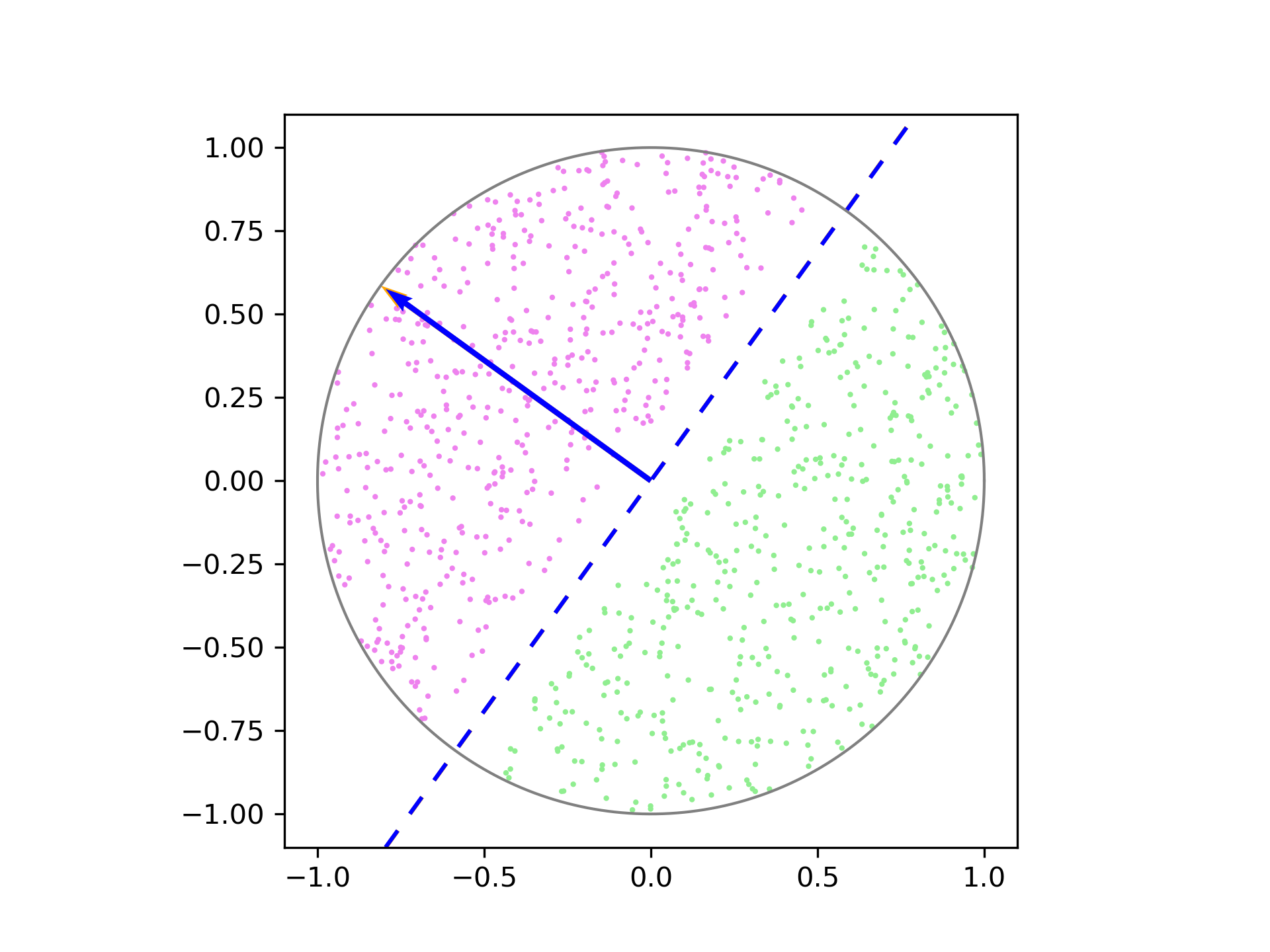}
    \end{subfigure}
    \begin{subfigure}{.19\linewidth}
        \centering
        \includegraphics[width=\linewidth]{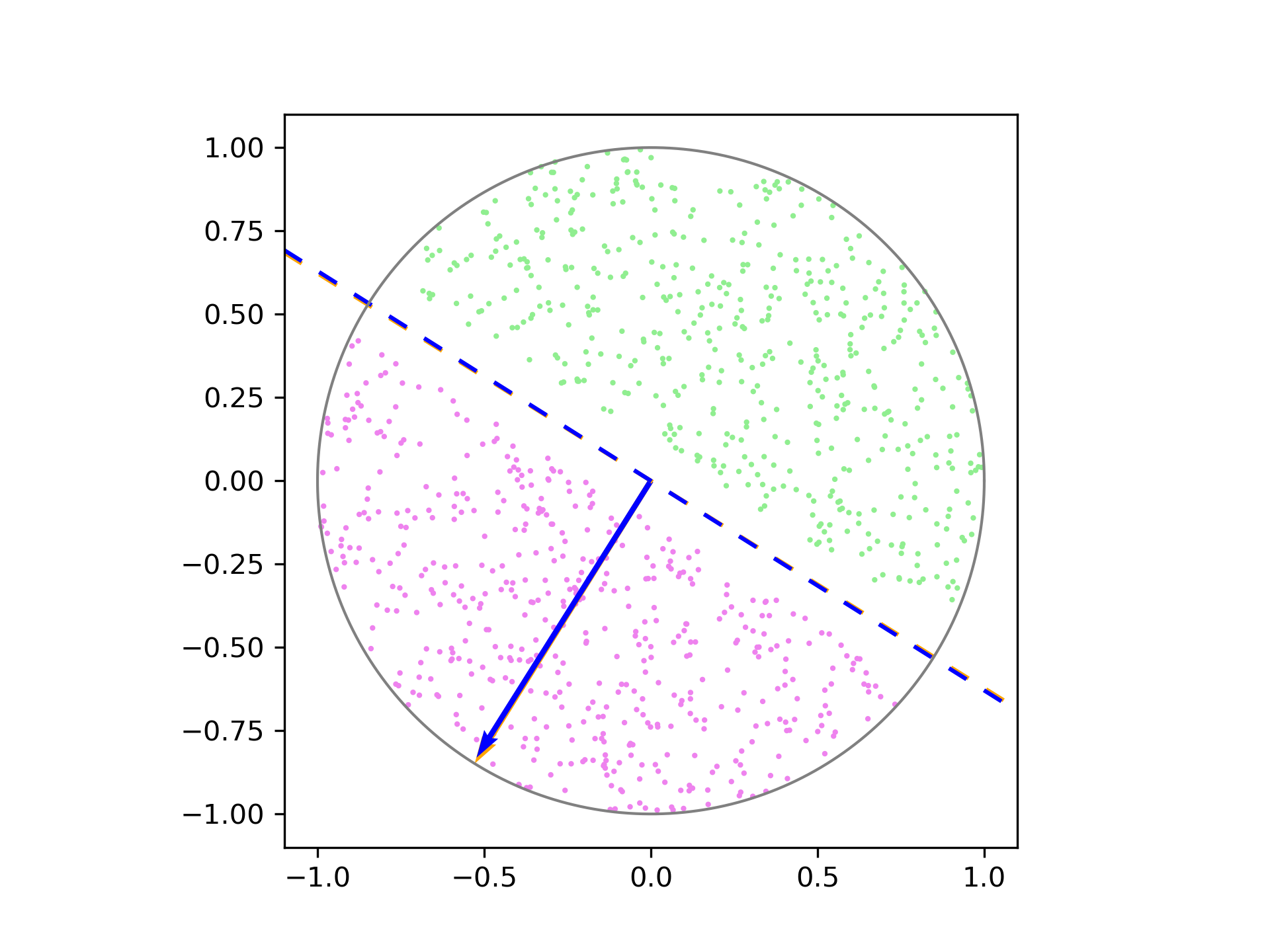}
    \end{subfigure}
    \begin{subfigure}{.19\linewidth}
        \centering
        \includegraphics[width=\linewidth]{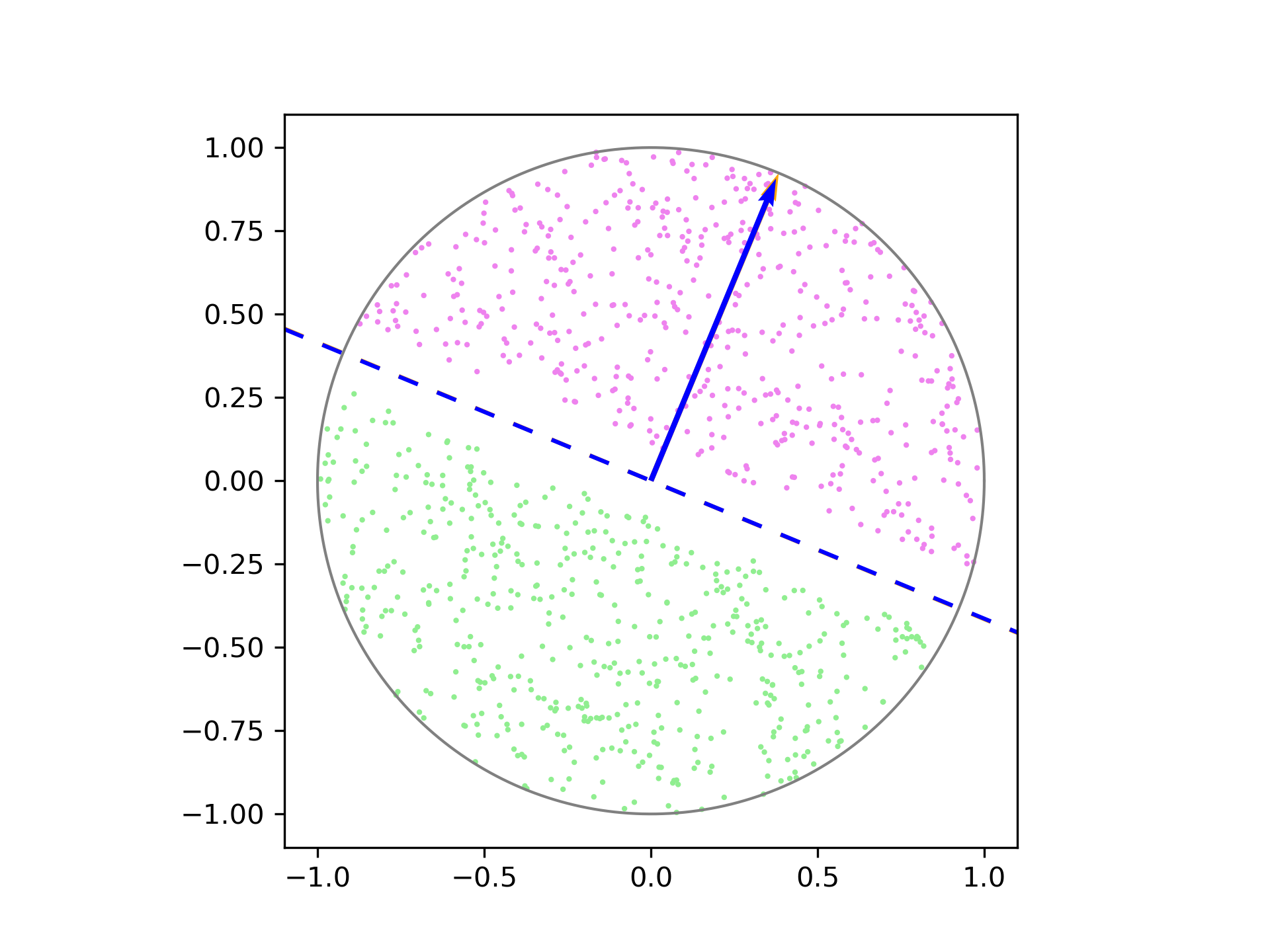}
    \end{subfigure}
    \begin{subfigure}{.19\linewidth}
        \centering
        \includegraphics[width=\linewidth]{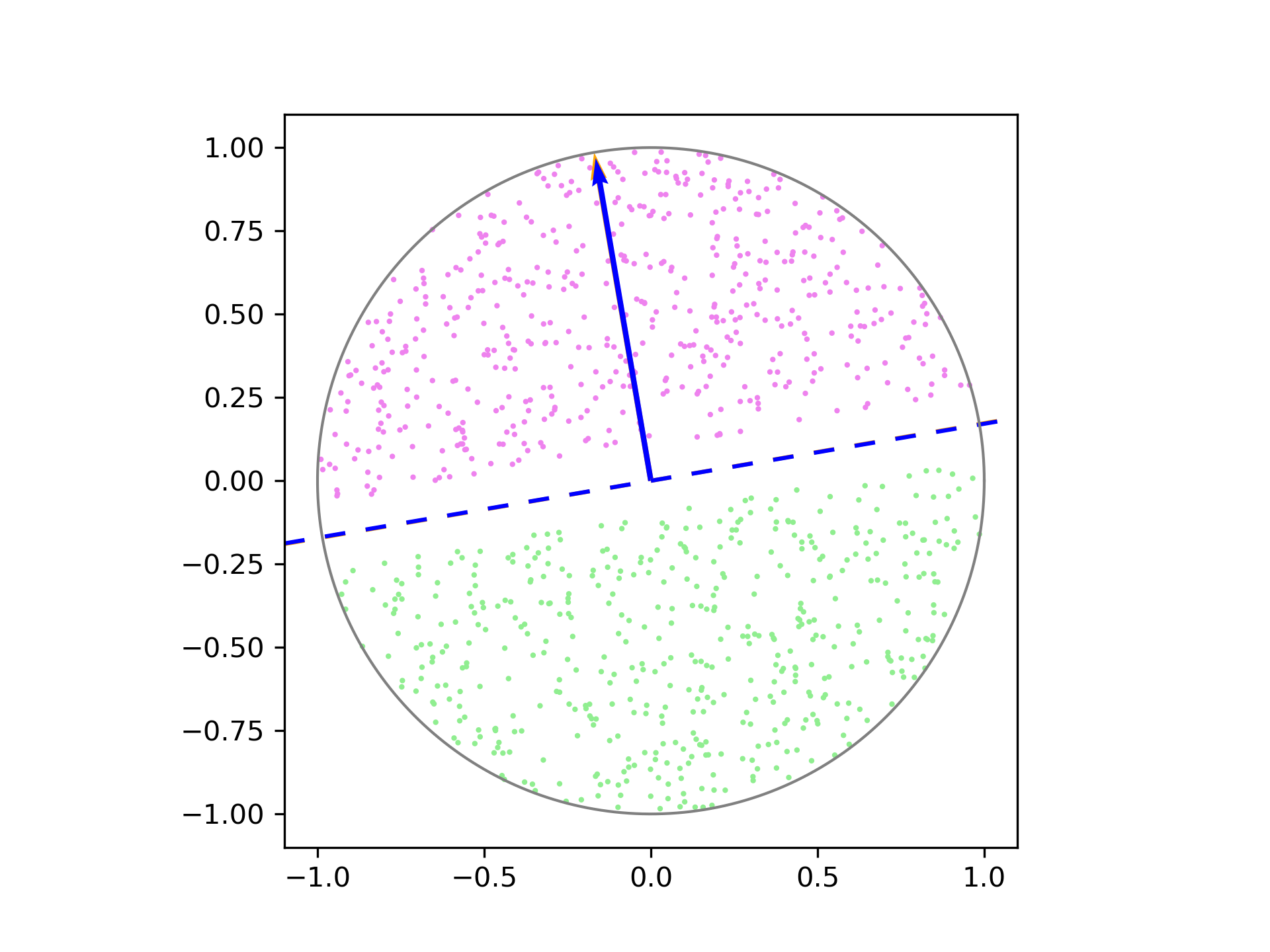}
    \end{subfigure}
    \caption{
    	Classifiers achieved by running \ref{SCMWU-ball} against \ref{MWU} in the SVM game (on $n = 10^3$ data points in $\R^2$) for time horizon $T = 10^4$. ({\textcolor{blue} {Blue}}: classifier $\overline{x}$ achieved by running \ref{SCMWU-ball} vs. \ref{MWU} in SVM game. {\textcolor{orange} {Orange}}: classifier $w$ used to generate data.)
	} 
    \label{fig:_SVM_viz_T10000}
\end{figure}

\section{Concluding Remarks}
 In this paper we introduce the framework  of Online Symmetric-Cone Optimization  that unifies and  generalizes the classical and quantum experts setups. Symmetric cones provide a  unifying framework for some of the most important optimization models, including linear, second-order cone, and semidefinite optimization. We introduce  the Symmetric-Cone Multiplicative Weights Update (\ref{SCMWU}), a no-regret algorithm that not only generalizes  the celebrated Multiplicative Weights Update algorithm over simplices and its matrix counterpart over density matrices but is also equivalent to both OMD and FTRL in this setting by nature of its entropic regularizer. That SCMWU is both OMD and FTRL is useful because there are different settings where one approach outperforms the other due to tricks/variations that can only apply in one of the two settings (see e.g., \cite{Fang}).

No-regret  algorithms for the experts setting (both the classical and quantum versions) have  important applications beyond online optimization that include distributed  learning dynamics  in game theory and algorithms for solving linear/semidefinite programs. One specific application of multiplicative weight algorithms to conic programming has been the application of \ref{MMWU} in devising efficient primal-dual algorithms for solving semidefinite programs \cite{ART:AK16,kale} and our work raises the immediate possibility of application in interesting and efficient algorithms for solving second-order cone programs with similar techniques and analysis as for MMWU. Inspired by these applications, in follow-up works  we explore uses of \ref{SCMWU} for learning in games where action spaces are symmetric cones, and for algorithms  for linear optimization over symmetric~cones.

\section*{Acknowledgments}
 
This research is supported in part by the National Research Foundation, Singapore and the Agency for Science, Technology and Research (A*STAR) under its Quantum Engineering Programme NRF2021-QEP2-02-P05, and by the National Research Foundation, Singapore and DSO National Laboratories under its AI Singapore Program (AISG Award No: AISG2-RP-2020-016), NRF 2018 Fellowship NRF-NRFF2018-07, NRF2019-NRF-ANR095 ALIAS grant, grant PIESGP-AI-2020-01, AME Programmatic Fund (Grant No.A20H6b0151) from A*STAR and Provost’s Chair Professorship grant RGEPPV2101. Wayne Lin gratefully acknowledges support from the SUTD President's Graduate Fellowship (SUTD-PGF).

\bibliography{arXivrefs}

\begin{thebibliography}{10}

\bibitem{Alizadeh}
F.~Alizadeh and D.~Goldfarb.
\newblock Second-order cone programming.
\newblock {\em Mathematical Programming}, 95:3--51, 2003.

\bibitem{spectral}
Z.~Allen-Zhu, Z.~Liao, and L.~Orecchia.
\newblock Spectral sparsification and regret minimization beyond matrix
  multiplicative update.
\newblock In {\em Proceedings of the Forty-Seventh Annual ACM Symposium on
  Theory of Computing}, (STOC 2015), page 237–245, New York, NY, USA, 2015.
  Association for Computing Machinery.

\bibitem{ART:AHK12}
S.~Arora, E.~Hazan, and S.~Kale.
\newblock The multiplicative weights update method: a meta-algorithm and
  applications.
\newblock {\em Theory of Computing}, 8(6):121--164, 2012.

\bibitem{ART:AK16}
S.~Arora and S.~Kale.
\newblock A combinatorial, primal-dual approach to semidefinite programs.
\newblock {\em J. ACM}, 63(2), 2016.

\bibitem{bauschke}
H.~H. Bauschke, J.~M. Borwein, et~al.
\newblock Legendre functions and the method of random {B}regman projections.
\newblock {\em Journal of convex analysis}, 4(1):27--67, 1997.

\bibitem{BOO:B04}
S.~Boyd, S.~P. Boyd, and L.~Vandenberghe.
\newblock {\em Convex optimization}.
\newblock Cambridge university press, 2004.

\bibitem{bubeck}
S.~Bubeck.
\newblock Introduction to online optimization.
\newblock Technical report, Department of Operations Research and Financial
  Engineering, Princeton University, Princeton, NJ, 2011.

\bibitem{carlen}
E.~A. Carlen.
\newblock Trace inequalities and quantum entropy: An introductory course.
\newblock Technical report, Department of Mathematics, Rutgers University,
  Piscataway, NJ, 2009.

\bibitem{carmon2019variance}
Y.~Carmon, Y.~Jin, A.~Sidford, and K.~Tian.
\newblock Variance reduction for matrix games.
\newblock {\em Advances in Neural Information Processing Systems}, 32, 2019.

\bibitem{cesabianchi}
N.~Cesa-Bianchi and G.~Lugosi.
\newblock {\em Prediction, learning, and games}.
\newblock Cambridge university press, 2006.

\bibitem{ART:CP10}
J.-S. Chen and S.~Pan.
\newblock An entropy-like proximal algorithm and the exponential multiplier
  method for convex symmetric cone programming.
\newblock {\em Computational Optimization and Applications}, 47:477--499, 11
  2010.

\bibitem{chen}
Y.~Chen and X.~Ye.
\newblock Projection onto a simplex.
\newblock {\em arXiv preprint arXiv:1101.6081}, 2011.

\bibitem{clarkson2012sublinear}
K.~L. Clarkson, E.~Hazan, and D.~P. Woodruff.
\newblock Sublinear optimization for machine learning.
\newblock {\em Journal of the ACM (JACM)}, 59(5):1--49, 2012.

\bibitem{duchi2008efficient}
J.~Duchi, S.~Shalev-Shwartz, Y.~Singer, and T.~Chandra.
\newblock Efficient projections onto the l1-ball for learning in high
  dimensions.
\newblock In {\em Proceedings of the 25th international conference on Machine
  learning}, pages 272--279, 2008.

\bibitem{Fang}
H.~Fang, N.~J.~A. Harvey, V.~S. Portella, and M.~P. Friedlander.
\newblock Online mirror descent and dual averaging: Keeping pace in the dynamic
  case.
\newblock {\em Journal of Machine Learning Research}, 23(121):1--38, 2022.

\bibitem{BOO:FK94}
J.~Faraut and A.~Kor{\'a}nyi.
\newblock {\em Analysis on Symmetric Cones}.
\newblock Oxford mathematical monographs. Clarendon Press, 1994.

\bibitem{freund}
Y.~Freund and R.~Shapire.
\newblock Adaptive game playing using multiplicative weights.
\newblock {\em Games and Economic Behavior}, pages 79--103, 1999.

\bibitem{BOO:H19}
E.~Hazan.
\newblock Introduction to online convex optimization.
\newblock {\em Found. Trends Optim.}, 2(3-4):157--325, 2016.

\bibitem{INP:HK08}
E.~Hazan and S.~Kale.
\newblock Extracting certainty from uncertainty: Regret bounded by variation in
  costs.
\newblock In {\em 21st Annual Conference on Learning Theory, (COLT 2008)},
  pages 57--67, Dec. 2008.

\bibitem{hildebrand}
R.~Hildebrand.
\newblock {A partial differential equation characterizing determinants of
  symmetric cones}.
\newblock In {\em {MAP 2012 - Mathematics, Algorithms, and Proofs 2012}},
  Constance, Germany, Sept. 2012.

\bibitem{jain2011qip}
R.~Jain, Z.~Ji, S.~Upadhyay, and J.~Watrous.
\newblock {QIP=PSPACE}.
\newblock {\em Journal of the ACM}, 58(6):1--27, 2011.

\bibitem{kale}
S.~Kale.
\newblock {\em Efficient Algorithms Using The Multiplicative Weights Update
  Method}.
\newblock PhD thesis, Department of Computer Science, Princeton University,
  2007.

\bibitem{lwar}
N.~Littlestone and M.~K. Warmuth.
\newblock The weighted majority algorithm.
\newblock {\em Information and Computation}, 108(2):212--261, 1994.

\bibitem{ART:McMahan}
H.~B. McMahan.
\newblock A survey of algorithms and analysis for adaptive online learning.
\newblock {\em J. Mach. Learn. Res.}, 18(1):3117–3166, jan 2017.

\bibitem{nem}
A.~Nemirovsky and D.~Yudin.
\newblock {\em Problem Complexity and Method Efficiency in Optimization}.
\newblock John Wiley \& Sons, 1983.

\bibitem{orabona}
F.~Orabona.
\newblock A modern introduction to online learning, 2022.

\bibitem{ART:SA03}
S.~Schmieta and F.~Alizadeh.
\newblock Extension of primal-dual interior point algorithms to symmetric
  cones.
\newblock {\em Mathematical Programming}, 96:409--438, 01 2003.

\bibitem{shwartz}
S.~Shalev-Shwartz.
\newblock Online learning and online convex optimization.
\newblock {\em Found. Trends Mach. Learn.}, 4(2):107–194, 2012.

\bibitem{ART:TWK21}
J.~Tao, G.~Q. Wang, and L.~Kong.
\newblock The {A}raki-{L}ieb-{T}hirring inequality and the {G}olden-{T}hompson
  inequality in {E}uclidean {J}ordan algebras.
\newblock {\em Linear and Multilinear Algebra}, 0(0):1--16, 2021.

\bibitem{ART:TRW05}
K.~Tsuda, G.~R{{\"a}}tsch, and M.~K. Warmuth.
\newblock Matrix exponentiated gradient updates for on-line learning and
  {B}regman projection.
\newblock {\em Journal of Machine Learning Research}, 6(34):995--1018, 2005.

\bibitem{vanden}
L.~Vandenberghe.
\newblock Symmetric cones [{L}ecture notes], {E}lectrical and computer
  engineering department, {UCLA}, {S}pring, 2016.

\bibitem{PHD:V07}
M.~Vieira.
\newblock {\em Jordan algebraic approach to symmetric optimization}.
\newblock PhD thesis, Electrical Engineering, Mathematics and Computer Science,
  Delft University of Technology, 2007.

\bibitem{vovk}
V.~Vovk.
\newblock Aggregating strategies.
\newblock In {\em Proceedings of the 3rd Workshop on Computational Learning
  Theory}, COLT '90, pages 371--383, San Francisco, CA, USA, 1990. Morgan
  Kaufmann Publishers Inc.

\bibitem{INP:WK06}
M.~K. Warmuth and D.~Kuzmin.
\newblock Online variance minimization.
\newblock In G.~Lugosi and H.~U. Simon, editors, {\em Learning Theory}, pages
  514--528, Berlin, Heidelberg, 2006. Springer Berlin Heidelberg.

\bibitem{yang2020revisiting}
F.~Yang, J.~Jiang, J.~Zhang, and X.~Sun.
\newblock Revisiting online quantum state learning.
\newblock In {\em Proceedings of the AAAI Conference on Artificial
  Intelligence}, volume~34, pages 6607--6614, 2020.

\bibitem{INP:Z03}
M.~Zinkevich.
\newblock Online convex programming and generalized infinitesimal gradient
  ascent.
\newblock In {\em Proceedings of the Twentieth International Conference on
  International Conference on Machine Learning}, (ICML 2003), pages 928--935,
  Washington, DC, USA, 2003. AAAI Press.

\end{thebibliography}
\bibliographystyle{abbrv}

\newpage
\appendix
\onecolumn

\section{Additional Technical Results} \label{appendix}

The following well-known theorem (see, e.g., Lemma 13 in \cite{ART:SA03}) gives a formula for the minimum eigenvalue of an element in $\cj$. We include the proof here for completeness.

\begin{theorem}\label{thm:eig_ineq}  Let $(\cj,\circ)$ be an EJA and $\cc$ its cone of squares. For any $x\in\cj$,  its smallest eigenvalue is given by
  $$\lam_{\min}(x)=\min_{s}\{ \la x, s\ra: {\rm tr}(s)=1, s\in \cc\}.$$  
\end{theorem}
\begin{proof}
Let $x=\sum_{i=1}^r\lambda_iq_i$ be the spectral decomposition of $x$. Then,
\begin{align*}
\la x, s\ra&={\rm tr}(x\circ s)={\rm tr } (\sum_i\lambda_iq_i\circ s )=\sum_i\lambda_i{\rm tr} (q_i\circ s)\\
&\geq\lambda_{\min}\sum_i{\rm tr}(q_i\circ s)=\lambda_{\min}{\rm tr}(e\circ s)=\lambda_{\min}{\rm tr}(s)=\lambda_{\min},
\end{align*} 
where the above inequality follows from the fact that the cone of squares is self-dual with respect to the trace inner product, and since both $q_i,s\in\cc$, this implies that 
${\rm tr} (q_i\circ s)\ge 0$
(see, e.g., Lemma 3.3.4 in \cite{PHD:V07}).

\noindent Now, let $q_{i^*}$ be the eigenvector corresponding to the smallest eigenvalue $\lambda_{\min} = \lambda_{i^*}$. Then, 
\begin{align*}
  \la x, q_{i^*}\ra&=\la \sum_i\lambda_i q_i, q_{i^*}\ra={\rm tr}( \sum_i\lambda_i q_i\circ q_{i^*})\\
  &=  \sum_i\lambda_i {\rm tr}(q_i\circ q_{i^*})=\lambda_{i^*}{\rm tr}(q_{i^*}\circ q_{i^*})\\
  &=\lambda_{i^*}{\rm tr}(q_{i^*}) =\lambda_{\min},   
\end{align*}
where the last equality follows since ${\rm tr}(q_{i^*})=1.$  
\end{proof}

The following result presents well-known properties of the EJA  exponential  and logarithm (see, e.g., \cite{hildebrand}).
 \begin{lemma}\label{expvslog}

The L\"{o}wner extensions  corresponding to the scalar exponential  $\exp(\cdot)$ and $\ln(\cdot)$ functions  defined as
\begin{align*}
\expl: \ &\cj\to {\rm int}(\cc), \quad x=\sum_{i=1}^r\lambda_iq_i \mapsto\sum_{i=1}^r \exp(\lambda_i) q_i\\
\lnl : \ & {\rm int}(\cc)\to \cj, \quad x=\sum_{i=1}^r\lambda_iq_i \mapsto\sum_{i=1}^r \ln(\lambda_i) q_i,
\end{align*}
are inverses to each other. 
 \end{lemma} 
 \begin{proof}The range of $\expl$ is $\inte(\cc)$ since for any $y = \sum_{i=1}^r \lambda_i q_i \in \inte(\cc)$, $\lambda_i > 0 \ \forall  i$ so $\ln(\lambda_i)$ is well-defined and 
 \[
    \expl(\sum_{i=1}^r \ln(\lambda_i) q_i) = \sum_{i=1}^r \exp(\ln(\lambda_i))q_i = \sum_{i=1}^r \lambda_i q_i = y.
\]

Similarly, the range of $\lnl$ is $\cj$ since for any $ x = \sum_{i=1}^r \lambda_i q_i \in \cj$,
\[
    \lnl(\sum_{i=1}^r \exp(\lambda_i) q_i)
    = \sum_{i=1}^r \ln(\exp(\lambda_i)) q_i
    = \sum_{i=1}^r \lambda_i q_i
    = x.
\]

Finally, $\forall x = \sum_{i=1}^r \lambda_i q_i \in \cj$,
\[
    \lnl(\expl(x)) = \lnl(\sum_{i=1}^r \exp(\lambda_i) q_i) = \sum_{i=1}^r \ln(\exp(\lambda_i)) q_i = \sum_{i=1}^r \lambda_i q_i = y,
\]
and $\forall y = \sum_{i=1}^r \lambda_i q_i \in \inte(\cj)$ so that $\lambda_i > 0 \ \forall i$,  
\[
    \expl(\lnl y) = \expl(\sum_{i=1}^r \ln(\lambda_i) q_i) = \sum_{i=1}^r \exp(\ln(\lambda_i)) = \sum_{i=1}^r \lambda_i q_i = x,
\]
so $\expl$ and $\lnl$ are inverses of each other.
 \end{proof} 
 
\begin{lemma}\label{lem:expe}
Let $\cj$ be an EJA with identity element $e$. For all $x\in\cj$ and $c\in\R$,
$$\expl(c e+x)=\exp(c)\expl(x).$$
\end{lemma}
\begin{proof}
Let $x=\sum_{i=1}^r\lambda_iq_i$ be the spectral decomposition of $x$. Then,
\begin{align*}
    \expl(c e+x)&=\expl(c\sum_iq_i+\sum_i\lambda_iq_i)\\
    &=\expl(\sum_i(c+\lambda_i)q_i)=\sum_i \exp(c+\lambda_i)q_i\\
    &=\exp(c)\sum_i\exp(\lambda_i)q_i\\
    &=\exp(c)\expl(x),
\end{align*}
where the first equality follows from the fact that $\sum_iq_i=~e.$
\end{proof}

\begin{lemma}\label{lem:cone_ineq}
Let $f,g:\R\to \R$ satisfy $g(t) \le f(t) \; \forall t\in S \subseteq \R$, and let $(\cj,\circ)$ be an EJA with cone of squares $\cc$. Then, for any element $x\in\cj$  with all its eigenvalues $\lambda_i \in S$, the following holds:
$$g(x) \preceq_\cc f(x).$$
\end{lemma}
\begin{proof}
Let $x=\sum_{i=1}^r\lambda_iq_i$ be the spectral decomposition of $x$. Then, the spectral decomposition of $f(x) - g(x)$ is given by
$$f(x) - g(x) = \sum_{i=1}^r (f(\lambda_i) - g(\lambda_i))q_i.$$
Moreover, since $\lambda_i \in S$ $\forall i$ we have that $f(\lambda_i) - g(\lambda_i) \geq 0$ $\forall i$, and hence the above is an element of $\cc$.   
\end{proof}

\begin{lemma} A quadratic upper bound on $\exp$ is given by
\label{bound:_exp_real_quad}
	\[
		\exp(-x) \leq 1 - x + x^2 \qquad \forall x \in \R, \; x \geq -\frac{3}{e}.
	\]
\end{lemma}

\begin{proof}
	By Taylor's theorem,
	\begin{equation*}
	\begin{split}
		\exp(-x)
		&= 
		\exp(-0) +\frac{ - \exp(-0) }{1} x + \frac{ \exp(-0) }{2} x^2 +R_3(x)
		\\
		&=
		1 - x + \frac{1}{2} x^2 + R_3(x),
	\end{split}
	\end{equation*}
	where $R_3(x) = - \frac{\exp(-\xi_L)}{3!} x^3$ for some $\xi_L$ in the interval between $0$ and $x$.
	
	For all $ x \geq 0$, we have that $R_3(x) \leq 0$, which implies that
	\[
		\exp(-x) \leq 1 - x + \frac{1}{2}x^2 \leq 1 - x + x^2.
	\]
	
	On the other hand, $\forall x \in [-\frac{3}{e}, 0)$,
	\[
		3 \geq -ex \geq - \exp(-\xi_L) x \qquad \forall \xi_L \in [x, 0]
	\]
	so that
	\[
		R_3 (x) = - \frac{\exp(-\xi_L)}{3!} x^3 \leq \frac{1}{2} x^2,
	\]
	and thus
	\[
		\exp(-x)
		=
		1 - x + \frac{1}{2} x^2 + R_3(x)
		\leq
		1 + x + x^2.
	\]
\end{proof}

\begin{lemma}
\label{lem:_entropy_central_bound}
    Let $(\cj,\circ)$ be a rank-$r$ EJA and and $\cc$ its cones of squares. Then, we have that
    $$\HH_\Phi \left(u,\frac{e}{r} \right) \leq \ln r  \quad \forall u\in \cc \text{ with } \tr(u)=1.$$
\end{lemma}

\begin{proof}
Since $\tr(\frac{e}{r}) = \tr(u) = 1$, we have that
\[
\HH_\Phi \left(u,\frac{e}{r} \right) =\tr(u \circ \lnl u)-\tr (u \circ \lnl \frac{e}{r} ).
\]
Let $u=\sum_i\lam_i q_i$. Note that $u$ and $\lnl u$ share the same  Jordan frame, so
$$\tr(u \circ \lnl u)=\sum_i\lam_i\ln (\lam_i)\le 0$$
as all  the eigenvalues of $u$ are in $[0,1]$.
Finally, 
$$u \circ \lnl \frac{e}{r}=(\sum_i\lam_i q_i)\circ \sum_i (\ln \frac{1}{r}) q_i=-(\ln r)\sum_i\lam_iq_i.$$
Thus,
$$\tr (u \circ \lnl \frac{e}{r} )=-(\ln r)(\sum_i\lam_i)=-\ln r$$
since $\tr(u) =\sum_i\lambda_i=1$. Putting everything together we get~$\HH_\Phi(u, \frac{e}{r}) \leq \ln r$.
\end{proof}

\section{An alternative  regret proof}
\label{appendix:_altreg}

\begin{theorem}\label{thm:alter}

 Let $(\cj,\circ)$ be an EJA of rank $r$ and $\cc$ be its cone of squares. For any $\eta>0$ and any sequence of loss vectors     
$ -e \preceq_\cc m_t\preceq_\cc e ,$ the iterates $\{p_t\}_t$ of \ref{SCMWU} with stepsize $\eta \leq 1$ satisfy
\begin{equation*}
	\sum_{t=1}^T \innerprod{m_t}{p_t}
	\leq
	\sum_{t=1}^T \innerprod{m_t}{u}
	+
	\eta \sum_{t=1}^T \innerprod{m_t^2}{p_t}
	+
	\frac{\ln r}{\eta}
\end{equation*}
for all $u \in \cc$ satisfying $\tr(u) = 1$.
\end{theorem}
Before we provide a proof we make some remarks on the above theorem. Since $-e \preceq_{\cc} m_t \preceq_{\cc} e$ we have that $0 \preceq_{\cc} m_t^2 \preceq_{\cc}  e$, and so $\innerprod{ m_t^2 }{ p_t } \leq \innerprod{ e }{ p_t } = 1$. Combining this with Theorem \ref{thm:alter} we get
\[
		  \sum_{t=1}^{T} \la m_t,p_t-u\ra 
            \leq
		\eta T
		+ \frac{1}{\eta} \ln r,
	\]
and using the notion of eigenvalues in an EJA and the corresponding variational characterizations of $\lambda_{\max}  /\lambda_{\min}$ (e.g. see Theorem \ref{thm:eig_ineq}) the regret bound becomes:
\[
		\sum_{t=1}^T \innerprod{ m_t }{ p_t }
		-
		\lambda_{\min}(\sum_{t=1}^T m_t) 
            \leq
		\eta T
		+ \frac{1}{\eta} \ln r.
	\]
The stepsize $\eta$ that minimizes this bound is $\sqrt{\ln r / T}$, which gives regret $2 \sqrt{T \ln r}$.

\begin{proof}
	Let $\{w_t\}_{t=1}^T$ be the unnormalized \ref{SCMWU} updates $w_t = \expl(-\eta \sum_{i=1}^{t-1} m_{i})$, initialized at $w_{1} = e$. The \ref{SCMWU} updates $p_t$ and the unnormalized \ref{SCMWU} updates $w_t$ are related by the expression $p_t = \frac{w_t}{\tr(w_t)}$. We can bound $\tr(w_{t+1})$ in terms of $\tr(w_t)$ via the following sequence of steps, which we justify below:
	\begin{equation*}
	\begin{split}
		\tr(w_{t+1})
		&=
		\tr (
			\expl (
				-\eta \sum_{i = 1}^t m_{i}
			)
		)
		\\
		&\overset{S1}{\le}
		\tr (
			\expl (
				-\eta \sum_{i = 1}^{t-1} m_{i}
			)
			\circ
			\expl(
				-\eta m_t
			)
		)
		\\
		&=
		\innerprod{w_t}{\expl( -\eta m_t )}
		\\
		&\overset{S2}{\le}
		\innerprod{
			w_t
		}
		{
			e - \eta m_t + \eta^2 m_t^2
		 }
		 \\
		 &=
		 \tr(w_t)
		 - \eta \innerprod{m_t}{w_t}
		 + \eta^2 \innerprod{m_t^2}{w_t}
		 \\
		 &=
		 \tr(w_t)
		 \left[
		 	1 
			- \eta \innerprod{m_t}{p_t} 
			+ \eta^2 \innerprod{m_t^2}{p_t}
		 \right]
		 \\
		 &\overset{S3}{\le}
		 \tr(w_t)
		 \exp(
			- \eta \innerprod{m_t}{p_t} 
			+ \eta^2 \innerprod{m_t^2}{p_t}
		 ).
	\end{split}
	\end{equation*}
    {\bf S1:} This is due to the generalized Golden-Thompson inequality \cite{ART:TWK21}, which states that 
    \[
        {\rm tr}(\expl(x+y))\le {\rm tr}(\expl(x)\circ \expl(y)) \quad \forall x,y\in\cj.
    \]
    
{\bf S2:} This is due to the self-duality of the symmetric cone and the fact that $\exp(-s) \leq 1 - s + s^2 \; \forall s \geq -1$ (Lemma \ref{bound:_exp_real_quad}), so by Lemma \ref{lem:cone_ineq} $\expl(-x) \preceq_{\cc} e - x + x^2 \; \forall x \succeq_{\cc} -e$.

	{\bf S3:} This is due to the fact that $\exp(s) \geq 1 + s \; \forall s \in \mathbb{R}$.

    Then, by induction starting from the base case $\tr(w_{1}) = r$ we have that
	\begin{equation}
        \label{bound:_pot_upper}
		\tr( w_{T+1} ) \leq r
		\exp(-\eta \sum_{t=1}^T \innerprod{m_t}{p_t} +\eta^2 \sum_{t=1}^T \innerprod{m_t^2}{p_t}).
	\end{equation}
	
	On the other hand,
	\begin{equation}
        \label{bound:_pot_lower}
	   \tr( w_{T+1})  = \tr( \expl( -\eta \sum_{t=1}^T m_t))  \geq 
	\exp( -\eta \lambda_{\min}( \sum_{t=1}^T m_t))
	\end{equation}
	since $$-\eta \lambda_{\min} ( \sum_{t=1}^T m_t ) = -\lambda_{\min} ( \eta \sum_{t=1}^T m_t ) = \lambda_{\max} (- \eta \sum_{t=1}^T m_t )$$ and $\forall x \in \cj,$
 \begin{equation*}
	\tr(\expl(x))
        =
        \sum_{k=1}^r \lambda_k( \expl(x))
	=
	\sum_{k=1}^r \exp(\lambda_k(x)) \geq
	\exp(\lambda_{\max}(x)).
 \end{equation*}
	
    Combining the upper and lower bounds \eqref{bound:_pot_upper} and \eqref{bound:_pot_lower} for $\tr(w_{T+1})$ gives us
\[
		\exp(
			-\eta \lambda_{\min} ( \sum_{t=1}^T m_t )
		) 
		\leq
		r
		\exp (
			-\eta \sum_{t=1}^T \innerprod{m_t}{p_t} + \eta^2 \sum_{t=1}^T \innerprod{m_t^2}{p_t}
		)
\]
	from which it follows by taking logarithms that
	\[
		-\eta \lambda_{\min} ( \sum_{t=1}^T m_t ) \leq \ln r
		-\eta \sum_{t=1}^T \innerprod{m_t}{p_t} + \eta^2 \sum_{t=1}^T \innerprod{m_t^2}{p_t} .
	\]
	Finally, rearranging and dividing by $\eta$ gives us the regret bound
	\[
		\sum_{t=1}^T \innerprod{m_t}{p_t}
		\leq
		\lambda_{\min} ( \sum_{t=1}^T m_t )
		+ \eta \sum_{t=1}^T \innerprod{m_t^2}{p_t}
		+ \frac{\ln r}{\eta},
	\]
    which is equivalent by Theorem \ref{thm:eig_ineq} to the desired regret bound
	\[
		\sum_{t=1}^T \la m_t, p_t\ra
		 \leq \min_{u \in \cc,\tr(u) = 1} \sum_{t=1}^T \la m_t, u\ra
		+ \eta \sum_{t=1}^T \innerprod{m_t^2}{p_t}
		+ \frac{\ln r}{\eta} .
	\]
\end{proof}

\end{document}